\documentclass[10pt,a4paper]{amsart}
\usepackage{verbatim}

\usepackage[toc]{appendix}
\usepackage[T1]{fontenc}

\usepackage{graphicx}
\usepackage{enumerate}
\usepackage{amsmath,amsfonts,amssymb}
\usepackage{color}
\def\loc{\operatorname{loc}}
\usepackage{cite}
\usepackage{ latexsym }
\definecolor{citation}{rgb}{0.11,0.67,0.84}
\definecolor{formula}{rgb}{0.1,0.2,0.6}
\definecolor{url}{rgb}{0.11,0.67,0.84}
\usepackage{pgf,tikz}
\usepackage{mathrsfs}
\usepackage{fancyhdr}
\usepackage{dutchcal}

\newcommand{\medint}{-\kern -,375cm\int}

\newcommand{\medintinrigo}{-\kern -,315cm\int}
\makeatletter
\newcommand{\linethrough}{\mathpalette\@thickbar}
\newcommand{\@thickbar}[2]{{#1\mkern0mu\vbox{
    \sbox\z@{$#1#2\mkern-0.5mu$}%
    \dimen@=\dimexpr\ht\tw@-\ht\z@+2\p@\relax 
    \hrule\@height0.5\p@ 
    \vskip\dimen@
    \box\z@}}
}
\makeatother

\newcommand{\mathstrike}[1]{\ensuremath{\linethrough{#1}}}

\newcommand{\nra}[1]{\mathstrike{\lVert} #1 \rVert}

\newtheorem{theorem}{Theorem}[section]
\newtheorem{lemma}[theorem]{Lemma}
\newtheorem{proposition}[theorem]{Proposition}

\newtheorem{corollary}[theorem]{Corollary}
\newtheorem{definition}[theorem]{Definition}
\newtheorem{remark}[theorem]{Remark}
\numberwithin{equation}{section}

\usepackage{dsfont}

\newcommand{\reqnomode}{\tagsleft@false}

\usepackage{hyperref}

\def\dx{\,{\rm d}x}

\def\dtt{\,{\rm d}t}

\def\ds{\,{\rm d}s}
\def\dt{\,{\rm d}t}

\def \d{\,{\rm d}}

\def\dist{\,{\rm dist}}

\def\supp{\,{\rm supp}}

\allowdisplaybreaks
\makeatletter
\DeclareRobustCommand*{\bfseries}{%
  \not@math@alphabet\bfseries\mathbf
  \fontseries\bfdefault\selectfont
  \boldmath
}

\makeatother

\newlength{\defbaselineskip}
\newcommand{\setlinespacing}[1]
           {\setlength{\baselineskip}{#1 \defbaselineskip}}
\newcommand{\mint}{\mathop{\int\hskip -1,05em -\, \!\!\!}\nolimits}

\newcommand{\N}{\mathbb{N}}
\newcommand{\R}{\mathbb{R}}
\newcommand{\Rn}{\mathbb{R}^{n}}
\newcommand{\RN}{\mathbb{R}^{N}}
\newcommand{\M}{\mathbb{R}^{N \times n}}

\newcommand{\ti}[1]{\tilde{#1}}
\newcommand{\mf}[1]{\textnormal{\texttt{#1}}}

\newcommand{\CC}{\operatorname{C}}
\newcommand{\WW}{\operatorname{W}}
\newcommand{\LL}{\mathrm{L}}

\newcommand{\dd}{\mathrm{d}}

\newcommand{\rrr}{\textnormal{\texttt{r}}}

\newcommand{\F}{\mathscr{F}}

\newcommand{\X}{\mathbb{R}^n}
\newcommand{\Y}{\mathbb{R}^N}

\newcommand{\rr}{\varrho}
\newcommand{\snr}[1]{\lvert #1\rvert}
\newcommand{\nr}[1]{\lVert #1 \rVert}

\newcommand{\tx}[1]{\textnormal{\texttt{#1}}}

\def\loc{\operatorname{loc}}

\def\eqn#1$$#2$${\begin{equation}\label#1#2\end{equation}}

\newcommand{\p}{\partial}

\newcommand{\e}{\varepsilon}
\def\supp{\,{\rm supp }}

\newcommand{\tn}[1]{\textnormal{#1}}

\def\XXint#1#2#3{{\setbox0=\hbox{$#1{#2#3}{\int}$}
     \vcenter{\hbox{$#2#3$}}\kern-.5\wd0}}

\title[Quantified Legendreness and the regularity of minima]{Quantified Legendreness and the regularity of minima}

\author[De Filippis]{Cristiana De Filippis}  \address{Cristiana De Filippis\\Dipartimento SMFI, Universit\'a di Parma\\
  Parco Area delle Scienze 53/A, 43124 Parma, Italy} \email{\url{cristiana.defilippis@unipr.it}}
\author[Koch]{Lukas Koch}  \address{Lukas Koch\\MPI for Mathematics in the Sciences, Inselstrasse 22, 04177 Leipzig, Germany} \email{\url{lkoch@mis.mpg.de}}
\author[Kristensen]{Jan Kristensen}
\address{Jan Kristensen\\
Mathematical Institute, University of Oxford\\
Andrew Wiles Building, Radcliffe Observatory Quarter, Woodstock Road, Oxford, OX2 6GG, Oxford, United Kingdom}
\email{\url{jan.kristensen@maths.ox.ac.uk}}

\begin{document}

\subjclass[2020]{49N60, 35J60 \vspace{1mm}} 

\keywords{Regularity, $(p,q)$-growth, Nonuniform ellipticity\vspace{1mm}}

\thanks{{\it Acknowledgements.}\ C. De Filippis is supported by the University of Parma via the project "Local vs Nonlocal: mixed type operators and nonuniform ellipticity",
  CUP\_D91B21005370003, and by the INdAM GNAMPA project "Problemi non locali: teoria cinetica e non uniforme ellitticit\`a", CUP\_E53C22001930001. We thank the anonymous referee for her/his helpful comments that eventually lead to an improved presentation of our results.
\vspace{1mm}}

\begin{abstract}
We introduce a new quantification of nonuniform ellipticity in variational problems via convex duality, and prove higher
differentiability and $2d$-smoothness results for vector valued minimizers of possibly degenerate functionals.
Our framework covers convex, anisotropic polynomials as prototypical model examples - in particular, we improve in an essentially
optimal fashion Marcellini's original results \cite{ma1}.
 \end{abstract}

\maketitle
\vspace{-9mm}
%

\setlinespacing{1.00}
\section{Introduction}\label{intro}
Motivated by the regularity theory for elliptic systems driven by differential operators that are homogeneous polynomials in the derivative
symbols obtained by Douglis \& Nirenberg \cite{dn}, Ladyzhenskaya \& Ural'tseva \cite{LU}, and Morrey \cite{m1}, we introduce a new class of
autonomous variational integrals of the type
\eqn{fff}
$$
\WW^{1,p}_{\loc}(\Omega,\mathbb{R}^{N})\ni w\mapsto \F(w;\Omega):=\int_{\Omega}F(\nabla w)\dd x,
$$
governed by a nonuniformly elliptic integrand $F\in \CC^{2}(\R^{N\times n})$ in the sense that the related ellipticity ratio
$$
\mathcal{R}_{F}(z):=\frac{\mbox{highest eigenvalue of }F^{\prime\prime}(z)}{\mbox{lowest eigenvalue of }F^{\prime\prime}(z)}
$$
might blow up on large values of the derivative variable. We obtain higher differentiability and low-dimensional smoothness results for local minimizers,
whose standard definition reads as follows.
\begin{definition}
  A map $u\in \WW^{1,p}_{\loc}(\Omega,\R^N)$ is a local minimizer of functional $\F$ if for any open subset $\Omega'\Subset \Omega$ it satisfies $\F(u;\Omega')<\infty$
  and $\F(u;\Omega')\le \F(w;\Omega')$ for all $w\in u+\WW^{1,p}_{0}(\Omega',\R^N)$.
\end{definition}
More precisely, we focus on a large class of convex integrands satisfying so-called $(p,q)$-growth conditions, according to Marcellini's foundational works
\cite{M0,ma1,Marcellini1991}:
\begin{equation}\label{coercivprim_int}
|z|^{p} \lesssim F(z) \lesssim |z|^{q} + 1,
\end{equation}
for all $z\in \mathbb{R}^{N\times n}$, and some exponents $2\le p\le q<\infty$, including anisotropic convex polynomials as model example.
This family of integrands is characterized by a potentially 
wild behavior of the ellipticity ratio at infinity, 
and the rate of
blow up of $\mathcal{R}_{F}$ as $\snr{z}\to \infty$ is measured in terms of the stress tensor, i.e.:
\eqn{elr}
$$
\mathcal{R}_{F}(z)\lesssim1+\snr{F^{\prime}(z)}^{\frac{q-p}{q-1}},\qquad \quad \quad 2\le p\le q.
$$
We shall refer to functionals as \eqref{fff} defined upon integrands satisfying \eqref{coercivprim_int}-\eqref{elr} as \emph{Legendre $(p,q)$-nonuniformly elliptic integrals},
terminology justified 
since \eqref{elr} comes as the quantification of the interaction between $F$ and its Fenchel conjugate $F^{\ast}$, the strong convexity of $F$, and the
$(p,q)$-growth conditions in \eqref{coercivprim_int}, see Section \ref{cd} for more details, and \cite{EkeTem,Rockafellar} for the abstract convex analytic setting.
We also stress that by convexity, in our local setting, the pointwise definition of functional $\F$ in \eqref{fff} coincides with its Lebesgue-Serrin-Marcellini extension,
\cite{kqc,M0,ma5}. The class of nonuniformly elliptic problems we propose falls into the (slightly larger) realm of functionals with $(p,q)$-growth, first studied by
Marcellini \cite{M0,ma1,Marcellini1991} in connection to some delicate issues of compressible elasticity including the phenomenon of cavitation \cite{ball,M0,ma5}.
Roughly speaking, variational integrals with $(p,q)$-growth are driven by possibly anisotropic integrands satisfying the unbalanced growth condition \eqref{coercivprim_int},
whose ellipticity ratio can only be controlled via
\eqn{elrpq}
$$
\mathcal{R}_{F}(z)\lesssim1+\snr{z}^{q-p}.
$$
By convexity, \eqref{coercivprim_int} and \eqref{elr}, it follows that Legendre $(p,q)$-nonuniform ellipticity \eqref{elr} implies the usual $(p,q)$-nonuniform
ellipticity \eqref{elrpq}. The crucial aspect of these problems is the subtle, quantitative relation existing between the rate of blow up of the ellipticity ratio
and the regularity of minima: in fact, Marcellini \cite{Marcellini1991,ma2} and Giaquinta \cite{g} exhibited examples of $(p,q)$-nonuniformly elliptic functionals
with unbounded minimizers provided that the exponents $(p,q)$ violate a closeness condition of the type 
\eqn{clo}
$$
q<p+\texttt{o}(n),
$$
where $\texttt{o}(n)\searrow 0$ as $n\to \infty$. It is then natural to relate the regularity of minima to the possibility of slowing down the rate of blow up
of the ellipticity ratio by choosing $p$ and $q$ 
not too far apart, cf. \eqref{clo}: in a nutshell, this was Marcellini's approach \cite{ma1} to local Lipschitz
continuity for scalar minimizers of certain anisotropic energies such as
\eqn{modelpq2}
$$
\WW^{1,2}_{\loc}(\Omega)\ni w\mapsto \int_{\Omega}\snr{\nabla w}^{2}+\sum_{i=1}^{n}\snr{\nabla_{i}w}^{q}\dd x,
$$
where 
\eqn{pq22}
$$
2\le  q<\frac{2n}{n-2} \ \ \mbox{if} \ \ n\ge 3,\qquad \quad 2\le q<\infty \ \ \mbox{if} \ \ n=2,
$$
or more generally \cite{Marcellini1991,m93},
\eqn{modelpq.2}
$$
\WW^{1,p}_{\loc}(\Omega)\ni w\mapsto \int_{\Omega}\snr{\nabla w}^{p}+\sum_{i=1}^{n}\snr{\nabla_{i}w}^{q_{i}}\dx
$$
for exponents 
\begin{equation}\label{pq23}
2\le p\leq q_1\le \cdots\le q_n<p+\frac{2p} n.
\end{equation}
After Marcellini's initial success, $(p,q)$-nonuniformly elliptic functionals have been the object of intensive investigation: considerable 
attention has been focused on full regularity \cite{bdp,bm,BS1,bs,bfbf,Bonfanti2012,bbbb1,bb,ckp1,ckp,cisc,demi1,dms,demi3,elm2,elm1,ELM,Hirsch,sch24}, boundary regularity
\cite{bdms,cm1,Cianchi2019,dp,ik,ko,ko1}, nonlinear potential theoretic results \cite{ba2,bm,BS1,cm1,dms,deqc,demi1,dp,ds} and partial regularity
\cite{deqc,ds,fra,gme,gkpq,li,Schaeffner2020,Schmidt2008,Schmidt2008a,Schmidt2009}, see also \cite{masu2} for a reasonable survey - recently, $(p,q)$-nonuniformly elliptic
regularity theory found interesting applications to nonlinear homogenization \cite{cg,rs}. Note in particular that Marcellini's variational approach, originally designed
to handle polynomial rates of nonuniformity, turns out to fit also more general, possibly nonautonomous problems whose ellipticity ratio blows up faster than powers:
see \cite{mar96} for the very first results on functionals at fast exponential growth, \cite{bm,demi1} for more recent progress, and \cite{eva03} for applications
to weak KAM theory. Back to $(p,q)$-nonuniform ellipticity, it is worth mentioning that the optimal value of the threshold to
impose on the ratio $q/p$ guaranteeing gradient regularity is known in the autonomous setting only in the scalar case\cite{Marcellini1991,Hirsch}. It covers
boundedness of minima, leaving open the delicate matter of gradient regularity. In this respect, at first Marcellini \cite{ma1} determined two bounds on
the size of $q/p$ for Lipchitz continuity of minima: first, \eqref{pq22}, formulated for the model \eqref{modelpq2}, that in its general form reads as
\eqn{ex:pq}
$$
2\le p\le q<\frac{np}{n-2}\quad \ \mbox{if} \ \ n\ge 3\qquad\quad \ \ \mbox{and}\qquad\quad \ \ 2\le p\le q \quad \ \mbox{if} \ \ n=2,
$$
which was found in \cite{Marcellini1991,m93} for minima belonging to higher energy classes (e.g. $\WW^{1,q}$-minimizers), and second, the more restrictive \eqref{pq23},
that allows deriving gradient boundedness for $\WW^{1,p}$-minimizers. Actually, this second setting is the natural framework to investigate, as given the growth/coercivity
prescribed in \eqref{coercivprim_int} it is only possible to prove the existence of minima in $\WW^{1,p}$. The constraint given by \eqref{pq23} is far from being optimal:
under natural growth conditions and in the vectorial case, Carozza \& Kristensen \& Passarelli di Napoli \cite{ckp} obtained gradient higher differentiability and
as a consequence that minima belong to $\WW^{1,q}_{\rm loc}$, provided that
$$
q<p+\frac{p}{n-1},
$$
by means of convex duality methods. The earlier results did not use duality theory and required stronger natural growth conditions, see in particular \cite{elm1}.
When considering Lipschitz continuity of $\WW^{1,p}$-minima in the scalar case, a first, substantial improvement is due to Bella \& Sch\"affner \cite{BS1,bs} in the
genuine controlled growth $(p,q)$-setting, updating \eqref{pq23} to
\eqn{bsbound}
$$
q<p+\frac{2p}{n-1},
$$
by using a refinement of De Giorgi's iteration technique via optimization on radial cut-off functions, leading to the application of Sobolev inequality on spheres
rather than on balls. Under Legendre $(p,q)$-growth conditions, we are able to further relax these bounds, both with regard to obtaining higher differentiability
results and to proving Lipschitz regularity in the scalar case. More precisely, for $\WW^{1,p}$-minimizers of Legendre $(p,q)$-nonuniformly elliptic integrals we allow
a faster blow-up rate of the ellipticity ratio, quantifiable in term of exponents $(p,q)$ as
\eqn{pq}
$$
2\le p\le q<\frac{p(n-1)}{n-3} \quad \mbox{if}\ \ n\ge 4,\qquad \qquad\quad  2\le p\le q<\infty\quad \mbox{if} \ \ n\in \{2,3\}.
$$
We remark that this constraint had previously been obtained by Bella \& Sch\"affner \cite{bs} provided the minimizer is a priori assumed to be of class $\WW^{1,q}$.
Our approach is based on the duality between the gradient of minima and the related stress tensors, a technique that in this context finds its roots in \cite{ckp}.
This is at the heart of the main result of this paper, which is gradient higher differentiability for minimizers of variational integrals satisfying Legendre $(p,q)$-growth.
\begin{theorem}\label{hd} Under assumptions \eqref{pq}, and \eqref{assf}-\eqref{assf.1}, let $u\in \WW^{1,p}_{\loc}(\Omega,\RN )$ be a local minimizer of functional $\F$. Then,
\begin{flalign}\label{hdes.4}
  V_{\mu,p}(\nabla u),\ V_{1,q'}(F^\prime(\nabla u))\in \WW^{1,2}_{\loc}(\Omega,\mathbb{R}^{N\times n} )\qquad \mbox{and}\qquad \sum_{s=1}^{n}\langle \partial_{s}F^{\prime}(\nabla u),
  \partial_{s}\nabla u\rangle\in \LL^{1}_{\loc}(\Omega),
\end{flalign}
the second inclusion in $\eqref{hdes.4}$ holding in the nondegenerate case $\mu>0$ in $\eqref{assf}_{3}$. In particular, if $B\Subset \Omega$ is any ball with radius less than one,
\eqn{hdes}
$$
\mint_{B/2}\snr{\nabla V_{\mu,p}(\nabla u)}^{2}+\snr{\nabla V_{1,q'}(F^{\prime}(\nabla u))}^{2}\dx\le c\left(\mint_{B}F(\nabla u)+1\dx\right)^{\tx{b}}
$$
holds true with $c\equiv c(n,N,L,L_{\mu},p,q)$, and $\tx{b}\equiv \tx{b}(n,p,q)$.
\end{theorem}
We immediately refer to Section \ref{sa} for the precise description of our assumptions and of the various quantities mentioned above.
The main building block in the proof of Theorem \ref{hd} the possibility of enlarging the ellipticity of $F$ by means of convex duality arguments in such a way that it quantitatively
controls both gradient of minima and stress tensor, granted by our Legendre $(p,q)$-growth, thus yielding essentially optimal results, \cite{ma2}.
As a consequence of the boost of integrability earned via Sobolev embedding from Theorem \ref{hd}, we also derive a lower order regularity result
in the spirit of Campanato \cite[Theorem 1.IV]{c2}.
\begin{corollary}\label{cor}
Under assumptions \eqref{pq}, and \eqref{assf}-\eqref{assf.1}, let $u\in \WW^{1,p}_{\loc}(\Omega,\RN )$ be a local minimizer of functional $\F$. The following holds:
\begin{itemize}
\item[(\emph{i}.)] if $p>n-2$ and $n\ge 3$, then $u\in \CC^{0,1-\frac{n-2}{p}}_{\loc}(\Omega, \RN )$;
\item[(\emph{ii}.)] if in addition 
\eqn{cor.1}
$$
2\le p\le q\le \frac{np}{n-2}\quad \mbox{if} \ \ n\ge 3\qquad \mbox{and}\qquad 2\le p\le q<\infty\quad \mbox{if} \ \ n=2,
$$
then $\nabla u\in \LL^{q}_{\loc}(\Omega,\mathbb{R}^{N\times n})$;
\item[(\emph{iii}.)] if $n=2$ then $u\in \CC^{0,\beta_{0}}_{\loc}(\Omega,\RN )\cap \WW^{1,m}_{\loc}(\Omega,\RN )$ and
  $F'(\nabla u)\in \LL^{m}_{\loc}(\Omega,\M )$ for all $\beta_{0}\in (0,1)$, $m\in [1,\infty)$.
\end{itemize}
\end{corollary}
The higher differentiability granted by Theorem \ref{hd} allows deriving full regularity in two ambient space dimensions.
\begin{theorem}\label{2dreg}
In ambient space dimension $n=2$, suppose that \eqref{pq}, and \eqref{assf}-\eqref{assf.1} are in force. Then, any local minimizer
$u\in \WW^{1,p}_{\loc}(\Omega,\mathbb{R}^{N})$ of functional $\F$ has locally H\"older continuous gradient. In particular, on all balls $B\Subset \Omega$,
the $\LL^{\infty}$-$\LL^{p}$ type estimate
\eqn{ll}
$$
\nr{\nabla u}_{\LL^{\infty}(B/8)}\le c\left(\mint_{B}F(\nabla u)+1\dx\right)^{\tx{b}}
$$
holds with $c\equiv c(N,L,L_{\mu},p,q)$ and $\tx{b}\equiv \tx{b}(p,q)$. Moreover, given any two open subsets $\Omega_{2}\Subset \Omega_{1}\Subset \Omega$ 
with $\dist(\Omega_{2},\partial\Omega_{1})\approx\dist(\Omega_{1},\partial \Omega)
$, there exists
$\mf{t}\equiv \mf{t}(N,L,L_{\mu},p,q,\F(u;\Omega_{1}),\dist(\Omega_{2},\partial \Omega_{1}))>1$ such that
$V_{p}(\nabla u),V_{1,q'}(F'(\nabla u))\in \WW^{1,2\mf{t}}(\Omega_{2},\mathbb{R}^{N\times 2})$ and the reverse H\"older type inequality
\eqn{ll.1}
$$
\left(\mint_{B/8}\snr{\nabla V_{\mu,p}(\nabla u)}^{2\tx{t}}+\snr{\nabla V_{1,q'}(F'(\nabla u))}^{2\tx{t}}\dx\right)^{\frac{1}{2\tx{t}}}\le c\left(\mint_{B}F(\nabla u)+1\dx\right)^{\tx{b}},
$$
is verified for any ball $B\Subset \Omega_{2}$, with $c\equiv c(N,L,L_{\mu},p,q,\F(u;\Omega_{1}),\dist(\Omega_{2},\partial\Omega_{1}))$, $\tx{b}\equiv \tx{b}(p,q)$. Finally, 
\begin{itemize}
\item if $\mu>0$ in $\eqref{assf}$, then the gradient of minima is locally H\"older continuous up to any exponent less than one;
\item if $\mu>0$ in $\eqref{assf}$, and integrand $F$ is real analytic, then $u$ is real analytic.
\end{itemize}
\end{theorem}
Theorems \ref{hd}-\ref{2dreg} and Corollary \ref{cor} cover a large number of models, including \eqref{modelpq2}-\eqref{modelpq.2}, and more generally, convex polynomials.
Actually, restating Corollary \ref{cor} and Theorem \ref{2dreg} for vector-valued minimizers of functionals driven by strongly convex, even polynomials, we extend
to the nonuniformly elliptic framework early results of Morrey \cite{MorreyB,m1}.
\begin{theorem}\label{polyt}
Let $u\in \WW^{1,p}_{\loc}(\Omega,\mathbb{R}^{N})$ be a local minimizer of functional $\F$, where the integrand $F$ is a convex, even polynomial with nonnegative
homogeneous components, lowest homogeneity degree larger or equal than $p$, and satisfying the lower bound in $\eqref{assf}_{3}$. The following holds:
\begin{itemize}
    \item[(\emph{i}.)] in three space dimensions $n=3$, minima are locally $(p-1)/p$-H\"older continuous; 
    \item[(\emph{ii}.)] in two space dimensions $n=2$, if $\mu>0$ in \eqref{assf}$_{3}$ then minima are real analytic.
\end{itemize}
\end{theorem}
It is worth highlighting that convex, even polynomials with nonnegative homogeneous components automatically satisfy our Legendre
$(p,q)$-nonuniform ellipticity condition,\footnote{Specifically, for polynomials $p$ denotes the ellipticity exponent in $\eqref{assf}_{3}$, while $q$ indicates the degree.}
see Section \ref{polyn} - in particular Theorem \ref{polyt} holds with no restriction on the degree of the polynomial. In the scalar setting, combining Theorem \ref{hd}
with a homogenised Moser iteration argument, we obtain Lipschitz regularity of minima.
\begin{theorem}\label{scat}
Under assumptions \eqref{pq}, and \eqref{assf}-\eqref{assf.1}, let $u\in \WW^{1,p}_{\loc}(\Omega)$ be a local minimizers of functional $\F$.
Then, $u\in \WW^{1,\infty}_{\loc}(\Omega)$. In particular, whenever $B\Subset \Omega$ is a ball with radius $\rrr\in (0,1]$, the Lipschitz bound
\begin{flalign}\label{scai}
\nr{\nabla u}_{\LL^{\infty}(B/8)}\le c\left(\mint_{B}F(\nabla u)+1\dx\right)^{\tx{b}}
\end{flalign}
holds with $c\equiv c(n,L,L_{\mu},p,q)$ and $\tx{b}\equiv \tx{b}(n,p,q)$.
\end{theorem}
We remark that if $p=2$, the bound in \eqref{pq} corresponds precisely to the one violated by the counterexample in \cite[Theorem 6.1]{Marcellini1991}.
This highlights the sharpness of our results for polynomial-type integrals with quadratic growth from below. Notice further that in the vectorial case, and
when the ambient space dimension $n$ is larger than two, minimizers need not be locally Lipschitz. In fact, this is already the case when the integrand
is smooth and uniformly elliptic, as shown by \v{S}ver\'{a}k \& Yan \cite{Sverak} for dimensions $n \geq 3$, $N \geq 5$. More recently, Mooney \& Savin \cite{ms16}
gave an example of a smooth and uniformly elliptic integrand on $2 \times 3$ matrices, so dimensions $n=3$, $N=2$, such that the corresponding variational integral admits a minimizer
that is Lipschitz but not $\CC^1$. It seems likely that their approach can be adapted to give examples of non-Lipschitz minimizers also in those dimensions.
The remainder of the section is devoted to a description of some key technical points appearing throughout the paper that we believe could be of wider use. 
We end with a quick outline of the structure of the paper.
\subsection{Techniques}\label{tec} From the technical point of view, the main contribution of this paper is threefold. First, we introduce a new way to control the rate of
blow up of the ellipticity ratio for superlinear, possibly degenerate functionals, that we refer to as the Legendre $(p,q)$-nonuniform ellipticity condition \eqref{elr}.
As already mentioned, \eqref{elr} is slightly more restrictive than the genuine $(p,q)$-nonuniform ellipticity \eqref{elrpq}, nonetheless it covers
the same models, i.e. convex polynomials or anisotropic energies, under essentially optimal bounds on exponents $(p,q)$, cf. \eqref{pq} - in particular, in two
and three space dimensions the ellipticity ratio is allowed to blow up arbitrarily fast, meaning for polynomials that no upper bound on the degree is required.
Our results crucially rely on a subtle interplay between the gradient of minima and the stress tensor, that can be handled via convex duality. This is a natural strategy,
and convex duality tools have already been employed by Zhikov \& Kozlov \& Oleinik in \cite{JKO} in connection to homogenized elasticity theory, by Seregin \cite{Seregin93}
in the theory of plasticity, and by Carozza \& Kristensen \& Passarelli di Napoli \cite{ckp2}, and Koch \& Kristensen \cite{Koch2022} on the validity of the
Euler-Lagrange system. The key aspect in our approach is the duality between the gradient of minima and the related stress tensor that, via \eqref{elr} and rather
elementary convex duality arguments, can be made quantitative, so that the ellipticity ratio of the integrand $F$ is enlarged as a natural consequence of the combination
between convexity and extremality conditions - in particular, primal and dual problems can be handled simultaneously. The second main novelty consists in a
streamlining of Bella \& Sch\"affner's optimization trick \cite{bs}, that
allows us to work directly on spheres rather than on solid balls and eventually leads to the (surprisingly simple) proof of gradient higher differentiability
of $\WW^{1,p}$-minimizers under the constraint \eqref{pq}. This brings us to the account of the third main novelty of this paper, that is a new proof of full
regularity for vector-valued minimizers of possibly degenerate, nonuniformly elliptic functionals in two space dimensions, obtained by means of a monotonicity
argument in the spirit of Frehse \cite{fr}, or via a renormalized version of classical Gehring - Giaquinta \& Modica Lemma \cite{fwg,gm1}, following the useful point of view introduced by
Beck \& Mingione \cite{bm}. This is a well-known issue in the multidimensional Calculus of Variations, handled for $p$-Laplacian type problems via Gehring - Giaquinta \& Modica
Lemma applied right after differentiating the Euler-Lagrange system of $\F$, see previous contributions by Ne\v{c}as \cite{nec}, Giaquinta \& Modica \cite{gm1}, and Campanato \cite{c2}.
The corresponding $2d$-smoothness result for nondegenerate, genuinely $(p,q)$-nonuniformly elliptic integrals, is a more recent achievement of Bildhauer \& Fuchs \cite{bfbf},
whose strategy strongly relies on the existence in $\LL^{2}$ of second derivatives - a distinctive features of minima of nondegenerate integrals that dramatically fails
already for the degenerate $p$-Laplacian \cite{Ur}. For this reason, Theorem \ref{2dreg} comes by no means as an adaptation of the arguments in \cite{bfbf}, and
covers degenerate problems. In this respect, we offer two independent proofs of $2d$-smoothness. The first one follows the arguments outlined in the proof of
Theorem \ref{hd} to derive a monotonicity formula resulting in a logarithmic Morrey-type decay for the $\LL^{2}$-norm of certain nonlinear functions of the gradient,
eventually implying $\CC^{1}$-regularity of minima. The second one is based on a quantitative version of Gehring - Giaquinta \& Modica lemma \cite{fwg,gm1},
applied to the differentiated Euler Lagrange system after bounding the ellipticity ratio via a power of the $\LL^{\infty}$-norm of the gradient, so that it becomes
uniformly elliptic - constants are then carefully tracked at each stage of the proof. Both our arguments are inspired by the principle that a suitable control
on the rate of blow up of the ellipticity ratio associated to nonuniformly elliptic functionals gives access to a technical toolbox that has classically been
employed in the Lipschitz regularity theory for uniformly elliptic structures, such as De Giorgi's level sets technique \cite{bm,BS1,demi1,dms,dp},
homogenized Moser's iteration \cite{demi3} and now monotonicity formula and Gehring - Giaquinta \& Modica lemma. Notice that after minor modifications,
Theorem \ref{2dreg} embraces also genuine $(p,q)$-nonuniformly elliptic integrals, cf. Remark \ref{rempq} below.
\subsubsection*{Outline of the paper} 
In Section \ref{sec:pre} we describe our notation and collect a number of auxiliary results that will be helpful at various stages of the paper.
In Section \ref{cd} we derive several important consequences of our Legendre $(p,q)$-growth via convex duality arguments, and provide examples of integrands
satisfying it, notably in the form of even, convex polynomials. In Section \ref{hdd} we establish higher differentiability and higher integrability results
for a suitable nonlinear function of the gradient of minima, and of the related stress tensor. Section \ref{2d} contains the proof of full regularity in
the two-dimensional case, while finally Section \ref{scalar} contains the proof of Lipschitz regularity in the scalar setting.

\section{Preliminaries}\label{sec:pre}
In this section we display our notation, describe the main assumptions governing the functional $\F$, and collect some basic results that will be helpful throughout the paper. 
\subsection{Notation}\label{notation} In this paper, $\Omega\subset \X$, $n\ge 2$, will always be an open domain.
We denote by $c$ a general constant larger than one, possibly depending on various parameters related to the problem under investigation.
We will still denote by $c$ distinct occurrences of the constant $c$ from line to line. Specific instances will be marked with symbols $c_*,  \tilde c$ or the like.
Significant dependencies on certain parameters will be outlined by putting them in parentheses, i.e. $c\equiv c(n,p)$ means that $c$ depends on $n$ and $p$.
Sometimes we shall employ symbols "$\lesssim$", "$\approx$" or "$\gtrsim$" to indicate that an inequality holds up to constants depending on basic parameters
governing our problem. By $ B_r(x_0):= \{x \in \mathbb{R}^n  :   |x-x_0|< r\}$ we indicate the open ball with center in $x_0$ and radius $r>0$, while
$Q_{r}(x_{0}):=\left\{x\in \mathbb{R}^{n}\colon \max_{i\in\{1,\cdots,n\}}\snr{x_{i}-x_{0;i}}<r\right\}$ denotes the open cube with half-side length equal to $r$
- when clear from the context,
we shall avoid specifying radius or center, i.e., $B \equiv B_r \equiv B_r(x_0)$, $Q\equiv Q_{r}\equiv Q_{r}(x_{0})$; this happens in particular with concentric balls or concentric cubes.
In particular, with $B$ being a given ball with
radius $\textnormal{\texttt{r}}$ and $\gamma$ being a positive number, we denote by $\gamma B$ the concentric ball with radius $\gamma \textnormal{\texttt{r}}$
and by $B/\gamma := (1/\gamma)B$. For a number $t\in [1,\infty]$, its conjugate exponent $t'\in [1,\infty]$ is defined as $t':=(1-1/t)^{-1}$ if $t>1$, $t':=\infty$ if
$t=1$, and $t'=1$ if $t=\infty$, while, given $\tx{n}\in \N$, its $\tx{n}$-dimensional Sobolev exponent, and its $\tx{n}$-dimensional "lower" Sobolev
exponent are given by $t^{\ast}_{\tx{n}}:=\tx{n}t/(\tx{n}-t)$ if $1\le t<\tx{n}$ and $t^{\ast}_{\tx{n}}:=\mbox{any number larger than }t$ if
$t\ge \tx{n}$, and $t_{*;\tx{n}}:=\max\{1,\tx{n}t/(\tx{n}+t)\}$ respectively - here, $\tx{n}$ will be either chosen as $\tx{n}=n$ or $\tx{n}=n-1$.
Moreover, if $A \subset \mathbb{R}^{n}$ is a measurable set with bounded positive Lebesgue measure $\snr{A}\in (0,\infty)$, and $g \colon  A \to \mathbb{R}^{d}$,
$d\geq 1$, is a measurable map, we set $(g)_{A}\equiv \medintinrigo_{A}g(x)\dd x:= \snr{A}^{-1}\int_{A}  g(x) \dd x$ to indicate its integral average, while if
$g\in \LL^{\gamma}(B,\mathbb{R}^{d})$ for some $\gamma>1$, we shorten its averaged norm as
$\nra{g}_{\LL^{\gamma}(B)}:=\left(\medintinrigo_{B}\snr{g}^{\gamma}\dx\right)^{1/\gamma}$. Furthermore, with $z\in \mathbb{R}^{k}$, and $\mu\ge 0$,
we set $\ell_{\mu}(z):=(\mu^{2}+\snr{z}^{2})^{1/2}$. Finally, if $t$ is any parameter, we indicate by $\texttt{o}(t)$ a
quantity that is infinitesimal as $t\to 0$ or $t\to \infty$.
\subsection{Structural assumptions}\label{sa}
Throughout the paper, $F\colon \mathbb{R}^{N\times n}\to \mathbb{R}$, $n\ge 2$, $N\ge 1$, is an autonomous integrand satisfying
\eqn{assf}
$$
\left\{
\begin{array}{c}
\displaystyle
\ F \mbox{ is } \CC^{2}(\mathbb{R}^{N\times n}),\\ [8pt]\displaystyle
\ L^{-1}\ell_{\mu}(z)^{p}\le F(z)\le L\ell_{\mu}(z)^{p}+L\ell_{\mu}(z)^{q},\\ [8pt]\displaystyle
\ L^{-1}\ell_{\mu}(z)^{p-2}|\xi|^{2}\le \langle F''(z)\xi,\xi\rangle,\quad \quad \snr{F^{\prime\prime}(z)}\le L\left(1+\snr{F'(z)}^{\frac{q-2}{q-1}}\right),
\end{array}
\right.
$$
for all $z, \xi\in \M$, some absolute constants $L> 1$, $\mu\in [0,1]$, and exponents $(p,q)$ satisfying \eqref{pq}.
If $\mu=0$ in \eqref{assf}, we need to prescribe a (rather natural, see Section \ref{cd} below) limitation on the rate of blow up of the ellipticity ratio:
\begin{flalign}\label{assf.1}
\frac{\snr{F''(z)}}{\snr{z}^{p-2}}\le L_{0}+L_{0}\snr{F'(z)}^{\frac{q-p}{q-1}}\qquad \mbox{for all} \ \ z\in \mathbb{R}^{N\times n},
\end{flalign}
where $L_{0}>1$ is an absolute constant. Moreover, \eqref{assf}$_{3}$ assures the strict convexity of $F$, and, combined with $\eqref{assf}_{2}$ guarantees also that
\eqn{f'}
$$
|F^{\prime}(z)| \leq c\ell_{\mu}(z)^{p-1}+c\ell_{\mu}(z)^{q-1} \ \Longrightarrow \ \ell_{\mu}(z)\ge \begin{cases}
\ c\snr{F'(z)}^{\frac{1}{p-1}}\quad &\mbox{if} \ \ 0\le \ell_{\mu}(z)\le 1\\
\ c\snr{F'(z)}^{\frac{1}{q-1}}\quad &\mbox{if} \ \ \ell_{\mu}(z)>1,
\end{cases}
$$
for $c\equiv c(L,p,q)$, see \cite[Lemma 2.1]{ma1}. Let us point out that in the nondegenerate case $\mu>0$ in \eqref{assf},
the constraint in \eqref{assf.1} comes as a consequence of $\eqref{assf}_{3}$ and \eqref{f'} up to constants depending on $(L,p,q,\mu)$, so overall,
\eqn{assf.2}
$$
\frac{\snr{F''(z)}}{\ell_{\mu}(z)^{p-2}}\le L_{\mu}+L_{\mu}\snr{F'(z)}^{\frac{q-p}{q-1}},
$$
with $L_{\mu}\equiv L_{0}$ if $\mu=0$ and $L_{\mu}=\mu^{2-p}c(L,p,q)$ if $\mu>0$ - in other words, thanks to \eqref{assf.1} the dependency
on $\mu$ occurs only if $\mu>0$. 
\begin{remark}
\emph{Assumption \eqref{assf.1} deserves some comment in relation to $\eqref{assf}_{3}$ and to \eqref{assf.2}, where a dependency on $\mu$ appears in the constants.
The restriction in \eqref{assf.1} is needed for general degenerate integrands to control
the rate of blow up of the ellipticity ratio, i.e., $\mu=0$ in $\eqref{assf}_{3}$, while if $\mu>0$ it is not needed at all as the amount of informations contained
in $\eqref{assf}_{3}$ suffices to derive the bound in \eqref{assf.2} (with explicit dependency on $\mu$). This seems to be unavoidable, given the strong inhomogeneity
displayed in the growth of $F''$, cf. \eqref{assf}$_{3}$. On the other hand, if more homogeneous growth conditions are imposed on $F''$, like those usually appearing
in the literature on $(p,q)$-nonuniformly elliptic problems \cite{bm,BS1,demi1,demi3,dp}, prescribing that
$ \snr{F''(z)}\le L \ell_{\mu}(z)^{p-2}+L\ell_{\mu}(F'(z))^{\frac{q-2}{q-1}}$ replaces the right-hand side of $\eqref{assf}_{3}$, then \eqref{assf.1} can
be disregarded and no dependency on $\mu$ appears in \eqref{assf.2}.}
\end{remark}
\begin{remark}
    \emph{If $p=q$ our results are classical, see \cite{c2,giu} and references therein, so given that our approach is of interpolative nature, to avoid trivialities throughout the paper we shall permanently work assuming the strict inequality $p<q$.}
\end{remark}

\subsection{Functional analytical tools for degenerate problems} The vector field $V_{\mu,\gamma}\colon \M\to \M$, defined as
\eqn{vpvqn}
$$
V_{\mu,\gamma}(z):=(\mu^{2}+\snr{z}^{2})^{\frac{\gamma-2}{4}}z,\qquad \gamma\in (1,\infty), \ \ \mu\in [0,1],
$$
for all $z\in \mathbb{R}^{N\times n}$, which encodes the scaling features of the $p$-Laplace operator, is a useful tool to handle singular or degenerate problems.
A couple of helpful related equivalences are the following:
\begin{flalign}\label{Vm}
\left\{
\begin{array}{c}
\displaystyle
\ \snr{V_{\mu,\gamma}(z_{1})-V_{\mu,\gamma}(z_{2})}\approx (\mu^{2}+\snr{z_{1}}^{2}+\snr{z_{2}}^{2})^{(\gamma-2)/4}\snr{z_{1}-z_{2}} \\ [12pt]\displaystyle
\ \int_0^1 (\mu^2+|z_{2}+\lambda(z_{1}-z_{2})|^2)^{\frac{\gamma-2}{2}}\dd \lambda\approx (\mu^2+|z_{1}|^2+|z_{2}|^2)^{\frac{\gamma-2}{2}}\\ [12pt]\displaystyle
\nra{V_{\mu,\gamma}(w)-(V_{\mu,\gamma}(w))_{A}}_{\LL^{p}(A)}\approx
\nra{V_{\mu,\gamma}(w)-V_{\mu,\gamma}((w)_{A})}_{\LL^{p}(A)},
\end{array}
\right.
\end{flalign}
holding for any $1<p<\infty$, $z_{1},z_{2}\in \mathbb{R}^{N\times n}$, all measurable subsets $A\subset \mathbb{R}^{n}$ (resp. $A\subset \mathbb{R}^{n-1}$) with
positive, finite $n$-dimensional Lebesgue measure (resp. $(n-1)$-dimensional Hausdorff measure), functions $w\in \LL^{p\gamma/2}(A,\mathbb{R}^{N\times n})$, up to constants
depending only on $(n,N,\gamma,p)$, cf. \cite[Section 2]{ha}, and \cite[(2.6)]{gm}. Let us also recall Sobolev embedding theorem and Sobolev-Poincar\'e
inequalities on spheres, cf. \cite[Section 3]{bs}. 
\begin{lemma}
Let 
$\partial B_{\rr}(x_{0})$ be an $(n-1)$-dimensional sphere with $n\ge 2$, $k\in \mathbb{N}$ be a number, and $w\in \WW^{1,1}(\partial B_{\rr}(x_{0}),\mathbb{R}^{k})$ be a function.
Then
\begin{itemize}
\item if $w\in \WW^{1,2}(\partial B_{\rr}(x_{0}),\mathbb{R}^{k})$, then
\eqn{spv}
$$
\nra{w-(w)_{\partial B_{\rr}(x_{0})}}_{\LL^{2^{*}_{n-1}}(\partial B_{\rr}(x_{0}))}\le c \rr\nra{\nabla w}_{\LL^{2}(\partial B_{\rr}(x_{0}))},
$$
with $c\equiv c(n,k)$;
\item if $w\in \WW^{1,p}(\partial B_{\rr}(x_{0}),\mathbb{R}^{k})$ for some $1<p<\infty$, then
\eqn{spv.2}
$$
\nra{w}_{\LL^{p}(\partial B_{\rr}(x_{0}))}\le c\rr\nra{\nabla w}_{\LL^{p_{*;n-1}}(\partial B_{\rr}(x_{0}))}+c\nra{w}_{\LL^{p_{*;n-1}}(\partial B_{\rr}(x_{0}))},
$$
for $c\equiv c(n,p,k)$.
\end{itemize}
\end{lemma}
We will also need an "unbalanced" version of Poincar\'e inequality, \cite[Section 3]{bs}.
\begin{lemma}
Let $B_{\rr}(x_{0})\subset \mathbb{R}^{n}$, $n\ge 2$, be a ball, $\sigma>0$ a number, and $w\in \WW^{1,2}(B_{\rr}(x_{0}))$ be any function. Then
\eqn{spv.3}
$$
\nra{w}_{\LL^{2}(B_{\rr}(x_{0}))}\le c\rr\nra{\nabla w}_{\LL^{2}(B_{\rr}(x_{0}))}+c\nra{w^{\sigma}}_{\LL^{1}(B_{\rr}(x_{0}))}^{\frac{1}{\sigma}},
$$
with $c\equiv c(n,\sigma)$.
\end{lemma}
We further record classical Sobolev-Morrey embedding theorem with sharp bounding constant, obtained by applying \cite[Theorem 1.1]{c} componentwise.
\begin{proposition}\label{shso}
Let $w\in W^{1,p}(\mathbb{R}^{n},\mathbb{R}^{k})$ with $p>n\ge 2$, $k\ge 1$ be a function such that $\snr{\supp(w)}<\infty$. Then
$$
\nr{w}_{L^{\infty}(\mathbb{R}^{n})}\le k^{1/2}n^{-1/p}\omega_{n}^{-1/n}\left(\frac{p-1}{p-n}\right)^{1/p'}\snr{\supp(w)}^{\frac{1}{n}-\frac{1}{p}}\nr{\nabla w}_{L^{p}(\mathbb{R}^{n})}.
$$
\end{proposition}
We conclude this section with the "simple but fundamental" iteration lemma, \cite[Chapter 6]{giu}.
\begin{lemma}\label{l5}
Let $h\colon [\rr_{0},\rr_{1}]\to \mathbb{R}$ be a non-negative and bounded function, and let $\theta \in (0,1)$, $A,B,\gamma_{1},\gamma_{2}\ge 0$ be numbers.
Assume that $h(t)\le \theta h(s)+A(s-t)^{-\gamma_{1}}+B(s-t)^{-\gamma_{2}}$ holds for all $\rr_{0}\le t<s\le \rr_{1}$. Then the following inequality holds
$h(\rr_{0})\le c(\theta,\gamma_{1},\gamma_{2})[A(\rr_{1}-\rr_{0})^{-\gamma_{1}}+B(\rr_{1}-\rr_{0})^{-\gamma_{2}}].$
\end{lemma}

\section{Growth conditions and convex duality}\label{cd}
In this section we recall certain elementary notions from convex analysis \cite{EkeTem,hl,HUL,Rockafellar}, and derive important
consequences of the structural assumptions \eqref{assf}. We start by recording some basic facts on convex functions.
The Fenchel conjugate of an integrand $F \colon \M \to \R$ is defined as the extended real-valued integrand $
\M\ni \xi\mapsto F^{\ast}( \xi ) := \sup_{z \in \M} \left( \langle z , \xi\rangle - F(z) \right).$ Recall that the Fenchel conjugate $F^\ast$
is real-valued precisely when $F$ is super-linear
at infinity, so precisely when $F(z)/|z| \to \infty$ as $|z| \to \infty$. In fact there is perfect symmetry here. Recall
that the convex and lower semicontinuous envelope of $F$ coincides with the double Fenchel conjugate
$F^{\ast \ast} \equiv ( F^{\ast})^{\ast}$, and so, when $F$ is convex, it is real-valued and super-linear precisely
when its Fenchel conjugate $F^\ast$ is so. For later reference, we also recall the Fenchel-Young inequality stating that
\eqn{fy}
$$
F(z) + F^{\ast}(\xi) \geq \langle z, \xi\rangle
$$
holds for all $z$, $\xi \in \M$. It is a direct consequence of the definition of Fenchel conjugation and it holds for any proper integrand $F$.
Equality in \eqref{fy} holds precisely when $\xi \in \partial F(z)$, the subdifferential of $F$ at $z$. In particular, we emphasize that if $F$
is $\CC^1$-regular, equality holds precisely when $F$ is convex at $z$ and $\xi = F^{\prime}(z)$. It is worth highlighting that if $F$ satisfies
the $(p,q)$-growth conditions $\eqref{assf}_{2}$, these are transformed under Fenchel conjugation into
\eqn{coercivdual}
$$
\frac{1}{c}\snr{\xi}^{q^{\prime}}-c\le F^{\ast}(\xi)\le c\snr{\xi}^{p^{\prime}}+c,
$$
for some $c\equiv c(L,p,q)\ge 1$, , cf. \cite[Section 2]{ckp} for more details. Let us point out that the integrand $F$ is super-linear, strictly
convex and $\CC^1$-regular precisely when its Fenchel conjugate $F^{\ast}$ is so. In this case we also have that both derivatives are homeomorphisms of $\M$
and that $(F^{\ast})^{\prime} = ( F^{\prime})^{-1}$. Let us record one of the implications, namely that when $F$ is real-valued, super-linear, strictly convex
and $\CC^1$-regular, then so is $F^\ast$ and
\begin{equation}\label{derivatives}
  ( F^{\ast} )^{\prime}( F^{\prime}(z) ) = z \quad \mbox{and} \quad F^{\prime} ( ( F^{\ast} )^{\prime}( \xi ) ) = \xi \qquad\mbox{for all} \ \ z,\xi \in \M .
\end{equation}
We further record the duality relations that exist between strict convexity and smoothness for an integrand and its Fenchel conjugate. 
\begin{lemma}\label{strictcvx}
Let $F \colon \M \to \R \cup \{ \infty \}$ be a convex and proper integrand. Then $F$ is real-valued, super-linear and strictly convex if and only if
its Fenchel conjugate $F^\ast$ is real-valued, super-linear and $\CC^1$-regular.
\end{lemma}
The following well-known result \cite[Corollary 4.2.10]{HUL} and its consequences play an important role in this paper.
\begin{lemma}\label{strongcvx}
Let $F\in \CC^{1}(\M)$ be super-linear and strictly convex. If $F\in \CC^{2}(B_{r}(z_{0}))$ for some ball $B_{r}(z_{0})\subset \M$, and
$\mathrm{det} (F^{\prime \prime}(z)) \neq 0$ for all $z \in B_{r}(z_{0})$, then the Fenchel conjugate $F^\ast $ is super-linear, strictly convex
and $\CC^1$-regular. In addition, $F^\ast\in \CC^{2}(F^{\prime}(B_{r}(z_{0})))$ with $( F^\ast )^{\prime \prime}(F^{\prime}(z) ) = F^{\prime \prime}(z)^{-1}$ for all $z\in B_{r}(z_{0})$.
\end{lemma}
We now turn our attention to the growth/ellipticity condition \eqref{assf}$_{3}$, that is a quantitative form of Legendreness for real-valued
$\CC^2$-regular integrands. Recall indeed that Rockafellar in \cite[Chapt. V]{Rockafellar} highlighted a special class
of convex functions, called there functions of Legendre type, for their useful features in optimization theory. In the context 
considered in this paper, these are the integrands that, together with their Fenchel conjugates, are real-valued and strictly convex.
Here it is clear that the lower bound in \eqref{assf}$_3$ quantifies the strict convexity of $F$, whereas the upper bound quantifies the strict
convexity of $F^\ast$ as we specify below.

\begin{lemma}\label{thebounds}
Let $F\in \CC^{2}(\M)$ be an integrand satisfying the Legendre $(p,q)$-growth condition \eqref{assf}$_3$ for exponents $2 \leq p \leq q < \infty$.
Then $F^{\ast}\in \CC^{1}(\M)\cap \CC^{2}(\M\setminus \{F^{\prime}(0)\})$ and
\begin{equation}\label{dualbound}
\frac{1}{2L}\ell_{1}(F^{\prime}(z))^{q^{\prime}-2}| \xi |^{2} \leq \langle ( F^{\ast} )^{\prime \prime}( F^{\prime}(z) )\xi , \xi  \rangle\leq
L\ell_{\mu}(z)^{2-p}| \xi |^{2}
\end{equation}
for all $z\in\M\setminus\{0\}$, $\xi \in \M$. If $\mu>0$ in \eqref{assf}$_3$, then $F^\ast$ is $\CC^2$-regular and \eqref{dualbound} holds for all $z\in\M$, $\xi\in \M$.
\end{lemma}
\begin{proof}
The lower bound in \eqref{assf}$_3$ implies in a routine manner that $F$ is strictly convex and super-linear, and it also guarantees that the Hessian matrix
$F^{\prime \prime}(z)$ is invertible for all $z \in \M\setminus\{0\}$. We can now infer from Lemma \ref{strongcvx} that $F^\ast$ is $\CC^2$-regular away from $F'(0)$
and that $( F^{\ast})^{\prime \prime} (F^{\prime}(z )) = F^{\prime \prime}(z)^{-1}$ for all $z \in \M\setminus\{0\}$. In order to derive the double bound \eqref{dualbound}
we note that \eqref{assf}$_3$ yields that all the eigenvalues of $F^{\prime \prime}(z)$ belong to the interval
$$
\left[ L^{-1} \ell_{\mu}(z)^{p-2},L\left( 1+|F^{\prime}(z)|^{\frac{q-2}{q-1}} \right) \right],
$$
and given that the eigenvalues of $( F^{\ast} )^{\prime \prime}(F^{\prime}(z)) = F^{\prime \prime}(z)^{-1}$ are their reciprocals, they must all belong to the interval
$$
\left[ \frac{1}{2L} \ell_{1}(F^{\prime}(z))^{\frac{2-q}{q-1}}, L\ell_{\mu}(z)^{2-p} \right] .
$$
Since $(2-q)/(q-1)=q^{\prime}-2$ this concludes the proof of \eqref{dualbound}. The final claim follows easily from the same considerations, and Lemma \ref{strongcvx}.
\end{proof}
We conclude this part with a monotonicity property that follows from the Legendre $(p,q)$-growth condition \eqref{assf}$_3$.
While this type of result is probably not surprising to the experts, we are not aware of any record of it in the literature. It is instrumental for the approach
of this paper and we remark that it allows us to use the full power of convex duality without having to go through the dual variational problem \cite{JKO}.
\begin{corollary}\label{keylemma}
Let $F\in \CC^{2}(\mathbb{R}^{N\times n})$ be an integrand that satisfies the Legendre $(p,q)$-growth condition \eqref{assf}$_{3}$. Then the monotonicity inequalities
\begin{equation}\label{keymono}
\left\{
\begin{array}{c}
\displaystyle
\ \langle F^{\prime}(z_{1})-F^{\prime}(z_{2}),z_{1}-z_{2} \rangle \geq c\snr{V_{\mu ,p}(z_{1})-V_{\mu ,p}(z_{2})}^{2}
  +c\snr{ V_{1,q^\prime}(F^{\prime}(z_{1}))-V_{1,q^{\prime}}(F^{\prime}(z_{2}))}^{2}\\[8pt]\displaystyle
  \ F(z)-F(0)-\langle F'(0),z\rangle \ge c\snr{V_{\mu ,p}(z)}^{2}
  +c\snr{ V_{1,q^\prime}(F^{\prime}(z)-F^{\prime}(0))}^{2}
  \end{array}
  \right.
\end{equation}
hold for all $z,z_{1},z_{2}\in \mathbb{R}^{N\times n}$, with $c\equiv c(L,p,q)$. Moreover, given any $\mathbb{R}^{N\times n}$-valued, differentiable function $w$
defined on an open set $\Omega\subset \mathbb{R}^{n}$, it holds
\eqn{keymono.1}
$$
\sum_{s=1}^n \langle F^{\prime\prime}(w)\p_s w,\p_s w \rangle \geq c\snr{\nabla V_{\mu,p}(w)}^{2}+c\snr{\nabla V_{1,q'}(F'(w))}^{2},
$$
for all $s\in\{1,\cdots,n\}$, and some $c\equiv c(L,p,q)$.
\end{corollary}
\begin{proof}
If $\mu\in (0,1]$ in \eqref{assf}$_{3}$, inequality \eqref{keymono}$_{1}$ comes as a direct consequence of Lemma \ref{thebounds} and \eqref{Vm}$_{1,2}$,
with the stated dependency of the constant $c$. In fact, letting $\xi_{1}:=F^{\prime}(z_{1})$, $\xi_{2}:=F^{\prime}(z_{2})$, we rewrite
\begin{eqnarray*}
\langle F^{\prime}(z_{1})-F^{\prime}(z_{2}),z_{1}-z_{2}\rangle&\stackrel{\eqref{derivatives}}{=}&\langle \xi_{1}-\xi_{2},(F^{\ast})^{\prime}(\xi_{1})-(F^{\ast})^{\prime}(\xi_{2})\rangle\nonumber \\
&=&\left\langle\left(\int_{0}^{1}(F^{\ast})^{\prime\prime}(\xi_{2}+\lambda(\xi_{1}-\xi_{2}))\d\lambda\right)(\xi_{1}-\xi_{2}),\xi_{1}-\xi_{2}\right\rangle\nonumber \\
&\stackrel{\eqref{dualbound}}{\ge}&\frac{1}{2L}\left(\int_{0}^{1}\ell_{1}(\xi_{2}+\lambda(\xi_{1}-\xi_{2}))^{q^{\prime}-2}\d\lambda\right)\snr{\xi_{1}-\xi_{2}}^{2}\nonumber \\
&\stackrel{\eqref{Vm}_{1,2}}{\ge}&c\snr{V_{1,q^{\prime}}(\xi_{1})-V_{1,q^{\prime}}(\xi_{2})}^{2},
\end{eqnarray*}
for $c\equiv c(L,p,q)$ that, together with the lower bound in $\eqref{assf}_{3}$ yields $\eqref{keymono}_{1}$. On the other hand, if $\mu=0$ we apply Lemma \ref{thebounds}
and \eqref{Vm}$_{1,2}$ to derive \eqref{keymono}$_{1}$ for all $z_{1},z_{2}\in \mathbb{R}^{N\times n}$ such that zero does not belong to the closed segment with endpoints $z_{1},z_{2}$.
Because the terms on the two sides of \eqref{keymono}$_{1}$ are continuous in $z_{1},z_{2}$ we obtain the general case by approximation. The bound in \eqref{keymono}$_{2}$ comes
as a straightforward consequence of \eqref{fy} and Lemma \ref{thebounds}. Indeed, observe that up to replace\footnote{Here we are subtracting a null Lagrangian from the convex
integrand $F$, so this does not affect variational problem \eqref{fff}.} $F(z)$ by $F(z)-F(0)-\langle F^{\prime}(0),z\rangle$ for all $z\in \mathbb{R}^{N\times n}$,
we can assume that $F(0)=0$, and $F^{\prime}(0)=0$, observe that this 
and the definition of $F^{\ast}$ imply that $F^{\ast}(0)=0$, set $\xi:=F^{\prime}(z)$  and bound
\begin{eqnarray*}
  F(z)&\stackrel{\eqref{fy},\eqref{derivatives}}{\ge}&\langle\xi,(F^{\ast})^{\prime}(\xi)\rangle-F^{\ast}(\xi)\stackrel{F^{\ast}(0)=0}{=}
  \int_{0}^{1}\langle (F^{\ast})^{\prime}(\xi)-(F^{\ast})^{\prime}(\lambda\xi),\xi\rangle\d\lambda\nonumber \\
&=&\int_{0}^{1}\int_{0}^{1}(1-\lambda)\langle(F^{\ast})^{\prime\prime}(\lambda\xi+t(1-\lambda)\xi)\xi,\xi\rangle\dtt\d\lambda\nonumber \\
&\stackrel{\eqref{dualbound}}{\ge}&\frac{1}{2L}\left(\int_{0}^{1}\int_{0}^{1}(1-\lambda)\ell_{1}(\lambda\xi+t(1-\lambda)\xi)^{q^{\prime}-2}\dtt\d\lambda\right)\snr{\xi}^{2}\nonumber \\
&\stackrel{q^{\prime}<2}{\ge}&c\ell_{1}(\xi)^{q^{\prime}-2}\snr{\xi}^{2}\stackrel{\eqref{vpvqn}}{\ge}c\snr{V_{1,q^{\prime}}(\xi)}^{2},
\end{eqnarray*}
with $c\equiv c(L,q)$, and $\eqref{keymono}_{2}$ follows including also the $p$-ellipticity information from $\eqref{assf}_{3}$.
Finally, \eqref{keymono.1} is a direct consequence of $\eqref{keymono}_{1}$ and standard difference quotients arguments.
\end{proof}

As a direct consequence of Corollary \ref{keylemma}, we deduce $\LL^{q'}$-integrability of the stress tensor for any function having $\F$ locally finite.
\begin{corollary}\label{ccc}
Let $B\Subset \Omega$ be any ball with radius smaller than one, and $v\in \WW^{1,p}(B,\mathbb{R}^{N})$ be a function such that $\F(v;B)<\infty$.
Then $F'(\nabla v)\in \LL^{q'}(B;\mathbb{R}^{N})$ with
\eqn{4.2.1}
$$
\nr{F'(\nabla v)}_{\LL^{q'}(B)}^{q'}\le c\F(v;B)+c,
$$
with $c\equiv c(n,L,p,q)$.
\end{corollary}
\begin{proof}
A direct computation gives
\begin{eqnarray*}
\snr{F'(\nabla v)}^{q'}&\stackrel{\eqref{f'}}{\le}&c+c\snr{F^{\prime}(\nabla v)-F^{\prime}(0)}^{q^{\prime}}\stackrel{\eqref{vpvqn}}{\le}c+\snr{V_{1,q'}(F^{\prime}(\nabla v)-F^{\prime}(0))}^{2}\nonumber \\
&\stackrel{\eqref{keymono}_{2}}{\le}&c+c\left(F(\nabla v)-F(0)-\langle F^{\prime}(0),\nabla v\rangle\right)\nonumber \\
&\stackrel{\eqref{f'}}{\le}& cF(\nabla v)+c\snr{\nabla v}^{p}+c\stackrel{\eqref{assf}_{2}}{\le}cF(\nabla v)+c,
\end{eqnarray*}
where we also used Young inequality with conjugate exponents $(p,p')$, and it is $c\equiv c(L,p,q)$. Integrating the content of the previous display on $B$ we obtain \eqref{4.2.1},
and the proof is complete.
\end{proof}

\subsection{Convex polynomials}\label{polyn} In this section we prove that for a reasonably large class of integrands, the usual natural growth
conditions and the quantified Legendre one actually coincide. In this respect, let us first point out that a straightforward consequence of \eqref{f'}, and \eqref{assf}$_{3}$ yield
\begin{equation}\label{controlled}
\snr{F^{\prime \prime}(z)}  \leq c\ell_{1}(z)^{q-2},
\end{equation}
for all $z\in \M$, and some positive $c\equiv c(L,q)$, that is the $(q-2)$-growth of $F''$ available in the controlled $(p,q)$-growth case.
It is easy to see that the converse is false, that is, the condition on the right-hand side of \eqref{assf}$_{3}$ is strictly stronger than \eqref{controlled} when $q>2$.
However, if we restrict our attention to a special family of polynomials we can show that the two growth conditions actually coincide. To this aim, let us recall that,
given a number $\tx{d}\in \mathbb{N}$, for our purposes a polynomial of degree $\tx{d}$ is a smooth real-valued function $P\colon \mathbb{R}^{N\times n}\to \mathbb{R}$ that
coincides with its Taylor expansion of order $\tx{d}$. For convenience, we shall group by homogeneity the various terms appearing in the polynomial, thus ultimately expressing
it as the sum of homogeneous components:
\eqn{polyy}
$$
P(z):=\sum_{s=0}^{\tx{d}}P_{s}(z)=\sum_{s=0}^{\tx{d}}\frac 1 {s \tn{!}} P^{(s)}(0)[\odot^{s} z],
$$
for all $z\in \mathbb{R}^{N\times n}$, where we used the definition of polynomial to explicitly identify $P_{s}$, that from now on, will be referred to as homogeneous
components of the polynomial. We are mostly interested in convex, even polynomial, resulting as the sum of nonnegative homogeneous components. In particular, the evenness
condition implies that the polynomial has even degree $2\tx{d}$, and in \eqref{polyy} only $2s$-homogeneous terms appear. Before proving that convex, even polynomials
of degree $2\tx{d}$ verify the quantified Legendre growth in $\eqref{assf}_{3}$ with $q=2\tx{d}$, let us recall the Euler relation for homogeneous functions: whenever
$\tx{g}\in \CC^{1}(\mathbb{R}^{N\times n})$ is positively $\kappa$-homogeneous for some $\kappa > 1$, then
\eqn{euler}
$$
\kappa\tx{g}(z)=\langle \tx{g}^{\prime}(z),z\rangle\qquad \mbox{for all} \ \ z\in \mathbb{R}^{N\times n}.
$$
Next, a technical lemma.
\begin{lemma}\label{lemma}
  Let $Q,H\in \CC^{1}(\mathbb{R}^{N\times n})$ be functions such that $H$ is convex, even, and $s$-homogeneous for some $s\ge 2$, and $\langle Q^{\prime}(z), z\rangle\ge 0$ for
  all $z\in \mathbb{R}^{N\times n}$. Then there exists $\delta\equiv \delta(H,s)\in [0,1)$, depending only on the structure of $H$ such that
\eqn{l.1}
$$
\langle Q^{\prime}(z), H^{\prime}(z)\rangle\ge-\delta\snr{Q^{\prime}(z)}\snr{H^{\prime}(z)}\qquad \mbox{for all} \ \ z\in \mathbb{R}^{N\times n}. 
$$
\end{lemma}
\begin{proof}
Without loss of generality we can assume that $H$ is not identically $0$. Convexity and $s$-homogeneity assure that $H(z)\ge 0$ for all $z\in \mathbb{R}^{N\times n}$.
Moreover, as $H^{1/s}$ is convex, $1$-homogeneous and nonnegative cf. \cite[Corollary 15.3.2]{Rockafellar}, it is
the support function of some compact, convex, symmetric set $K\subset \mathbb{R}^{N\times n}$, i.e., $H(z)=\left(\sup_{\xi\in K}\xi\cdot z\right)^s$ for all
$z\in \mathbb{R}^{N\times n}$. We then set $\mathbb{V}:=\tn{span}(K)$, decompose $\mathbb{R}^{N\times n}=\mathbb{V}\oplus \mathbb{V}^{\perp}$, and accordingly split
$\mathbb{R}^{N\times n}\ni z=z_{1}+z_{2}\in \mathbb{V}\oplus \mathbb{V}^{\perp}$, so that $H(z)\equiv H(z_{1})$, $H^{\prime}(z)\equiv H^{\prime}(z_{1})$ for all $z\in \mathbb{R}^{N\times n}$.
Being $K$ symmetric, zero belongs to the relative interior of $K$, and as $H$ is not identically $0$ there exists a positive radius $\sigma_{H}\equiv \sigma_{H}(H)>0$ such
that $\bar{B}_{\sigma_{H}}(0)\cap \mathbb{V}\subset K$ and it is 
\begin{equation}\label{infs}
\inf_{\omega\in \partial B_{\sigma_{H}}(0)}H(\omega)>0 \ \stackrel{\eqref{euler}}{\Longrightarrow} \ \inf_{\omega\in \partial B_{\sigma_{H}}(0)}\snr{H^{\prime}(\omega)}>0.
\end{equation}
Back to \eqref{l.1}, if $z\in \mathbb{V}^{\perp}$ there is nothing to prove as $H^{\prime}(z)=0$, so for the rest of the proof
we will take $z\in \mathbb{V}\setminus \{0\}$. We then decompose $H^{\prime}$ and $Q^{\prime}$ in an orthonormal fashion along $z$ - for $G\in \{H,Q\}$ we have:
\begin{flalign*}
  G^{\prime}(z)=\frac{\langle G^{\prime}(z), z\rangle z}{\snr{z}^{2}}+\left(G^{\prime}(z)-\frac{\langle G^{\prime}(z), z\rangle z}{\snr{z}^{2}}\right)
  =:\frac{\langle G^{\prime}(z), z\rangle z}{\snr{z}^{2}}+E_{G}(z).
\end{flalign*}
A direct computation then shows that
\begin{eqnarray*}
\langle H^{\prime}(z), Q^{\prime}(z)\rangle &=&\frac{\langle H^{\prime}(z), z\rangle\langle Q^{\prime}(z), z\rangle}{\snr{z}^{2}}+\langle E_{H}(z), E_{Q}(z)\rangle \nonumber \\
&\stackrel{\eqref{euler}}{=}&\frac{s H(z)\langle Q^{\prime}(z), z\rangle}{\snr{z}^{2}}+\langle E_{H}(z), E_{Q}(z)\rangle\nonumber \\
&\stackrel{\langle Q^{\prime}(z), z\rangle \ge 0}{\ge}&\langle E_{H}(z), E_{Q}(z)\rangle \ge -\snr{E_{H}(z)}\snr{Q^{\prime}(z)}.
\end{eqnarray*}
Now, if $\snr{E_{H}(z)}=0$ there is nothing to prove, otherwise by homogeneity we bound
\begin{eqnarray*}
  0&<&\snr{E_{H}(z)}^{2}=\snr{H^{\prime}(z)}^{2}-\frac{\langle H^{\prime}(z), z\rangle^{2}}{\snr{z}^{2}}\stackrel{\eqref{euler}}{=}\snr{H^{\prime}(z)}^{2}-\frac{s^{2}H(z)^{2}}{\snr{z}^{2}}
\nonumber \\
&=&\snr{H^{\prime}(z)}^{2}\left(1-\frac{s^{2}H(\sigma_{H} z/\snr{z})^2}{\sigma_{H}^{2}\snr{H^{\prime}(\sigma_{H}z/\snr{z})}^{2}}\right)\stackrel{\eqref{infs}}{\le}\snr{H'(z)}^{2}
\left(1-\inf_{\omega\in \partial B_{\sigma_{H}}(0)\cap \mathbb{V}}\frac{s^{2}H(\omega)^{2}}{\sigma_{H}^{2}\snr{H^{\prime}(\omega)}^{2}}\right)=:\snr{H^{\prime}(z)}^{2}\delta^{2},
\end{eqnarray*}
where we used that by continuity ($H\in \CC^{1}(\mathbb{R}^{N\times n})$) and thanks to \eqref{infs}, the infimum of $H/\snr{H^{\prime}}$ taken over
$\partial B_{\sigma_{H}}(0)\cap \mathbb{V}$ is strictly positive, thus $\delta\in[0,1)$.
\end{proof}
Now we are ready to state the main result of this section.
\begin{proposition}\label{poly} Let $\tx{d}\in \N$, and $P$ be an even, convex polynomial of degree $2\tx{d}$ with nonnegative homogeneous components.
  Then for each $i\in \{0,\ldots,2\tx{d}-2\}$, there exists $c\equiv c(n,N,P,\tx{d})>0$ such that
\begin{equation}\label{ePolyn}
\snr{P^{(i+2)}(z)}\leq c \left(1+\snr{P^\prime(z)}^{\frac{2\tx{d}-2-i}{2\tx{d}-1}}\right)\qquad \mbox{for all} \ \ z\in \R^{N\times n}.
\end{equation}
\end{proposition}
\begin{proof}
The proof relies on an induction argument on the degree of the polynomial.
\subsection*{Base step: \texorpdfstring{$\tx{d}=1$}{}}  Recalling that $P$ is even, by definition we have $P(z)=P(0)+ 2^{-1}\langle P^{\prime\prime}(0)z,z\rangle$,
with $P^{\prime\prime}(0)$ being positive semi-definite, and in particular not identically zero in the sense of matrices. Therefore we trivially have
$\snr{P^{\prime \prime}(z)}=\snr{P^{\prime\prime}(0)}\equiv \snr{P^{\prime\prime}(0)}(1+\snr{P^{\prime}(z)}^{0})$, that is \eqref{ePolyn} with $\tx{d}=1$, and $c_{0}:=\snr{P^{\prime\prime}(0)}>0$.
\subsection*{Inductive step} Assume now that \eqref{ePolyn} holds for all polynomials of degree $2j$ whenever $1\le j\le \tx{d}$, and let
$P\colon \mathbb{R}^{N\times n}\to \mathbb{R}$ be a convex, even polynomial of degree $2\tx{d}+2$, meaning in particular that $P^{(2\tx{d}+2)}(0)\not \equiv 0$,
and the $(2\tx{d}+2)$-homogeneous part of $P$ is given by
$$
\mathbb{R}^{N\times n}\ni z\mapsto H(z):=\frac{1}{(2\tx{d}+2)\tn{!}}P^{(2\tx{d}+2)}(0)[\odot^{2\tx{d}+2}z],
$$
so that $P=Q+H$, where $Q\colon \mathbb{R}^{N\times n}\to \mathbb{R}$ is an even polynomial of degree $2\ti{\tx{d}}\le 2\tx{d}$. By homogeneity we have
$$
\frac{P(tz)}{t^{2\tx{d}+2}}\to H(z)\quad \mbox{as} \ \ t\to \infty,
$$
pointwise with respect to $z\in \mathbb{R}^{N\times n}$, thus $H$ is convex. But then $z\mapsto H^{1/(2\tx{d}+2)}(z)$ is convex, $1$-homogeneous, and nonnegative, see \cite[Corollary 15.3.2]{Rockafellar}, so it must be the support function for a compact and convex subset $K\subset \mathbb{R}^{N\times n}$, i.e.: 
\eqn{supp}
$$
H(z)^{1/(2\tx{d}+2)}=\sup_{\xi\in K}\langle\xi, z\rangle\qquad \mbox{for all} \ \ z\in \mathbb{R}^{N\times n}.
$$ 
Since any nonempty, convex, closed set is uniquely determined by its support function, and $H\not \equiv 0$, then $K$ is symmetric, thus $0$ belongs to
the relative interior of $K$. Set $\mathbb{V}:=\textnormal{span}(K)$, and let $\sigma_{P}\equiv \sigma_{P}(P)>0$ be the largest number such that
$\bar{B}_{\sigma_{P}}(0)\cap\mathbb{V}\subset K$ - $\sigma_{P}$ is positive as $0$ is in the relative interior of $K$. After splitting
$z=z_{1}+ z_{2}\in  \mathbb{V}\otimes \mathbb{V}^{\perp}$, the formulation of $H$ as the support function of $K$ in \eqref{supp} yields that
$H^{(i+2)}(z)\equiv H^{(i+2)}(z_{1})$ for all $i\in \{0,\cdots,2\tx{d}\}$, so we bound
\begin{eqnarray*}
(2\tx{d}+2)\left(\frac{\snr{z_{1}}}{\sigma_{P}}\right)^{2\tx{d}+2}\inf_{\omega\in\partial B_{\sigma_{P}}(0)\cap\mathbb{V}}H(\omega)&\le& (2\tx{d}+2)H(z_{1})\nonumber \\
&\stackrel{\eqref{euler}}{=}& \langle H^\prime(z_{1}), z_{1}\rangle\le\snr{H^\prime(z_{1})}\snr{z_{1}}\nonumber \\
&\le& \frac{\snr{z_{1}}^{2\tx{d}+2}}{\sigma_{P}^{2\tx{d}+1}}\sup_{\omega\in\partial B_{\sigma_{P}}(0)\cap\mathbb{V}}\snr{H'(\omega)},
\end{eqnarray*}
so for some $c\equiv c(n,N,P,\tx{d})\ge 1$ it is
\eqn{uplow}
$$
c^{-1}\snr{z_{1}}^{2\tx{d}+1}\le \snr{H^{\prime}(z_{1})}\le c\snr{z_{1}}^{2\tx{d}+1},
$$
where we used the $(2\tx{d}+1)$-homogeneity of $H^{\prime}$. The very definition of $H$ now yields, for $i\in \{0,\ldots,2\tx{d}\}$,
\begin{equation}\label{eqH}
\snr{H^{(i+2)}(z_1)}\le c \snr{z_1}^{2\tx{d}-i} \stackrel{\eqref{uplow}}{\le} c \snr{H^\prime(z_1)}^\frac{2\tx{d}-i}{2\tx{d}+1},
\end{equation}
with $c\equiv c(n,N,P,\tx{d})$. We then jump back to the decomposition $P=Q+H$. Recalling that $H\equiv H(z_{1})$, on $\mathbb{V}^{\perp}$ it is $P=Q$,
so $\left.Q\right|_{\mathbb{V}^{\perp}}$ is an even, convex polynomial of degree $2\ti{\tx{d}}\le 2\tx{d}$. In particular, by the induction hypothesis we have
\begin{equation}\label{indQ}
\snr{Q^{(i+2)}(z_{2})}\le c\left(1+\snr{Q^\prime(z_{2})}^{\frac{2\ti{\tx{d}}-2-i}{2\ti{\tx{d}}-1}}\right)\qquad \mbox{for all} \ \ z_{2}\in \mathbb{V}^{\perp}, \ \ i\in \{0,\cdots,2\ti{\tx{d}}-2\}.
\end{equation}
Next, Taylor-expanding around $z_{2}$, we obtain
\begin{eqnarray}\label{3.18}
\snr{P^{(2+i)}(z)}&\leq&\snr{Q^{(2+i)}(z_{1}+z_{2})}+\snr{H^{(2+i)}(z_{1})}\nonumber \\
&\le&\left|\ \sum_{s=0}^{2\ti{\tx{d}}-2-i}\frac{1}{s\tn{!}}Q^{(2+i+s)}(z_{2})[\odot^{s}z_{1}]\ \right|+\snr{H^{(2+i)}(z_{1})}\nonumber \\
&\leq&c\sum_{s=1}^{2\ti{\tx{d}}-2-i}\frac{1}{s\tn{!}}\snr{Q^{(2+i+s)}(z_{2})}\snr{z_{1}}^s+c\snr{Q^{(2+i)}(z_{2})}+c\snr{H^{(2+i)}(z_{1})}\nonumber \\
&\leq& c\sum_{s=1}^{2\tilde{\tx d}-2-i} \snr{Q^{(2+i+s)}(z_2)}^\frac{2\tx{d}-i}{2\tx{d}-i-s}+c\snr{z_1}^{2\tx{d}-i}+\snr{Q^{(2+i)}(z_2)}+\snr{H^{(2+i)}(z_1)}\nonumber \\
&\stackrel{\eqref{eqH},\eqref{indQ}}{\leq}&c \sum_{s=0}^{2\ti{\tx d}-2-i}\left(1+\snr{Q^\prime(z_2)}^\frac{(2\ti{\tx{d}}-s-i-2)(2\tx{d}-i)}{(2\ti{\tx d}-1)(2\tx{d}-i-s)}\right)
+c \snr{H^\prime(z_1)}^\frac{2\tx{d}-i}{2\tx{d}+1}\nonumber \\
&\leq& c\left(1+\snr{Q^\prime(z_2)}^\frac{2\tx{d}-i}{2\tx{d}+1}\right)+c\snr{H^\prime(z_1)}^\frac{2\tx{d}-i}{2\tx{d}+1},
\end{eqnarray}
where $c\equiv c(n,N,P,\tx{d})$ and we used Young inequality with conjugate exponents $\left(\frac{2\tx{d}-i}{2\tx{d}-i-s},\frac{2\tx{d}-i}{s}\right)$,
and observed that for $s\geq 0$ and $i\geq -1$, it is
\begin{equation}\label{exponents}
\frac{2\tilde{\tx d}-s-i-2}{2\tx{d}-i-s}\leq \frac{2{\tilde{\tx{d}}}-1}{2{\tx{d}}+1}.
\end{equation}
Via Taylor expansion, we further estimate
\begin{eqnarray*}
\snr{Q^{\prime}(z_{2})-Q^{\prime}(z_{1}+z_{2})}&\le&c\sum_{s=1}^{2\ti{\tx{d}}-1}\snr{Q^{(1+s)}(z_{2})}\snr{z_{1}}^{s}\nonumber \\
&\le&c\sum_{s=1}^{2\ti{\tx{d}}-1}\left(\varepsilon\snr{Q^{(1+s)}(z_{2})}^{\frac{2\tx{d}+1}{2\tx{d}+1-s}}+\frac{1}{\varepsilon^{\frac{2\tx{d}+1-s}{s}}}\snr{z_{1}}^{2\tx{d}+1}\right)\nonumber \\
&\stackrel{\eqref{indQ}}{\le}&c\varepsilon\sum_{s=1}^{2\ti{\tx{d}}-1}\left(1+\snr{Q^\prime(z_{2})}^{\frac{2\ti{\tx{d}}-1-s}{2\ti{\tx{d}}-1}}\right)^{\frac{2\tx{d}+1}{2\tx{d}+1-s}}
+\frac{c}{\varepsilon^{2\tx{d}}}\snr{z_{1}}^{2\tx{d}+1}\nonumber \\
&\stackrel{\eqref{exponents}_{i=-1}}{\le}&c\varepsilon\left(1+\snr{Q^{\prime}(z_{2})}\right)+\frac{c}{\varepsilon^{2\tx{d}}}\snr{z_{1}}^{2\tx{d}+1},
\end{eqnarray*}
for $c\equiv c(n,N,P,\tx{d})$ - here we used Young inequality with conjugate exponents $\left(\frac{2\tx{d}+1}{2\tx{d}+1-s},\frac{2\tx{d}+1}{s}\right)$.
The content of the previous display then gives
\begin{eqnarray*}
\snr{Q^{\prime}(z_{2})}&\le& c\varepsilon\left(1+\snr{Q^{\prime}(z_{2})}\right)+\snr{Q^{\prime}(z)}+\frac{c}{\varepsilon^{2\tx{d}}}\snr{z_{1}}^{2\tx{d}+1}\nonumber\\
&\stackrel{\eqref{uplow}}{\le}&c\varepsilon\left(1+\snr{Q^{\prime}(z_{2})}\right)+\snr{Q^{\prime}(z)}+\frac{c}{\varepsilon^{2\tx{d}}}\snr{H^{\prime}(z_{1})},
\end{eqnarray*}
so choosing $\varepsilon\equiv \varepsilon(n,N,P,\tx{d})\in (0,1)$ sufficiently small we end up with
\eqn{3.17}
$$
\snr{Q^{\prime}(z_{2})}\le c(1+\snr{Q^{\prime}(z)})+c\snr{H^{\prime}(z_{1})},
$$
for $c\equiv c(n,N,P,\tx{d})$. Merging estimates \eqref{3.18} and \eqref{3.17} we eventually get
\begin{flalign*}
  \snr{P^{(2+i)}(z)} \le c\left(1+\snr{Q^{\prime}(z)}^{2}+\snr{H^{\prime}(z)}^{2}\right)^{\frac{2\tx{d}-i}{2(2\tx{d}+1)}},
\end{flalign*}
with $c\equiv c(n,N,P,\tx{d})$. Finally, recalling that $Q$ is (at most) the sum of $2s$-homogeneous terms of $P$ with $s\in \{0,\cdots,\tx{d}\}$, i.e.:
$$
Q(z)=P(z)-H(z)=\sum_{s=0}^{2\tx{d}}P_{s}(z),\qquad \quad P_{2s+1}\equiv 0\quad \mbox{for all} \ \ s\in \{0,\cdots,\tx{d}-1\},
$$
the nonnegativity of the $P_{s}$'s imply
$$
\langle P^{\prime}_{s}(z),z\rangle \stackrel{\eqref{euler}}{=}sP_{s}(z)\ge 0 \ \Longrightarrow \ \langle Q^{\prime}(z),z\rangle \ge 0.
$$
Keeping in mind that $H$ is convex, even and $(2\tx{d}+2)$-homogeneous, by Lemma \ref{lemma} we find $\delta\equiv \delta(n,N,P,\tx{d})\in [0,1)$ such that
$$
  (1-\delta)(\snr{Q^\prime(z)}^2+\snr{H^\prime(z)}^2)\leq \snr{Q^\prime(z)}^2+\snr{H^\prime(z)}^2-2\delta \snr{Q^\prime(z)}\snr{H^\prime(z)}
  \stackrel{\eqref{l.1}}{\leq} \snr{Q^\prime(z)+H^\prime(z)}^2= \snr{P^\prime(z)}^2.
$$
Combining the last four displays we end up with \eqref{ePolyn}$_{\tx{d}+1}$ and the proof is complete.
\end{proof}
As a direct consequence of Lemma \ref{lemma} and Proposition \ref{poly} we can show that whenever a convex, even, positively $s$-homogeneous function is
added to an integrand satisfying Legendre $(p,q)$-growth conditions, the sum verifies Legendre $(p,\max(s,q))$-growth conditions.
\begin{corollary}\label{cor38}
Let $Q,H\in \CC^{2}(\mathbb{R}^{N\times n})$ be functions such that $Q$ satisfies $\eqref{assf}_{3}$, for some exponents $2\le p\le q$, and $H$ is convex, even,
and $s$-homogeneous for some $s\ge 2$. Then the sum $Q+H$ satisfies $(p,\max\{q,s\})$-Legendre growth conditions.
\end{corollary}
\begin{proof}
The lower bound in $\eqref{assf}_{3}$ is preserved by the sum $Q+H$ as $Q$ satisfies $\eqref{assf}_{3}$, and $H$ is convex. Concerning the upper bound, if $z=0$,
there is nothing to prove, otherwise, being $H$ convex, $s$-homogeneous ($s\ge 2$), and even, we can repeat verbatim the construction detailed in
the "Inductive step" of the proof of Proposition \ref{poly} up to \eqref{uplow} to decompose $\mathbb{R}^{N\times n}=\mathbb{V}\oplus \mathbb{V}^{\perp}$, so that
on $\mathbb{V}$ the double bound in \eqref{uplow} holds true for the map $H$. Splitting $\mathbb{R}^{N\times n}\ni z=z_{1}+z_{2}\in \mathbb{V}\oplus\mathbb{V}^{\perp}$, we control
\begin{eqnarray*}
  \snr{Q^{\prime\prime}(z)+H^{\prime\prime}(z)}&\stackrel{\eqref{assf}_{3},\eqref{supp}}{\le}&c\left(1+\snr{Q^{\prime}(z)}^{\frac{q-2}{q-1}}\right)
  +\snr{z_{1}}^{s-2}\left | H^{\prime\prime}\left(z_{1}/\snr{z_{1}}\right)\right|\nonumber \\
&\stackrel{\eqref{uplow}}{\le}&c\left(1+\snr{Q^{\prime}(z)}^{\frac{q-2}{q-1}}\right)+c\snr{H^{\prime}(z_{1})}^{\frac{s-2}{s-1}}\nonumber \\
&\leq& c\left(1+\snr{Q^\prime(z)}^2+\snr{H^\prime(z)}^2\right)^\frac{\max\{s,q\}-2}{2(\max\{s,q\}-1)}\\
&\stackrel{\eqref{l.1}}{\le}&c\left(1+\snr{Q^{\prime}(z)+H^{\prime}(z)}^{\frac{\max\{s,q\}-2}{\max\{s,q\}-1}}\right),
\end{eqnarray*}
with $c\equiv c(n,N,L,p,q,s,H)$.
\end{proof}
Next, we focus on those convex, even polynomials that are strongly $p$-elliptic in the sense of $\eqref{assf}_{3}$ (left-hand side), and show that in the
degenerate case $p>2$, $\mu=0$, the bound in \eqref{assf.1} holds true.
\begin{corollary}\label{poly.er}
Let $\tx{d}_{0},\tx{d}\in \mathbb{N}$, be such that $2< 2\tx{d}_{0}<2\tx{d}$, and let $P$  be an even, convex polynomial of degree $2\tx{d}$ with nonnegative
homogeneous components $P_{s}$ such that $P_{s}\equiv 0$ for all $s\in \{1,\cdots,2(\tx{d}_{0}-1)\}$. 
Then \eqref{assf.1} is verified with $q=2\tx{d}$, for all $2< p\le 2\tx{d}_{0}$, that is
\eqn{er.0}
    $$
  \frac{\snr{P^{\prime\prime}(z)}}{\snr{z}^{p-2}}\le c\left(1+\snr{P^{\prime}(z)}^{\frac{2\tx{d}-p}{2\tx{d}-1}}\right)\qquad \quad \mbox{for all} \ \ z\in \mathbb{R}^{N\times n}\setminus \{0\},
    $$
    with $c\equiv c(n,N,p,\tx{d},P)$.
\end{corollary}

\begin{proof}
By definition\footnote{By evenness all the homogeneous components of $P$ of odd degree vanish.} of $P$,
\eqn{hom.1}
$$
\snr{P^{\prime\prime}(z)}\le \sum_{s=2\tx{d}_{0}}^{2\tx{d}}\snr{P_{s}^{\prime\prime}(z)}\le\sum_{s=\tx{d}_{0}}^{\tx{d}}\snr{z}^{2s-2}
\left|P_{2s}^{\prime\prime}\left(z/\snr{z}\right)\right|\le c \sum_{s=\tx{d}_{0}}^{\tx{d}}\snr{z}^{2s-2},
$$
with $c\equiv c(P,\tx{d})$. Now, keeping in mind that $2< p\le 2\tx{d}_{0}$, if $\snr{z}\in (0,1)$ we bound
$$
\frac{\snr{P^{\prime\prime}(z)}}{\snr{z}^{p-2}}\stackrel{\eqref{hom.1}}{\le} c\sum_{s=\tx{d}_{0}}^{\tx{d}}\snr{z}^{2s-p}\stackrel{\snr{z}\le 1}{\le}c,
$$
for $c\equiv c(P,p,\tx{d})$. We next observe that, whenever $\snr{z}\ge 1$, the $(2s-1)$-homogeneity of $P_{2s}^\prime$ gives
\eqn{p'}
$$
\snr{P^{\prime}(z)}\le \sum_{s=2\tx{d}_{0}}^{2\tx{d}}\snr{P_{s}}^{\prime}(z)\le \sum_{s=\tx{d}_{0}}^{\tx{d}}\snr{z}^{2s-1}\left|P_{2s}^{\prime}\left(z/\snr{z}\right)\right|
\le c\sum_{s=\tx{d}_{0}}^{\tx{d}}\snr{z}^{2s-1}\stackrel{\snr{z}\ge 1}{\le}c\snr{z}^{2\tx{d}-1},
$$
for $c\equiv c(P,\tx{d})$, so we estimate
\begin{flalign*}
  \frac{\snr{P^{\prime\prime}(z)}}{\snr{z}^{p-2}}\stackrel{\eqref{ePolyn}}{\le}\frac{c}{\snr{z}^{p-2}}\left(1+\snr{P^{\prime}(z)}^{\frac{2\tx{d}-2}{2\tx{d}-1}}\right)
  \stackrel{\snr{z}\ge 1}{\le}c+\frac{c\snr{P^{\prime}(z)}^{\frac{2\tx{d}-2}{2\tx{d}-1}}}{\snr{z}^{p-2}}\stackrel{\eqref{p'}}{\le}c\left(1+\snr{P^{\prime}(z)}^{\frac{2\tx{d}-p}{2\tx{d}-1}}\right),
\end{flalign*}
with $c\equiv c(P,p,\tx{d})$, and the proof is complete.
\end{proof}
\begin{remark}
  \emph{Thanks to Corollary \ref{cor38}, we immediately see that the class of integrands featuring Legendre $(p,q)$-growth conditions embraces the anisotropic
    polynomial examples introduced by Marcellini in \cite{ma1}. In fact, repeated applications of Corollary \ref{cor38} to the integrand in
    \eqref{modelpq2}-\eqref{modelpq.2} yield that $P$ satisfies $(p,q_{n})$-Legendre growth conditions, and the prototypical examples of integrands with $(p,q)$-growth are covered.}
\end{remark}

\subsection{An abstraction of the quantified Legendre condition}\label{abstraction} Let $F$, $G \colon \M \to \R$ be two strictly convex, $\CC^1$ integrands
that are super-linear, i.e.:
$$
\frac{F(z)}{|z|} \to \infty \, \mbox{ and } \, \frac{G(z)}{|z|} \to \infty \, \mbox{ as } \, |z| \to \infty .
$$
Consequently also the Fenchel conjugates $F^\ast$ and $G^\ast$ are strictly convex, $\CC^1$ and super-linear, that is, $F$ and $G$ real-valued, are super-linear Legendre integrands.
The goal of this subsection is the following result and its corollary, that can be seen as an abstract counterpart of $\eqref{assf}_{3}$, covering more general convexity conditions than the usual power type one.

\begin{proposition}\label{abstractprop}
The difference of the Fenchel conjugates $F^{\ast}-G^{\ast}$, is convex if and only if 
$$
F(z+z_{0})-F(z_{0})-\langle F^{\prime}(z_{0}),z\rangle \leq G\left( z+ (G^{\ast})^{\prime}(F^{\prime}(z_{0})) \right) - G\left( (G^{\ast})^{\prime}(F^{\prime}(z_{0})) \right)
- \langle G^{\prime}\left( (G^{\ast})^{\prime}(F^{\prime}(z_{0})) \right),z\rangle
$$
holds for all $z$, $z_{0} \in \M$.
\end{proposition}

The condition can be recast as a second order condition if we interpret it in the viscosity sense:

\begin{corollary}\label{abstractcor}
The difference of the Fenchel conjugates $F^{\ast}-G^{\ast}$, is convex if and only if 
$$
F^{\prime \prime}(z_{0}) \leq G^{\prime \prime}\left( (G^{\ast})^{\prime}(F^{\prime}(z_{0})) \right)
$$
holds as quadratic forms in the viscosity sense for all $z_{0} \in \M$.
\end{corollary}

\begin{proof}[Proof of Proposition \ref{abstractprop}]
Note that the derivatives $F^\prime$ and $G^\prime$ are homeomorphisms of $\M$ with inverses $(F^{\ast})^\prime$ and $(G^{\ast})^\prime$, respectively.
Fix $z_0 \in \M$ and put $\xi_{0} = F^{\prime}(z_{0})$. Now $F^{\ast}-G^{\ast}$ is convex at $\xi_0$ if and only if
$$
\left( F^{\ast}-G^{\ast}\right)( \xi ) \geq \left( F^{\ast}-G^{\ast}\right)( \xi_{0} )+\langle\left( F^{\ast}-G^{\ast}\right)^{\prime}(\xi_{0} ),\xi - \xi_{0}\rangle
$$
holds for all $\xi \in \M$, that is,
$$
F^{\ast}( \xi ) - F^{\ast}( \xi_{0}) - \langle(F^{\ast})^{\prime}( \xi_{0}) , \xi - \xi_{0}\rangle \geq 
G^{\ast}( \xi ) - G^{\ast}( \xi_{0}) - \langle(G^{\ast})^{\prime}( \xi_{0}), \xi - \xi_{0}\rangle
$$
for all $\xi \in \M$. Because $(F^{\ast})^{\prime}( \xi_{0})=z_0$ we get by Fenchel conjugation that this inequality is equivalent to
\begin{flalign*}
&\sup_{\xi\in \mathbb{R}^{N\times n}} \left( \langle z, \xi\rangle - F^{\ast}( \xi ) + F^{\ast}( \xi_{0}) + \langle (F^{\ast})^{\prime}( \xi_{0}),  \xi - \xi_{0}\rangle \right) \nonumber \\
&\qquad \qquad \quad \leq 
\sup_{\xi\in \mathbb{R}^{N\times n}} \left( \langle z, \xi\rangle - G^{\ast}( \xi ) + G^{\ast}( \xi_{0}) + \langle (G^{\ast})^{\prime}( \xi_{0}) , \xi - \xi_{0}\rangle \right)
\end{flalign*}
for all $z \in \M$. Recalling that $F^{\ast \ast} = F$, $G^{\ast \ast} = G$ we can rewrite this as
$$
F(z+z_{0}) +F^{\ast}( \xi_{0})-\langle z_{0}, F^{\prime}(z_{0})\rangle \leq G(z+(G^{\ast})^{\prime}(\xi_{0}))+G^{\ast}(\xi_{0})-\langle (G^{\ast})^{\prime}( \xi_{0}) , \xi_{0}\rangle
$$
for all $z \in \M$. Because $F(z_{0})+F^{\ast}(\xi_{0})=\langle z_{0}, \xi_{0}\rangle$ and
$G^{\ast}(\xi_{0})+G( (G^{\ast})^{\prime}( \xi_{0})) = \langle\xi_{0}, (G^{\ast})^{\prime}( \xi_{0})\rangle$ the last inequality can be rewritten as
$F(z+z_{0})-F(z_{0}) \leq G(z+(G^{\ast})^{\prime}( \xi_{0})) - G( (G^{\ast})^{\prime}( \xi_{0}))$ for all $z \in \M$. Consequently, using that
$\xi_{0}=F^{\prime}(z_{0})$ and $G^{\prime}( (G^{\ast})^{\prime}(F^{\prime}(z_{0})))=\xi_{0}$ we arrive at the inequality
$$
F(z+z_{0})-F(z_{0})-\langle F^{\prime}(z_{0}) , z\rangle \leq G(z+(G^{\ast})^{\prime}(F^{\prime}(z_{0}))) - G((G^{\ast})^{\prime}(F^{\prime}(z_{0}))) - \langle F^{\prime}(z_{0}) , z\rangle
$$
for all $z \in \M$, as required.
\end{proof}

\section{Higher differentiability under Legendre \texorpdfstring{$(p,q)$}{}-growth}\label{hdd}
In this section we derive some regularity results for local minimizers of variational integrals under Legendre $(p,q)$-growth.
Our focus is on the higher differentiability of minima, that will subsequently be used to prove finer regularity in low dimension and in the scalar setting.

\subsection{Approximation scheme}\label{as} Our approximation scheme is rather basic and aims at correcting two relevant structural issues of the integrand $F$:
unbalanced growth and possible degeneracy. Let $B\equiv B_{\rrr}(x_{B})\Subset \Omega$ be a ball with radius $0<\texttt{r}\le 1$.
We regularize $u\in \WW^{1,p}_{\loc}(\Omega,\mathbb{R}^{N})$ via convolution against a sequence of mollifiers $\{\phi_{\varepsilon}\}_{\varepsilon>0}\subset \CC^{\infty}_{c}(\Rn )$
thus determining a sequence $\{\ti{u}_{\varepsilon}\}_{\varepsilon>0}:=\{u*\phi_{\varepsilon}\}_{\varepsilon>0}\subset \CC^{\infty}_{\loc}(\Omega,\Y)$, set
\eqn{ge}
$$
\gamma_{\varepsilon}:=\left(1+\varepsilon^{-1}+\varepsilon^{-1}\nr{\nabla \ti{u}_{\varepsilon}}_{\LL^{q}(B)}^{2q}\right)^{-1}\qquad
\mbox{so that}\qquad \gamma_{\varepsilon}\nr{\nabla\ti{u}_{\varepsilon}}_{\LL^{q}(B)}^{q}\searrow_{\varepsilon\to 0}0,
$$
define integrand
\eqn{ge.1}
$$
\M \ni z\mapsto F_{\varepsilon}(z):=F(z)+\gamma_{\varepsilon}\ell_{1}(\nabla u_{\varepsilon})^{q},
$$
and introduce the family of approximating functionals
$$
\WW^{1,q}(B,\Y)\ni w\mapsto  \F_{\varepsilon}(w;B):=\int_{B}F_{\varepsilon}(\nabla w)\d x.
$$
By \eqref{assf}$_{3}$ and basic direct methods, the Dirichlet problem
\eqn{pde}
$$
\ti{u}_{\varepsilon}+\WW^{1,q}_{0}(B,\Y)\ni w\mapsto \min \F_{\varepsilon}(w;B)
$$
admits a unique solution $u_{\varepsilon}\in \left(\ti{u}_{\varepsilon}+\WW^{1,q}_{0}(B,\Y)\right)$, verifying by minimality the integral identity
\eqn{els}
$$
0=\int_{B}\langle F'_{\varepsilon}(\nabla u_{\varepsilon}),\nabla w\rangle\d x\qquad \mbox{for all} \ \ w\in \WW^{1,q}_{0}(B,\Y).
$$
The convergence features of the sequence of minima obtained solving the approximating problems in \eqref{pde} are well known, see e.g. \cite[Section 3]{ckp1}:
\eqn{approx}
$$
\F_{\varepsilon}(u_{\varepsilon};B)\to \F(u;B),\qquad \gamma_{\varepsilon}\nr{\nabla u_{\varepsilon}}_{L^{q}(B)}^{q}\searrow 0,\qquad u_{\varepsilon}
\to u \ \ \mbox{strongly in} \ \ \WW^{1,p}(B,\Y).
$$
Let us quickly show that $F_{\varepsilon}\in \CC^{2}(\mathbb{R}^{N\times n})$ enjoys both growth/ellipticity features of the $q$-Laplacian type, and Legendre $(p,q)$-growth.
\begin{lemma}
With \eqref{assf} in force, let $F_{\varepsilon}$ be the integrand defined in \eqref{ge.1}. Then
\begin{flalign}\label{cdq}
\left\{
\begin{array}{c}
\displaystyle
\ \gamma_{\varepsilon}\ell_{1}(z)^{q}\le F_{\varepsilon}(z)\le \Lambda\ell_{1}(z)^{q},\\ [8pt]\displaystyle
\ \frac{1}{\Lambda}\ell_{\mu}(z)^{p-2}\snr{\xi}^{2}+\frac{\gamma_{\varepsilon}}{\Lambda}\ell_{1}(z)^{q-2}\snr{\xi}^{2}\le \langle F''_{\varepsilon}(z)\xi,\xi\rangle,\\ [8pt]\displaystyle
\ \snr{F''_{\varepsilon}(z)}\le \Lambda\ell_{1}(z)^{q-2},\\ [8pt]\displaystyle
\ \snr{F''_{\varepsilon}(z)}\le \Lambda \left(1+\snr{F'_{\varepsilon}(z)}^{\frac{q-2}{q-1}}\right),
\end{array}
\right.
\end{flalign}
for all $z,\xi\in \mathbb{R}^{N\times n}$ and some $\Lambda\equiv \Lambda(n,N,L,p,q)$.
\end{lemma}
\begin{proof}
The first three bounds in \eqref{cdq} are a straightforward consequence of \eqref{assf}, \eqref{f'} and the very definition of $F_{\varepsilon}$, so we focus on $\eqref{cdq}_{4}$.
We first prove that there exists a constant $\ti{c}\equiv \ti{c}(L,p,q)$ such that if $\snr{z}\ge \ti{c}$, then $\langle F'(z),z\rangle>0$.
In fact, by Young inequality with conjugate expoenents $(p,p')$ we have
\begin{eqnarray*}
\langle F'(z),z\rangle&=&\langle F'(z)-F'(0),z\rangle+\langle F'(0),z\rangle\nonumber \\
&\stackrel{\eqref{keymono}_{1},\eqref{f'}}{\ge}&c\snr{V_{\mu,p}(z)}^{2}-c\snr{z}\ge c'\snr{z}^{p}-c'',
\end{eqnarray*}
with $c',c''\equiv c',c''(L,p,q)$, thus letting $\ti{c}:=\max\{(2c''/c')^{1/p},1\}\equiv \ti{c}(L,p,q)\ge 1$ we get 
\eqn{cdq.1}
$$
\snr{z}\ge \ti{c} \ \Longrightarrow \ \langle F'(z),z\rangle\ge \min\left\{\frac{c'}{2},c''\right\}\snr{z}^{p}>0.
$$
Next, by $\eqref{cdq}_{3}$ we have 
\eqn{cdq.3}
$$
\snr{z}\le \ti{c}\ \Longrightarrow \ \snr{F_{\varepsilon}''(z)}\le c(n,N,L,p,q),
$$
while if $\snr{z}\ge \ti{c}$ we compute
\begin{eqnarray}\label{cdq.2}
\snr{F_{\varepsilon}'(z)}^{2}&=&\snr{F'(z)}^{2}+q^{2}\gamma_{\varepsilon}^{2}\ell_{1}(z)^{2(q-2)}\snr{z}^{2}+2q\gamma_{\varepsilon}\ell_{1}(z)^{q-2}\langle F'(z),z\rangle\nonumber \\
&\stackrel{\eqref{cdq.1}}{\ge}&\snr{F'(z)}^{2}+\gamma_{\varepsilon}^{2}q^{2}\ell_{1}(z)^{2(q-2)}\snr{z}^{2} \ \Longrightarrow \ \snr{F'(z)}+q\gamma_{\varepsilon}\ell_{1}(z)^{q-2}\snr{z}\le 2\snr{F_{\varepsilon}'(z)},
\end{eqnarray}
therefore
\begin{eqnarray}\label{cdq.4}
\snr{F_{\varepsilon}''(z)}&\stackrel{\eqref{assf}_{3}}{\le}&L\left(1+\snr{F'(z)}^{\frac{q-2}{q-1}}\right)+c\gamma_{\varepsilon}\ell_{1}(z)^{q-2}\nonumber \\
&\stackrel{\snr{z}\ge \ti{c},\ \gamma_{\varepsilon}<1}{\le}&L\left(1+\snr{F'(z)}^{\frac{q-2}{q-1}}\right)+c\left(\gamma_{\varepsilon}\ell_{1}(z)^{q-2}\snr{z}\right)^{\frac{q-2}{q-1}}
\stackrel{\eqref{cdq.2}}{\le}c\left(1+\snr{F_{\varepsilon}'(z)}^{\frac{q-2}{q-1}}\right),
\end{eqnarray}
for $c\equiv c(n,N,L,p,q)$. Combining \eqref{cdq.3} and \eqref{cdq.4} we obtain $\eqref{cdq}_{4}$ and the proof is complete.
\end{proof}
Thanks to \eqref{cdq}$_{1,2,3}$, by-now classical regularity theory \cite[Chapter 8]{giu} yields that 
\eqn{3.1.0}
$$
V_{1,q}(\nabla u_{\varepsilon}), \  \nabla u_{\varepsilon}\in \WW^{1,2}_{\loc}(B,\mathbb{R}^{N\times n})
$$
in particular, by \eqref{3.1.0} via difference quotients arguments, \eqref{els} can be differentiated: system
\eqn{elsd}
$$
0=\int_{B}\langle F''_{\varepsilon}(\nabla u_{\varepsilon})\partial_{s}\nabla u_{\varepsilon},\nabla w\rangle\dx
$$
holds for any $w\in \WW^{1,2}(B,\Y)$ such that $\supp(w)\Subset B$, and all $s\in \{1,\cdots,n\}$. Finally, we record for later use the elementary energy estimates
\begin{flalign}
L^{-1}\nr{\nabla u_{\varepsilon}}_{\LL^{p}(B)}^{p}+\gamma_{\varepsilon}\nr{\nabla u_{\varepsilon}}_{\LL^{q}(B)}^{q}\stackrel{\eqref{assf}_{2}}{\le}\F_{\e}(u_{\varepsilon};B)\stackrel{\eqref{approx}_{1}}{\le}
\F(u;B)+\texttt{o}(\varepsilon).\label{enes}
\end{flalign}

\subsection{Proof of Theorem \ref{hd}}\label{ptt}
We start by recording the uniform $L^{q'}$-bound on the stress tensors $\{F'(\nabla u_{\varepsilon})\}_{\varepsilon>0}$:
\begin{flalign}
  \nr{F'(\nabla u_{\varepsilon})}_{\LL^{q'}(B)}^{q'}\stackrel{\eqref{4.2.1}}{\le}c\F(u_{\varepsilon};B)+\texttt{o}(\varepsilon)
  +c\stackrel{\eqref{enes}}{\le}c\F(u;B)+\texttt{o}(\varepsilon)+c,\label{4.2.0}
\end{flalign}
for $c\equiv c(n,L,p,q)$. We fix a ball $B_{r}(x_{0})\subseteq B$, let $\rr\in (0,3r/4)$, introduce the following symbols:
\begin{flalign*}
&\texttt{H}_{\varepsilon}(x_{0};\rr):=\sum_{s=1}^{n}\int_{B_{\rr}(x_{0})}\langle F_{\varepsilon}''(\nabla u_{\varepsilon})\partial_{s}\nabla u_{\varepsilon},\partial_{s}\nabla u_{\varepsilon}\rangle\dx,
\qquad \texttt{S}_{\varepsilon}(x_{0};\rr):=\int_{B_{\rr}(x_{0})}\ell_{1}(F'(\nabla u_{\varepsilon}))^{q'}\dx\nonumber \\
&\texttt{P}_{\varepsilon}(x_{0};\rr):=\int_{B_{\rr}(x_{0})}\ell_{\mu}(\nabla u_{\varepsilon})^{p}\dx,\qquad \qquad \qquad\qquad \qquad \texttt{Q}_{\varepsilon}(x_{0};\rr)
:=\gamma_{\varepsilon}\int_{B_{\rr}(x_{0})}\ell_{1}(\nabla u_{\varepsilon})^{q}\dx\nonumber \\
&\texttt{V}_{\varepsilon,p}(x_{0};\rr):=\snr{V_{\mu,p}(\nabla u_{\varepsilon})-(V_{\mu,p}(\nabla u_{\varepsilon}))_{\partial B_{\rr}(x_{0})}},
\qquad \quad \ \ \texttt{V}_{\varepsilon,q}(x_{0};\rr):=\snr{V_{1,q}(\nabla u_{\varepsilon})-(V_{1,q}(\nabla u_{\varepsilon}))_{\partial B_{\rr}(x_{0})}},
\end{flalign*}
and notice that by \eqref{3.1.0} all the above quantities are finite. Before proceeding further, a brief remark about notation:
since all balls considered from now on will be centered at $x_{0}$, we shall omit denoting it and simply write $B_{\rr}(x_{0})\equiv B_{\rr}$.
With $0<\delta<\min\{1,\dist(B_{\rr},\partial B)/100\}$, we regularize the characteristic function of $B_{\rr}$ by convolution against a sequence
$\{\phi_{\delta}\}_{\delta>0}\subset \CC^{\infty}_{c}(\mathbb{R}^{n})$ of standard mollifiers , thus defining
$\{\eta_{\delta}\}_{\delta>0}:=\{\mathds{1}_{B_{\rr}}*\phi_{\delta}\}_{\delta>0}\subset \CC^{\infty}_{c}(B_{\rr+2\delta})$, test \eqref{elsd} against
$w_{\delta}:=\eta_{\delta}(\partial_{s}u_{\varepsilon}-(\partial_{s}u_{\varepsilon})_{\partial B_{\rr}})$, which is admissible by \eqref{3.1.0}, and sum on $s\in \{1,\cdots,n\}$.
We obtain:
\begin{flalign*}
&\sum_{s=1}^{n}\int_{B_{\rr+2\delta}}\eta_{\delta}\langle F_{\varepsilon}''(\nabla u_{\varepsilon})\partial_{s}\nabla u_{\varepsilon},\partial_{s}\nabla u_{\varepsilon}\rangle\dx\nonumber \\
  &\qquad \qquad =-\sum_{s=1}^{n}\int_{B_{\rr+2\delta}}\langle F_{\varepsilon}''(\nabla u_{\varepsilon})\partial_{s}\nabla u_{\varepsilon},
  (\partial_{s}u_{\varepsilon}-(\partial_{s}u_{\varepsilon})_{\partial B_{\rr}})\otimes \nabla \eta_{\delta}\rangle\dx.
\end{flalign*}
Recalling that $\eta_{\delta}\rightharpoonup^{*} \mathds{1}_{B_{\rr}}$ in $\mathrm{BV}$, we let $\delta\to 0$ in the above display to derive via Cauchy Schwarz inequality,
and H\"older inequality with conjugate exponents $\left(q/(q-p),q/p\right)$,
\begin{eqnarray}\label{4.12}
  \texttt{H}_{\varepsilon}(x_{0};\rr)&=&\sum_{s=1}^{n}\int_{\partial B_{\rr}}\langle F_{\varepsilon}''(\nabla u_{\varepsilon})\partial_{s}\nabla u_{\varepsilon},
  (\partial_{s}u_{\varepsilon}-(\partial_{s}u_{\varepsilon})_{\partial B_{\rr}})\otimes (x-x_{0})/\rr\rangle\d\mathcal{H}^{n-1}(x)\nonumber \\
  &\le&c\texttt{H}_{\varepsilon}'(x_{0};\rr)^{\frac{1}{2}}\left(\int_{\partial B_{\rr}}\snr{F_{\varepsilon}''(\nabla u_{\varepsilon})}\snr{\nabla u_{\varepsilon}
    -(\nabla u_{\varepsilon})_{\partial B_{\rr}}}^{2}\d\mathcal{H}^{n-1}(x)\right)^{\frac{1}{2}}\nonumber \\
  &\le&c\texttt{H}_{\varepsilon}'(x_{0};\rr)^{\frac{1}{2}}\left(\int_{\partial B_{\rr}}\frac{\snr{F''(\nabla u_{\varepsilon})}}{\ell_{\mu}(\nabla u_{\varepsilon})^{p-2}}
  \ell_{\mu}(\nabla u_{\varepsilon})^{p-2}\snr{\nabla u_{\varepsilon}-(\nabla u_{\varepsilon})_{\partial B_{\rr}}}^{2}\d\mathcal{H}^{n-1}(x)\right)^{\frac{1}{2}}\nonumber \\
  &&+c\gamma_{\varepsilon}^{\frac{1}{2}}\texttt{H}_{\varepsilon}'(x_{0};\rr)^{\frac{1}{2}}\left(\int_{\partial B_{\rr}}\ell_{1}(\nabla u_{\varepsilon})^{q-2}
  \snr{\nabla u_{\varepsilon}-(\nabla u_{\varepsilon})_{\partial B_{\rr}}}^{2}\d\mathcal{H}^{n-1}(x)\right)^{\frac{1}{2}}\nonumber \\
  &\stackrel{\eqref{assf.2}}{\le}&c\texttt{H}_{\varepsilon}'(x_{0};\rr)^{\frac{1}{2}}\left(\int_{\partial B_{\rr}}\left(1+\snr{F'(\nabla u_{\varepsilon})}^{\frac{q-p}{q-1}}\right)
  \ell_{\mu}(\nabla u_{\varepsilon})^{p-2}\snr{\nabla u_{\varepsilon}-(\nabla u_{\varepsilon})_{\partial B_{\rr}}}^{2}\d\mathcal{H}^{n-1}(x)\right)^{\frac{1}{2}}\nonumber \\
  &&+c\gamma_{\varepsilon}^{\frac{1}{2}}\texttt{H}_{\varepsilon}'(x_{0};\rr)^{\frac{1}{2}}\left(\int_{\partial B_{\rr}}\left(1+\snr{\nabla u_{\varepsilon}}^{2}
  +\snr{(\nabla u_{\varepsilon})_{\partial B_{\rr}}}^{2}\right)^{\frac{q-2}{2}}\snr{\nabla u_{\varepsilon}-(\nabla u_{\varepsilon})_{\partial B_{\rr}}}^{2}\d\mathcal{H}^{n-1}(x)\right)^{\frac{1}{2}}\nonumber \\
  &\stackrel{\eqref{Vm}_{1,3}}{\le}&c\texttt{H}_{\varepsilon}'(x_{0};\rr)^{\frac{1}{2}}\texttt{S}_{\varepsilon}'(x_{0};\rr)^{\frac{q-p}{2q}}
  \left(\int_{\partial B_{\rr}}\texttt{V}_{\varepsilon,p}^{\frac{2q}{p}}(x_{0};\rr)\d\mathcal{H}^{n-1}(x)\right)^{\frac{p}{2q}}\nonumber \\
&&+c\gamma_{\varepsilon}^{\frac{1}{2}}\texttt{H}_{\varepsilon}'(x_{0};\rr)^{\frac{1}{2}}\left(\int_{\partial B_{\rr}}\texttt{V}_{\varepsilon,q}^{2}(x_{0};\rr)\d\mathcal{H}^{n-1}(x)\right)^{\frac{1}{2}}\nonumber \\
  &\stackrel{\eqref{Vm}_{1}}{\le}&c\rr^{\frac{p(n-1)}{2q}}\texttt{H}_{\varepsilon}'(x_{0};\rr)^{\frac{1}{2}}\texttt{S}_{\varepsilon}'(x_{0};\rr)^{\frac{q-p}{2q}}
  \left(\mint_{\partial B_{\rr}}\texttt{V}_{\varepsilon,p}^{\frac{2q}{p}}(x_{0};\rr)\d\mathcal{H}^{n-1}(x)\right)^{\frac{p}{2q}}+c\texttt{H}_{\varepsilon}'(x_{0};\rr)^{\frac{1}{2}}
  \texttt{Q}_{\varepsilon}'(x_{0};\rr)^{\frac{1}{2}},
\end{eqnarray}
with $c\equiv c(n,N,L,L_{\mu},p,q)$. We then notice that, up to choosing $2^{*}_{n-1}$ so large that $2^{*}_{n-1}>2q/p$ when $n=2$ or $n=3$,
by \eqref{pq} it is $2q/p<2^{*}_{n-1}$, so we can write
\begin{flalign}\label{4}
\frac{2q}{p}=2\lambda+(1-\lambda)2^{*}_{n-1} \ \Longrightarrow \ \lambda=\frac{p2^{*}_{n-1}-2q}{p(2^{*}_{n-1}-2)},\qquad 1-\lambda=\frac{2(q-p)}{p(2^{*}_{n-1}-2)},
\end{flalign}
as to control via H\"older inequality with conjugate exponents $\left(1/\lambda,1/(1-\lambda)\right)$,
\begin{eqnarray}\label{4.13}
  \nra{\texttt{V}_{\varepsilon,p}(x_{0};\rr)}_{\LL^{\frac{2q}{p}}(\partial B_{\rr})}&\le& \nra{\texttt{V}_{\varepsilon,p}(x_{0};\rr)}_{\LL^{2}(\partial B_{\rr})}^{\frac{p\lambda}{q}}
  \nra{\texttt{V}_{\varepsilon,p}(x_{0};\rr)}_{\LL^{2^{*}_{n-1}}(\partial B_{\rr})}^{\frac{p(1-\lambda)2^{*}_{n-1}}{2q}}\nonumber \\
  &\stackrel{\eqref{spv}}{\le}&c\rr^{\frac{p(1-\lambda)2^{*}_{n-1}}{2q}}\nra{\texttt{V}_{\varepsilon,p}(x_{0};\rr)}_{\LL^{2}(\partial B_{\rr})}^{\frac{p\lambda}{q}}
  \nra{\nabla V_{\mu,p}(\nabla u_{\varepsilon})}_{\LL^{2}(\partial B_{\rr})}^{\frac{p(1-\lambda)2^{*}_{n-1}}{2q}}\nonumber \\
  &\stackrel{\eqref{keymono.1}}{\le}&c\rr^{\frac{p}{2q}\left(1-n+(1-\lambda)\left(n-1-2^{*}_{n-1}(n-3)/2\right)\right)}
  \texttt{P}_{\varepsilon}'(x_{0};\rr)^{\frac{p\lambda}{2q}}\texttt{H}_{\varepsilon}'(x_{0};\rr)^{\frac{p(1-\lambda)2^{*}_{n-1}}{4q}},
\end{eqnarray}
for $c\equiv c(n,N,L,p,q).$ Merging \eqref{4.12} and \eqref{4.13} we get
\begin{flalign}\label{4.14}
  \texttt{H}_{\varepsilon}(x_{0};\rr)\le c\rr^{\beta_{0}}\texttt{S}_{\varepsilon}'(x_{0};\rr)^{\frac{q-p}{2q}}\texttt{P}_{\varepsilon}'(x_{0};\rr)^{\frac{p\lambda}{2q}}
  \texttt{H}_{\varepsilon}'(x_{0};\rr)^{\frac{1}{2}+\frac{p(1-\lambda)2^{*}_{n-1}}{4q}}+c\texttt{H}_{\varepsilon}'(x_{0};\rr)^{\frac{1}{2}}\texttt{Q}_{\varepsilon}'(x_{0};\rr)^{\frac{1}{2}},
\end{flalign}
where we set
$$
\beta_{0}:=\frac{p(1-\lambda)}{2q}\left(n-1-\frac{2^{*}_{n-1}(n-3)}{2}\right),
$$
and it is $c\equiv c(n,N,L,L_{\mu},p,q)$. Next, we fix parameters $r/2\le \tau_{2}<\tau_{1}\le 3r/4<\dist(x_{0},\partial B)$, and, for
$\texttt{K}_{\varepsilon}\in \{\texttt{H}_{\varepsilon},\texttt{P}_{\varepsilon},\texttt{Q}_{\varepsilon},\texttt{S}_{\varepsilon}\}$, set 
\begin{flalign*}
  \mathcal{I}_{\texttt{K}_{\varepsilon}}:=
  \left\{\rr\in (\tau_{2},\tau_{1})\colon \texttt{K}_{\varepsilon}'(x_{0};\rr) \leq \frac{4}{\tau_{1}-\tau_{2}}\int_{\tau_{2}}^{\tau_{1}}\texttt{K}'_{\varepsilon}(x_{0};t)\dt\right\}.
\end{flalign*}
From the definition we get, since $\texttt{K}_{\varepsilon}(x_{0}; \cdot )$ is absolutely continuous and increasing, that
$\mathcal{L}^{1}\left((\tau_{2},\tau_{1})\setminus\mathcal{I}_{\texttt{K}_{\varepsilon}}\right) < (\tau_{1}-\tau_{2})/4$, and consequently
$\mathcal{L}^{1}\left(\mathcal{I}_{\texttt{H}_{\varepsilon}}\cap \mathcal{I}_{\texttt{P}_{\varepsilon}}\cap \mathcal{I}_{\texttt{Q}_{\varepsilon}}\cap \mathcal{I}_{\texttt{S}_{\varepsilon}}\right) >0$.
Therefore we can pick $\rr\in \mathcal{I}_{\texttt{H}_{\varepsilon}}\cap \mathcal{I}_{\texttt{P}_{\varepsilon}}\cap\mathcal{I}_{ \texttt{Q}_{\varepsilon}}\cap \mathcal{I}_{\texttt{S}_{\varepsilon}}$
for which also \eqref{4.14} holds. Hereby
\begin{flalign}\label{sc.4}
\texttt{H}_{\varepsilon}(x_{0};\rr)\le\frac{c\rr^{\beta_{0}}\texttt{S}_{\varepsilon}(x_{0};\tau_{1})^{\frac{q-p}{2q}}
  \texttt{P}_{\varepsilon}(x_{0};\tau_{1})^{\frac{p\lambda}{2q}}\texttt{H}_{\varepsilon}(x_{0};\tau_{1})^{\frac{1}{2}+\frac{p(1-\lambda)2^{*}_{n-1}}{4q}}}{(\tau_{1}-\tau_{2})^{\frac{3q-p}{2q}}}
+\frac{c\texttt{Q}_{\varepsilon}(x_{0};\tau_{1})^{\frac{1}{2}}\texttt{H}_{\varepsilon}(x_{0};\tau_{1})^{\frac{1}{2}}}{(\tau_{1}-\tau_{2})},
\end{flalign}
with $c\equiv c(n,N,L_{\mu},p,q)$. Now we observe that
$$
\eqref{4} \ \Longrightarrow \ \frac{1}{2}+\frac{p(1-\lambda)2^{*}_{n-1}}{4q}<1,
$$
so we can apply Young inequality with conjugate exponents $\left(\frac{2q}{2q-\lambda p},\frac{2q}{\lambda p}\right)$ and $(2,2)$ to have
\begin{eqnarray*}
\texttt{H}_{\varepsilon}(x_{0};\tau_{2})&\le&\texttt{H}_{\varepsilon}(x_{0};\rr)\nonumber \\
&\le&\frac{1}{4}\texttt{H}_{\varepsilon}(x_{0};\tau_{1})+\frac{c\rr^{\alpha_{0}}\texttt{S}_{\varepsilon}(x_{0};\tau_{1})^{\kappa_{2}}
  \texttt{P}_{\varepsilon}(x_{0};\tau_{1})}{(\tau_{1}-\tau_{2})^{\kappa_{1}}}+\frac{c\texttt{Q}_{\varepsilon}(x_{0};\tau_{1})}{(\tau_{1}-\tau_{2})^{2}}\nonumber \\
&\stackrel{\eqref{enes},\eqref{4.2.0}}{\le}&\frac{1}{4}\texttt{H}_{\varepsilon}(x_{0};\tau_{1})+\frac{c\rr^{\alpha_{0}}\left(\F(u_{\varepsilon};B)+
  1+\texttt{o}(\varepsilon)\right)^{\kappa_{2}+1}}{(\tau_{1}-\tau_{2})^{\kappa_{1}}}+\frac{c\texttt{Q}_{\varepsilon}(x_{0};\tau_{1})}{(\tau_{1}-\tau_{2})^{2}}
\end{eqnarray*}
where we set 
\begin{flalign}\label{exx}
\alpha_{0}:=\frac{2q\beta_{0}}{\lambda p},\qquad \qquad \kappa_{1}:=\frac{3q-p}{\lambda p},\qquad \qquad \kappa_{2}:=\frac{q-p}{\lambda p},
\end{flalign}
and it is $c\equiv c(n,N,L_{\mu},p,q)$. Lemma \ref{l5} then yields
\begin{eqnarray}\label{hdes.1}
\int_{B_{r/2}}\snr{\nabla V_{\mu,p}(\nabla u_{\varepsilon})}^{2}+\snr{\nabla V_{1,q'}(F'(\nabla u_{\varepsilon}))}^{2}\dx&\stackrel{\eqref{keymono.1}}{\le}&\texttt{H}_{\varepsilon}(x_{0};r/2)\nonumber \\
&\stackrel{\eqref{assf}_{1},\eqref{approx}_{1,2}}{\le}&cr^{\alpha_{0}-\kappa_{1}}\left(\F(u;B)+1+\texttt{o}(\varepsilon)\right)^{\kappa_{2}+1}+r^{-2}\texttt{o}(\varepsilon),
\end{eqnarray}
for $c\equiv c(n,N,L,L_{\mu},p,q)$. 
Finally, we use $\eqref{approx}_{2,3}$ and $\eqref{assf}_{1}$ to send $\varepsilon\to 0$ in \eqref{hdes.1}
, and fix $B_{r}(x_{0})\equiv B$ to conclude with \eqref{hdes}. A standard covering argument then yields that
$V_{\mu,p}(\nabla u), V_{1,q'}(F'(\nabla u))\in \WW^{1,2}_{\loc}(\Omega,\mathbb{R}^{N\times n})$, and \eqref{hdes.4}$_{1}$ is proven.
Now we only need to prove the validity of $\eqref{hdes.4}_{2}$. To this end, we assume $\mu>0$ in \eqref{assf}$_{3}$, observe that thanks to \eqref{hdes.1} and \eqref{vpvqn}, sequence $\{u_{\varepsilon}\}_{\varepsilon>0}$ of solutions to problem \eqref{pde} is now uniformly bounded also in $W^{2,2}(B_{r/2},\mathbb{R}^{N})$, and observe that by $\eqref{cdq}_{3}$ and \eqref{3.1.0} it is
\begin{eqnarray*}
  \LL^{1}_{\loc}(B)\ni\sum_{s=1}^{n}\langle\partial_{s} F'(\nabla u_{\varepsilon}),\partial_{s}\nabla u_{\varepsilon}\rangle&=&
  \sum_{s=1}^{n}\left\langle\frac{ F''(\nabla u_{\varepsilon})}{\ell_{\mu}(\nabla u_{\varepsilon})^{p-2}}
  \ell_{\mu}(\nabla u_{\varepsilon})^{\frac{p-2}{2}}\partial_{s}\nabla u_{\varepsilon},\ell_{\mu}(\nabla u_{\varepsilon})^{\frac{p-2}{2}}\partial_{s}\nabla u_{\varepsilon}\right\rangle\nonumber \\
&=:&\sum_{s=1}^{n}\mathcal{D}(G''(\nabla u_{\varepsilon}),V_{\varepsilon}^{s}),
\end{eqnarray*}
where we set $\bigodot(\mathbb{R}^{N\times n})\times \mathbb{R}^{N\times n}\ni (A,z)\mapsto \mathcal{D}(A,z):=\langle Az,z\rangle$,
$\mathbb{R}^{N\times n}\ni z\mapsto G''(z):=F''(z)\ell_{\mu}(z)^{2-p}$, and $V_{\varepsilon}^{s}:=\ell_{\mu}(\nabla u_{\varepsilon})^{\frac{p-2}{2}}\partial_{s}\nabla u_{\varepsilon}$. Function $\mathcal{D}$ is continuous by definition, and, if the first argument is restricted to positive semi-definite forms, it is also nonnegative.
Notice that by \eqref{assf}$_{1}$ and \eqref{approx}$_{3}$, it is $G''(\nabla u_{\varepsilon})\to F''(\nabla u)\ell_{\mu}(\nabla u)^{2-p}$ in $\mathcal{L}^{n}$-measure.
Moreover, observing that
$$
\partial_{s}V_{\mu,p}(\nabla u_{\varepsilon})-\left(\frac{p-2}{4}\right)\ell_{\mu}(\nabla u_{\varepsilon})^{\frac{p-6}{2}}\partial_{s}
\snr{\nabla u_{\varepsilon}}^{2}\nabla u_{\varepsilon}=\ell_{\mu}(\nabla u_{\varepsilon})^{\frac{p-2}{2}}\partial_{s}\nabla u_{\varepsilon},
$$
being $p\ge 2$, for each $s\in \{1,\cdots,n\}$ by \eqref{hdes.1} we have $V_{\varepsilon}^{s}\rightharpoonup \ell_{\mu}(\nabla u)^{\frac{p-2}{2}}\partial_{s}\nabla u$
weakly in $\LL^{2}(B_{r/2},\mathbb{R}^{N\times n})$. By \cite[Theorem 4.1 and Remark 4.2]{Rindler}, up to nonrelabelled subsequences, each of the sequences
$\{G''(\nabla u_{\varepsilon}),V_{\varepsilon}^{s}\}_{\varepsilon>0}$ generates a Young measure and, via \cite[Lemma 5.19]{Rindler}, it is
$\{G''(\nabla u_{\varepsilon}),V_{\varepsilon}^{s}\}_{\varepsilon>0}\stackrel{\mathbf{Y}}{\to}\{\nu_{x}^{s}\}$ where $\nu_{x}^{s}=\delta_{G''(\nabla u)}\otimes \kappa^{s}_{x}$,
and $\{\kappa^{s}_{x}\}$ is a Young measure generated by $\{V_{\varepsilon}^{s}\}_{\varepsilon>0}$. Therefore
\begin{flalign*}
\mint_{B_{r/2}}\int_{\mathbb{R}^{N\times n}}\snr{z}^{2}\d\kappa^{s}_{x}(z)\dx<\infty\qquad \mbox{and}\qquad \int_{\mathbb{R}^{N\times n}}z\d\kappa^{s}_{x}(z)=\ell_{\mu}(\nabla u)^{\frac{p-2}{2}}\partial_{s}\nabla u.
\end{flalign*}
By weak lower semicontinuity for Young measures \cite[Proposition 4.6]{Rindler}, we have
\begin{eqnarray}\label{ym.0}
\sum_{s=1}^{n}\mint_{B_{r/2}}\langle \nu^{s}_{x},\mathcal{D}(\cdot,\cdot)\rangle\dx&\le&\liminf_{\varepsilon\to 0}\sum_{s=1}^{n}\mint_{B_{r/2}}\mathcal{D}(G''(\nabla u_{\varepsilon}),V_{\varepsilon}^{s})\dx\nonumber \\
&\le&\sup_{\varepsilon>0}\sum_{s=1}^{n}\mint_{B_{r/2}}\langle \partial_{s}F'(\nabla u_{\varepsilon}),\partial_{s}\nabla u_{\varepsilon}\rangle\dx\nonumber \\
&\le&\sup_{\varepsilon>0}\sum_{s=1}^{n}\mint_{B_{r/2}}\langle F''_{\varepsilon}(\nabla u_{\varepsilon})\partial_{s}\nabla u_{\varepsilon},\partial_{s}\nabla u_{\varepsilon}\rangle\dx\stackrel{\eqref{hdes.1}}{<}\infty,
\end{eqnarray}
so we can apply Jensen inequality \cite[Lemma 5.11]{Rindler} to derive
\begin{eqnarray*}
\sum_{s=1}^{n}\mint_{B_{r/2}}\langle\partial_{s}F'(\nabla u),\partial_{s}\nabla u\rangle\dx&=&\sum_{s=1}^{n}\mint_{B_{r/2}}\langle F''(\nabla u)\partial_{s}\nabla u,\partial_{s}\nabla u\rangle\dx\nonumber \\
&=&\sum_{s=1}^{n}\mint_{B_{r/2}}\langle G''(\nabla u)\ell_{\mu}(\nabla u)^{\frac{p-2}{2}}\partial_{s}\nabla u,\ell_{\mu}(\nabla u)^{\frac{p-2}{2}}\partial_{s}\nabla u\rangle\dx\nonumber \\
&\le&\sum_{s=1}^{n}\mint_{B_{r/2}}\langle \nu_{x}^{s},\mathcal{D}(\cdot,\cdot)\rangle\dx\stackrel{\eqref{ym.0}}{<}\infty.
\end{eqnarray*}
A standard covering argument then gives $\eqref{hdes.4}_{2}$ and completes the proof.
\begin{remark}\label{rem}
\emph{Theorem \ref{hd} is valid} a fortiori \emph{if the integrand governing $\F$ is of the form $\mathbb{R}^{N\times n}\ni z\mapsto \ti{F}(z):=F(z)+a_{0}(1+\snr{z}^{2})^{q/2}$
with $F$ as in \eqref{assf} and some constant $a_{0}\ge 0$. In this case estimate \eqref{hdes} includes also the informations coming from the $q$-elliptic term, i.e.:}
\begin{flalign*}
\nra{\nabla V_{\mu,p}(\nabla u)}_{\LL^{2}(B/2)}+\nra{\nabla V_{1,q}(F'(\nabla u))}_{\LL^{2}(B/2)}+a_{0}^{\frac{1}{2}}\nra{\nabla V_{1,q}(\nabla u)}_{\LL^{2}(B/2)}\le c\nra{(\ti{F}(\nabla u)+1)}_{\LL^{1}(B)}^{\tx{b}},
\end{flalign*}
\emph{with $c\equiv c(n,N,L,L_{\mu},p,q)$.}
\end{remark}

\subsection{Proof of Corollary \ref{cor}}
The proof of statements (\emph{ii}.)-(\emph{iii}.) comes as a direct consequence of Sobolev-Morrey embedding theorem and Theorem \ref{hd}. Concerning the result in (\emph{i}.),
Theorem \ref{hd} and Sobolev embedding theorem yields that $\nabla u\in \LL^{\frac{np}{n-2}}_{\loc}(\Omega,\mathbb{R}^{N\times n})$ and, if $p>n-2$, it is $np/(n-2)>n$,
thus Sobolev-Morrey embedding theorem allows to complete the proof. 
\subsubsection{Proof of Theorem \ref{polyt}, \emph{(}i.\emph{)}}\label{poly11} Proposition \ref{poly} and Corollary \ref{poly.er} yield that Corollary \ref{cor} applies
to the class of convex polynomials in the statement of Theorem \ref{polyt}, and in three space dimensions $n=3$ the local H\"older continuity result in (\emph{i}.)
immediately follows with no upper restriction on the polynomial degree. Part (\emph{ii}.) of Theorem \ref{polyt} will be proven in Section \ref{iiii} below.

\section{Smoothness in two space dimensions}\label{2d}
In this section we offer two independent proofs of gradient boundedness in $2d$ for vector-valued minimizers of functional $\F$,
based on an inhomogeneous monotonicity formula, and on a renormalized Gehring - Giaquinta \& Modica type lemma. From this, we deduce full regularity in $2d$.
As already mentioned in Section \ref{tec}, our approach allows handling simultaneously degenerate and nondegenerate problems and entails new results already
in the genuine $(p,q)$-setting, see Remark \ref{rempq} below. For the sake of clarity, we shall divide the reminder of the section in three parts,
beginning by detailing the two arguments leading to gradient boundedness and concluding by proving the full regularity statement, leading to the proof of Theorem \ref{2dreg}.
\subsection{Proof of Theorem \ref{2dreg}: \texorpdfstring{$2d$}{}-gradient boundedness via monotonicity formula}\label{2dm}
We look back at Section \ref{ptt}, and exploit the low dimension - if $n=2$ we are allowed to choose $2^{*}_{1}$ arbitrarily large - to study the asymptotics of the exponents in \eqref{exx}. We have
\begin{flalign*}
&\alpha_{0}-\kappa_{1}=-\frac{2q}{p}\left(1-\frac{2q}{p2^{*}_{1}}\right)^{-1}+\frac{4(2q-p)}{p2^{*}_{1}-2q}\nearrow -\frac{2q}{p}\qquad \mbox{as} \ \ 2^{*}_{1}\to \infty,\nonumber \\
&\kappa_{2}+1=\frac{q}{p}\left(1-\frac{2q}{p2^{*}_{1}}\right)^{-1}-\frac{2(2q-p)}{p2^{*}_{1}-2q}\searrow \frac{q}{p}\qquad \mbox{as} \ \ 2^{*}_{1}\to \infty,
\end{flalign*}
so we can increase $2^{*}_{1}$ in such a way that $\alpha_{0}-\kappa_{1}\geq-4q/p$ and $\kappa_{2}+1\leq 2q/p$, and be more precise on estimate \eqref{hdes.1}, that now becomes
\eqn{hdes.5}
$$
\nr{\nabla V_{\mu,p}(\nabla u_{\varepsilon})}_{\LL^{2}(B/2)}^{2}+\nr{\nabla V_{1,q'}(F'(\nabla u_{\varepsilon}))}_{\LL^{2}(B/2)}^{2}\le c\tx{H}_{\varepsilon}(x_{B};\rrr/2)\le c\nra{F(\nabla u)+1}_{\LL^{1}(B)}^{\frac{2q}{p}},
$$
for $c\equiv c(N,L,L_{\mu},p,q)$. By \eqref{4.2.0}, \eqref{hdes.5}, and Trudinger inequality \cite[Theorem 2]{tr},
there exists a parameter $\tx{d}\approx \nra{F(\nabla u)+1}_{\LL^{1}(B)}^{-2q/p}$, with constants implicit in "$\approx$" depending on $(N,L,L_{\mu},p,q)$ such that
\eqn{mt}
$$
\mint_{B/2}\exp\left\{\tx{d}\snr{V_{1,q'}(F'(\nabla u_{\varepsilon}))}^{2}\right\}\dx\lesssim 1,
$$
again up to constants depending on $(N,L,L_{\mu},p,q)$ - in fact, an inspection of the proof of \cite[Theorem 2]{tr} shows that thanks to
\eqref{hdes.5} and \eqref{4.2.0} the constant $\tx{d}$, in principle proportional to the quantity
${\nr{\nabla V_{1,q'}(F'(\nabla u_{\varepsilon}))}_{\LL^{2}(B/2)}^{2}+\nra{V_{1,q'}(F'(\nabla u_{\varepsilon}))}_{\LL^{2}(B/2)}^{2}}$ can be suitably enlarged
to replace any dependency on $\varepsilon$ with a dependency on $(N,L,L_{\mu},p,q)$. In particular, whenever $B_{\rr}\subseteq B/2$ we have by Jensen's inequality,
\begin{eqnarray}\label{log}
\nra{V_{1,q'}(F'(\nabla u_{\varepsilon}))}_{\LL^{2}(B_{\rr})}^{2}&\le&\tx{d}^{-1}\log\left(\mint_{B_{\rr}}\exp\left\{\tx{d}\snr{V_{1,q'}(F'(\nabla u_{\varepsilon}))}^{2}\right\}\dx\right)\nonumber \\
&\le&c\tx{d}^{-1}\log\left(\left(\frac{\rrr}{\rr}\right)^{2}\mint_{B/2}\exp\left\{\tx{d}\snr{V_{1,q'}(F'(\nabla u_{\varepsilon}))}^{2}\right\}\dx\right)\nonumber \\
&\stackrel{\eqref{mt}}{\le}&c\log\left(\frac{\rrr}{\rr}\right)\nra{F(\nabla u)+1}_{\LL^{1}(B)}^{\frac{2q}{p}},
\end{eqnarray}
for $c\equiv c(N,L,L_{\mu},p,q)$. Next, taking $x_{0}\in B/4$, and $\rr\in (0,\rrr/6)$, we fix a ball $B_{\rr}(x_{0})\Subset B/2$ and refine estimate \eqref{4.12} as 
\begin{eqnarray}\label{log.1}
  \tx{H}_{\varepsilon}(x_{0};\rr)&\stackrel{\eqref{spv}}{\le}&c\rr\tx{H}_{\varepsilon}'(x_{0};\rr)^{\frac{1}{2}}
  \left(\mint_{\partial B_{\rr}}\ell_{1}(F'(\nabla u_{\varepsilon}))^{q'}\d\mathcal{H}^{1}(x)\right)^{\frac{q-p}{2q}}
  \left(\int_{\partial B_{\rr}}\snr{\nabla V_{\mu,p}(\nabla u_{\varepsilon})}^{2}\d\mathcal{H}^{1}(x)\right)^{\frac{1}{2}}\nonumber \\
&&+c\rr\tx{H}_{\varepsilon}'(x_{0};\rr)^{\frac{1}{2}}\left(\int_{\partial B_{\rr}}\gamma_{\varepsilon}\snr{\nabla V_{1,q}(\nabla u_{\varepsilon})}^{2}\d\mathcal{H}^{1}(x)\right)^{\frac{1}{2}}\nonumber \\
  &\stackrel{\eqref{keymono.1}}{\le}&c\rr \tx{H}_{\varepsilon}'(x_{0};\rr)
  \left(1+\nra{\ell_{1}(F'(\nabla u_{\varepsilon}))}_{\LL^{q'}(\partial B_{\rr})}^{\frac{q-p}{2(q-1)}}\right)\stackrel{\eqref{Vm}_{1}}{\le}
  c\rr \tx{H}_{\varepsilon}'(x_{0};\rr)\left(1+\nra{V_{1,q'}(F'(\nabla u_{\varepsilon}))}_{\LL^{2}(\partial B_{\rr})}^{\frac{q-p}{q}}\right),
\end{eqnarray}
where we also used H\"older inequality with conjugate exponents $\left(q/(q-p),q/p\right)$, and it is $c\equiv c(N,L,L_{\mu},p,q)$. By the trace theorem \cite[Section 2.3]{gk}, we control
\begin{eqnarray*}
\nra{V_{1,q'}(F'(\nabla u_{\varepsilon}))}_{\LL^{2}(\partial B_{\rr})}&\le&c\nra{V_{1,q}(F'(\nabla u_{\varepsilon}))}_{\LL^{2}(B_{\rr})}+c\nr{\nabla V_{1,q'}(F'(\nabla u_{\varepsilon}))}_{\LL^{2}(B_{\rr})}\nonumber \\
&\stackrel{\eqref{hdes.5}}{\le}&c\nra{V_{1,q}(F'(\nabla u_{\varepsilon}))}_{\LL^{2}(B_{\rr})}+c\nra{F(\nabla u)+1}_{\LL^{1}(B)}^{\frac{q}{p}}\nonumber \\
&\stackrel{\eqref{log}}{\le}&c\log\left(\frac{\rrr}{\rr}\right)^{\frac{1}{2}}\nra{F(\nabla u)+1}_{\LL^{1}(B)}^{\frac{q}{p}},
\end{eqnarray*}
for $c\equiv c(N,L,L_{\mu},p,q)$, therefore we can update \eqref{log.1} to
\begin{flalign}\label{log.2}
\tx{H}_{\varepsilon}(x_{0};\rr)\le c\rr\log\left(\frac{\rrr}{\rr}\right)^{\frac{q-p}{2q}}\nra{F(\nabla u)+1}_{\LL^{1}(B)}^{\frac{q-p}{p}}\tx{H}_{\varepsilon}'(x_{0};\rr),
\end{flalign}
with $c\equiv c(N,L,L_{\mu},p,q)$. We then let $\tx{B}:=c\nra{F(\nabla u)+1}_{\LL^{1}(B)}^{\frac{q-p}{p}}$, pick any $\sigma\in (0,\rrr/6^{4})$,
and integrate the inequality in \eqref{log.2} on $\rr\in (\sigma,\rrr/6)$ to derive
\begin{eqnarray*}
  \frac{q}{(q+p)\tx{B}}\log\left(\frac{\rrr}{\sigma}\right)^{\frac{q+p}{2q}}&\le&
  \frac{2q}{(q+p)\tx{B}}\left(\log\left(\frac{\rrr}{\sigma}\right)^{\frac{q+p}{2q}}-\log(6)^{\frac{q+p}{2q}}\right)\nonumber \\
  &\le&\tx{B}^{-1}\int_{\sigma}^{\rrr/6}\frac{1}{\rr\log\left(\frac{\rrr}{\rr}\right)^{\frac{q-p}{2q}}}\d\rr\stackrel{\eqref{log.2}}{\le}
  \log\left(\frac{\tx{H}_{\varepsilon}(x_{0};\rrr/6)}{\tx{H}_{\varepsilon}(x_{0};\sigma)}\right).
\end{eqnarray*}
Next, we set $\beta:=\left(q\tx{B}^{-1}/(q+p)\right)^{2q/(q+p)}$, $\gamma:=(q+p)/(2q)$, and pass to the exponentials in the previous display to recover
\begin{eqnarray*}
  c\nr{\nabla V_{\mu,p}(\nabla u_{\varepsilon})}_{\LL^{2}(B_{\sigma})}^{2}+c\nr{\nabla V_{1,q'}(F'(\nabla u_{\varepsilon}))}_{\LL^{2}(B_{\sigma})}^{2}&\stackrel{\eqref{keymono.1}}{\le}&
  \tx{H}_{\varepsilon}(x_{0};\sigma)\nonumber \\
&\le&\exp\left\{-\left(\beta\log\left(\rrr/\sigma\right)\right)^{\gamma}\right\} \tx{H}_{\varepsilon}(x_{0};\rrr/6)\nonumber \\
&\le&\exp\left\{-\left(\beta\log\left(\rrr/\sigma\right)\right)^{\gamma}\right\} \tx{H}_{\varepsilon}(x_{B};\rrr/2)\nonumber \\
&\stackrel{\eqref{hdes.5}}{\le}&c\exp\left\{-\left(\beta\log\left(\rrr/\sigma\right)\right)^{\gamma}\right\}\nra{F(\nabla u)+1}_{\LL^{1}(B)}^{\frac{2q}{p}},
\end{eqnarray*}
for $c\equiv c(N,L,L_{\mu},p,q)$. We then let $\varepsilon\to 0$ above and use \eqref{approx}$_{3}$ to secure
\begin{eqnarray}\label{log.3}
  \nr{\nabla V_{\mu,p}(\nabla u)}_{\LL^{2}(B_{\sigma})}^{2}+\nr{\nabla V_{1,q'}(F'(\nabla u))}_{\LL^{2}(B_{\sigma})}^{2}&\le&
  c\exp\left\{-\left(\beta\log\left(\rrr/\sigma\right)\right)^{\gamma}\right\}\nra{F(\nabla u)+1}_{\LL^{1}(B)}^{\frac{2q}{p}}\nonumber \\
&\le&c\log\left(\frac{\rrr}{\sigma}\right)^{-(\gamma+2)}\nra{F(\nabla u)+1}_{\LL^{1}(B)}^{\tx{b}},
\end{eqnarray}
where we used that $(\beta\log(\rrr/\sigma))^{\gamma+2}\exp\left\{-(\beta\log(\rrr/\sigma))^{\gamma}\right\}\to 0$ as $\sigma\to 0$, and it is
$\tx{b}:=(\gamma+2)\left(\frac{q-p}{p}\right)+\frac{2q}{p}$, $c\equiv c(N,L,L_{\mu},p,q)$. Since \eqref{log.3} holds true for all balls
$B_{\sigma}(x_{0})\Subset B/2$, $x_{0}\in B/4$, from a variant of Morrey lemma \cite[Lemma 1.1]{fr1}, see also \cite[page 287]{fr},
we deduce that $V_{\mu,p}(\nabla u)$, $V_{1,q'}(F'(\nabla u))$ are continuous on $B/8$ and therefore bounded, and a standard covering argument
eventually yields that $F'(\nabla u),\ \nabla u\in \LL^{\infty}_{\loc}(\Omega,\mathbb{R}^{N\times 2})$. 

\subsection{Proof of Theorem \ref{2dreg}: \texorpdfstring{$2d$}{}-gradient boundedness via renormalized Gehring - Giaquinta \& Modica lemma} \label{2dg}
Our first move is a quantitative version of classical Gehring lemma \cite{fwg}, after Giaquinta \& Modica \cite{gm1}.
\begin{lemma}\label{lglg}
Assume that $\varphi$ is a nonnegative, decreasing function on $[t_{0},\infty)$, infinitesimal as $t\to \infty$, and satisfying
\eqn{6.2.3}
$$
-\int_{t}^{\infty}s^{1-\mf{m}}\d\varphi(s)\le c_{0}t^{1-\mf{m}}\tx{M}\varphi(t),
$$
for some absolute constants $c_{0}\ge 1$, $\tx{M}\ge 1$, $\tx{m}\in (0,1)$ and all $t\ge t_{0}$. There exists a positive number 
\eqn{rrr}
$$
\mf{t}:=\frac{2c_{0}\tx{M}-\mf{m}}{2c_{0}\tx{M}-1}\in (1,2)\qquad \mbox{verifying} \ \ \mf{t}\searrow 1 \ \ \mbox{as} \ \ \tx{M}\to \infty,
$$
such that
\eqn{6.2.7}
$$
-\int_{t_{0}}^{\infty}s^{\mf{t}-\mf{m}}\d\varphi(s)\le -2t_{0}^{\mf{t}-1}\int_{t_{0}}^{\infty}s^{1-\mf{m}}\d\varphi(s).
$$
\end{lemma}
\begin{proof}
We assume for the moment that there is $\kappa\ge 1$ such that $\varphi(t)=0$ whenever $t\ge \kappa$. We will remove this restriction in the end. With $\tx{d}>0$ we set
\begin{flalign*}
\mathcal{I}_{\tx{d}}(t):=-\int_{t}^{\kappa}s^{\tx{d}}\d\varphi(s)\qquad \mbox{and}\qquad \mathcal{I}_{\tx{d}}:=\mathcal{I}_{\tx{d}}(t_{0}).
\end{flalign*}
For some $\mf{t}>1>\mf{m}$ to be determined, integrating by parts we have
\begin{eqnarray*}
\mathcal{I}_{\mf{t}-\mf{m}}&=&-\int_{t_{0}}^{\kappa}s^{\mf{t}-1}s^{1-\mf{m}}\d\varphi(s)=-\int_{t_{0}}^{\kappa}s^{\mf{t}-1}\d\mathcal{I}_{1-\mf{m}}(s)\nonumber \\
&=&t_{0}^{\mf{t}-1}\mathcal{I}_{1-\mf{m}}+(\mf{t}-1)\int_{t_{0}}^{\kappa}s^{\mf{t}-2}\mathcal{I}_{1-\mf{m}}(s)\ds\nonumber \\
&\stackrel{\eqref{6.2.3}}{\le}&t_{0}^{\mf{t}-1}\mathcal{I}_{1-\mf{m}}+c_{0}(\mf{t}-1)\tx{M}\int_{t_{0}}^{\kappa}s^{\mf{t}-1-\mf{m}}\varphi(s)\ds\nonumber \\
&=&t_{0}^{\mf{t}-1}\mathcal{I}_{1-\mf{m}}+\frac{c_{0}(\mf{t}-1)\tx{M}}{\mf{t}-\mf{m}}\left(\mathcal{I}_{\mf{t}-\mf{m}}-t_{0}^{\mf{t}-\mf{m}}\varphi(t_{0})\right)\le
t_{0}^{\mf{t}-1}\mathcal{I}_{1-\mf{m}}+\frac{c_{0}(\mf{t}-1)\tx{M}\mathcal{I}_{\mf{t}-\mf{m}}}{(\mf{t}-\mf{m})}.
\end{eqnarray*}
We now fix $\mf{t}>1$ so close to one that
$$
\frac{c_{0}(\mf{t}-1)\tx{M}}{\mf{t}-\mf{m}}=\frac{1}{2} \ \Longleftrightarrow \ \mf{t}=\frac{2c_{0}\tx{M}-\mf{m}}{2c_{0}\tx{M}-1}\stackrel{\tx{m}<1}{>}1,
$$
and \eqref{rrr} is satisfied. The above choice of $\mf{t}$ allows us to conclude that
\begin{flalign}
\mathcal{I}_{\mf{t}-\mf{m}}\le 2t_{0}^{\mf{t}-1}\mathcal{I}_{1-\mf{m}} \ \stackrel{\varphi(t)\equiv 0 \ \ \tiny{\mbox{if}} \ \ t\ge \kappa}{\Longrightarrow} \ \eqref{6.2.7},\label{6.2.5}
\end{flalign}
and we are done in the case $\varphi(t)\equiv 0$ for $t\ge \kappa$. Let us take care of the general case. Observe that being $\varphi$ nonincreasing on $[t_{0},\infty)$, we have for any given $T\ge\kappa$:
$$
-\int_{\kappa}^{T}s^{1-\mf{m}}\d\varphi(s)\ge -\kappa^{1-\mf{m}}\int_{\kappa}^{T}\d\varphi(s)=\kappa^{1-\mf{m}}\left(\varphi(\kappa)-\varphi(T)\right),
$$
so letting $T\to \infty$ above we get
\eqn{6.2.4}
$$
-\int_{\kappa}^{\infty}s^{1-\mf{m}}\d\varphi(s)\stackrel{\varphi(T)\to 0}{\ge}  \kappa^{1-\mf{m}}\varphi(\kappa).
$$
Next, we set $\varphi_{\kappa}(t):=\varphi(t)$ if $t\le \kappa$, and $\varphi_{\kappa}(t):=0$ if $t>\kappa$. For $t\le \kappa$ we have
\begin{eqnarray}\label{6.2.6}
-\int_{t}^{\infty}s^{1-\mf{m}}\d\varphi_{\kappa}(s)&=&-\int_{t}^{\kappa}s^{1-\mf{m}}\d\varphi(s)-\int_{\kappa}^{\infty}s^{1-\mf{m}}\d\varphi_{\kappa}(s)\nonumber \\
&\le&-\int_{t}^{\kappa}s^{1-\mf{m}}\d\varphi(s)+\kappa^{1-\mf{m}}\varphi(\kappa)\nonumber \\
&\stackrel{\eqref{6.2.4}}{\le}&-\int_{t}^{\kappa}s^{1-\mf{m}}\d\varphi(s)-\int_{\kappa}^{\infty}s^{1-\mf{m}}\d\varphi(s)
=-\int_{t}^{\infty}s^{1-\mf{m}}\d\varphi(s),
\end{eqnarray}
while if $t>\kappa$ the above relation is trivially true, given that $\varphi_{\kappa}(t)\equiv 0$ if $t>\kappa$. Thanks to what we just proved, we have, 
\begin{eqnarray}
-\int_{t_{0}}^{\infty}s^{\mf{t}-\mf{m}}\d\varphi_{\kappa}(s)&=&-\int_{t_{0}}^{\kappa}s^{\mf{t}-\mf{m}}\d\varphi_{\kappa}(s)\stackrel{\eqref{6.2.5}}{\le}-2t_{0}^{\mf{t}-1}\int_{t_{0}}^{\kappa}s^{1-\mf{m}}\d\varphi_{\kappa}(s)\nonumber \\
&\le&-2t_{0}^{\mf{t}-1}\int_{t_{0}}^{\infty}s^{1-\mf{m}}\d\varphi_{\kappa}(s)\stackrel{\eqref{6.2.6}}{\le}-2t_{0}^{\mf{t}-1}\int_{t_{0}}^{\infty}s^{1-\mf{m}}\d\varphi(s).
\end{eqnarray}
The conclusion now follows sending $\kappa\to \infty$.
\end{proof}
The previous lemma is instrumental to establish a self-improving property for functions satisfying suitable reverse H\"older inequalities.
\begin{lemma}\label{revh}
Let $Q_{1}(0)\subset \mathbb{R}^{n}$ be the unitary cube centered at the origin and $v\in \LL^{1}(Q_{1}(0))$ be a nonnegative function verifying
\eqn{6.2.1}
$$
\mint_{Q_{\rr/2}(x_{0})}v\dx\le \hat{c}\tx{M}\left(\mint_{Q_{\rr}(x_{0})}v^{\tx{m}}\dx\right)^{1/\tx{m}},
$$
for all cubes $Q_{\rr}(x_{0})\Subset Q_{1}(0)$, some exponent $\tx{m}\in (0,1)$, and absolute constants $\hat{c},\tx{M}\ge 1$. There exists a constant $c_{*}\equiv c_{*}(n,\hat{c})\ge 1$, and a positive number $\tx{t}\in (1,2)$ as in \eqref{rrr}, i.e.:
\eqn{rrr.1}
$$
\tx{t}=\frac{2c_{*}\tx{M}-\tx{m}}{2c_{*}\tx{M}-1}\in (1,2)\qquad \mbox{satisfying} \ \ \tx{t}\searrow 1 \ \ \mbox{as} \ \ \tx{M}\to \infty,
$$
such that
$$
\left(\mint_{Q_{1/2}(0)}v^{\tx{t}}\dx\right)^{1/\tx{t}}\le 2^{4n+4}\mint_{Q_{1}(0)}v\dx.
$$
In particular, for any fixed threshold $\tx{s}_{0}\ge 1$, the constant $c_{*}$ in \eqref{rrr.1} can be taken so large that
\eqn{*}
$$c_{*}>\tx{s}_{0},$$
up to update dependency $c_{*}\equiv c_{*}(n,\hat{c},\tx{s}_{0})$.
\end{lemma}

\begin{proof}
We ask the reader to have \cite[Section 6.4]{giu} at hand as we shall carefully track the dependency of the constants appearing in the various arguments leading to the proof of \cite[Theorem 6.6]{giu}. To be as close as possible to the setting of \cite[Section 6.4]{giu}, we let $\tx{d}(x):=\dist(x,\partial Q_{1})$, introduce function $\tx{F}(x):=\tx{d}(x)^{n}v(x)$, and denoting by $\ti{P}\Subset Q_{1}$ a cube and $P$ the cube concentric to $\ti{P}$ with half side, and rearrange \eqref{6.2.1} in terms of $\tx{F}$ as 
\eqn{6.2.2}
$$
\nra{\tx{F}}_{\LL^{1}(P)}\le c\tx{M}\nra{\tx{F}^{\tx{m}}}_{\LL^{1}(\ti{P})}^{\frac{1}{\tx{m}}},
$$
holding for all cubes $\ti{P}\Subset Q_{1}$, with $c\equiv c(n,\hat{c})$, that is \cite[inequality (6.47)]{giu}. We next look into \cite[Lemma 6.2]{giu}, that relies on a Calder\'on-Zygmund type argument combined with integral estimates on level sets. Introducing the superlevel set $\Phi_{t}:=\left\{x\in Q_{1}\colon \tx{F}(x)>t\right\}$ for all 
\eqn{t0}
$$
t\ge t_{0}:=\nra{v}_{\LL^{1}(Q_{1})}
$$
and carefully tracking the occurrences of constant $\tx{M}$ along the proof of \cite[Lemma 6.2]{giu} (fix $\lambda\approx \tx{M}$ up to dimensional constants there) we end up with
\eqn{6.2.2.2}
$$
\int_{\Phi_{t}}\tx{F}\dx\le c_{*}t^{1-\mf{m}}\tx{M}\int_{\Phi_{t}}\tx{F}^{\mf{m}}\dx,
$$
for $c_{*}\equiv c_{*}(n,\hat{c})$. Notice that there is no loss of generality in arbitrarily enlarging the value of $c_{*}$, and in particular to take it larger than the assigned $\tx{s}_{0}$, thus fixing dependencies $c_{*}\equiv c_{*}(n,\hat{c},\tx{s}_{0})$. In the light of \cite[Lemma 6.3]{giu}, inequality \eqref{6.2.2.2} can be rewritten as
\eqn{6.2.3.1}
$$
-\int_{t}^{\infty}s^{1-\mf{m}}\d\varphi(s)\le c_{*}t^{1-\mf{m}}\tx{M}\varphi(t)\qquad \mbox{for all} \ \ t\ge t_{0},
$$
where we set $\varphi(t):=\int_{\Phi_{t}}\tx{F}^{\mf{m}}\dx$. Notice that by definition, $\varphi$ is nonincreasing on $[t_{0},\infty)$ and, since $\tx{F}\in \LL^{1}(Q_{1})$, we see that $\varphi(t)\to 0$ as $t\to \infty$, so the assumptions of Lemma \ref{lglg} are verified - in particular, \eqref{6.2.3.1} is exactly \eqref{6.2.3} with $c_{0}\equiv c_{*}$, and $t_{0}$ as defined in \eqref{t0}, thus \eqref{6.2.7} holds true with exponent $\mf{t}$ resulting from \eqref{rrr} that is the one in \eqref{rrr.1}, and, via \cite[Lemma 6.3]{giu} we can write
$$
\int_{\Phi_{t_{0}}}\tx{F}^{\mf{t}}\dx\le 2t_{0}^{\mf{t}-1}\int_{\Phi_{t_{0}}}\tx{F}\dx.
$$
On the other hand, on $Q_{1}\setminus \Phi_{t_{0}}$ we have
$$
\int_{Q_{1}\setminus \Phi_{t_{0}}}\tx{F}^{\mf{t}}\dx\le t_{0}^{\mf{t}-1}\int_{Q_{1}\setminus \Phi_{t_{0}}}\tx{F}\dx,
$$
therefore we can conclude with
$$
\int_{Q_{1}}\tx{F}^{\mf{t}}\dx\le 4t_{0}^{\mf{t}-1}\int_{Q_{1}}\tx{F}\dx,
$$
and, coming back to the function $v$ we eventually get
$$
\int_{Q_{1/2}}v^{\mf{t}}\dx\le 2^{2n+2}t_{0}^{\mf{t}-1}\int_{Q_{1}}v\dx=2^{4n+2}\snr{Q_{1}}\left(\mint_{Q_{1}}v\dx\right)^{\mf{t}},
$$
where we used \eqref{t0} and $\eqref{rrr}_{1}$ to remove any dependency of the bounding constants on $\mf{t}$. The proof is complete.
\end{proof}
We return to Dirichlet problem \eqref{pde}, and observe that conditions $\eqref{cdq}_{1,2,3}$ guarantee that \cite[Theorem V]{c2} applies and
\eqn{5.1.0}
$$
u_{\varepsilon}\in \WW^{1,\infty}_{\loc}(B,\mathbb{R}^{N})\cap \WW^{2,2}_{\loc}(B,\mathbb{R}^{N}).
$$
We then let $Q\equiv Q_{r_{0}}(x_{Q})\Subset B/16$ be any cube with half-side length $r_{0}\in (0,1]$, scale the whole problem on $Q_{1}$ by letting
$v_{\varepsilon}(x):=u_{\varepsilon}(x_{Q}+r_{0}x)/r_{0}$, $v(x):=u(x_{Q}+r_{0}x)/r_{0}$, observe that the amount of regularity in \eqref{5.1.0} is obviously
preserved by 
the blown up map $v_{\varepsilon}$ so that the definition
\eqn{mmm}
$$
\tx{M}_{\varepsilon}\ge \max\left\{\nr{F^{\prime}(\nabla u_{\varepsilon})}_{\LL^{\infty}(Q)}^{\frac{q-p}{q-1}},1\right\},
$$
makes sense, and recall that system \eqref{elsd} holds \emph{verbatim} for the scaled maps $v_{\varepsilon}$ on $Q_{1}$.
We then let $Q_{\rr}(x_{0})\Subset Q_{1}$ be any cube with half side length $\rr\in (0,1]$, $\eta\in \CC^{1}_{c}(Q_{\rr}(x_{0}))$ be such that
$\mathds{1}_{Q_{\rr/2}(x_{0})}\le \eta\le \mathds{1}_{Q_{\rr}(x_{0})}$ and $\snr{\nabla \eta}\lesssim \rr^{-1}$, and test \eqref{elsd} against
$w_{\varepsilon}:=\eta^{2}(\partial_{s}v_{\varepsilon}-(\partial_{s} v_{\varepsilon})_{Q_{\rr}(x_{0})})$ to get, after using Cauchy-Schwarz inequality, \eqref{keymono.1},
\eqref{Vm}$_{1,3}$, \eqref{assf.2}, and Sobolev-Poincar\'e inequality,
\begin{eqnarray}\label{5.1.1}
  \mint_{Q_{\rr/2}(x_{0})}\tx{V}_{\varepsilon}\dx&:=&\mint_{Q_{\rr/2}(x_{0})}\left(\snr{\nabla V_{\mu,p}(\nabla v_{\varepsilon})}^{2}+
  \snr{\nabla V_{1,q^{\prime}}(F^{\prime}(\nabla v_{\varepsilon}))}^{2}+\gamma_{\varepsilon}\snr{\nabla V_{1,q}(\nabla v_{\varepsilon})}^{2}\right)\dx\nonumber \\
&\le&\frac{c}{\rr^{2}}\mint_{Q_{\rr}(x_{0})}\snr{F^{\prime\prime}(\nabla v_{\varepsilon})}\snr{\nabla v_{\varepsilon}-(\nabla v_{\varepsilon})_{Q_{\rr}(x_{0})}}^{2}\dx\nonumber \\
&&+\frac{c\gamma_{\varepsilon}}{\rr^{2}}\mint_{Q_{\rr}(x_{0})}\ell_{1}(\nabla v_{\varepsilon})^{q-2}\snr{\nabla v_{\varepsilon}-(\nabla v_{\varepsilon})_{Q_{\rr}(x_{0})}}^{2}\dx\nonumber \\
&\le&\frac{c\tx{M}_{\varepsilon}}{\rr^{2}}\mint_{Q_{\rr}(x_{0})}\snr{V_{\mu,p}(\nabla v_{\varepsilon})-(V_{\mu,p}(\nabla v_{\varepsilon}))_{Q_{\rr}(x_{0})}}^{2}\dx\nonumber \\
&&+\frac{c\gamma_{\varepsilon}}{\rr^{2}}\mint_{Q_{\rr}(x_{0})}\snr{V_{1,q}(\nabla v_{\varepsilon})-(V_{1,q}(\nabla v_{\varepsilon}))_{Q_{\rr}(x_{0})}}^{2}\dx\nonumber \\
&\le&c\tx{M}_{\varepsilon}\left(\mint_{Q_{\rr}(x_{0})}\snr{\nabla V_{\mu,p}(\nabla v_{\varepsilon})}^{2_{*;2}}\dx\right)^{\frac{2}{2_{*;2}}}\nonumber \\
  &&+c\left(\mint_{Q_{\rr}(x_{0})}\gamma_{\varepsilon}^{\frac{2_{*;2}}{2}}\snr{\nabla V_{1,q}(\nabla v_{\varepsilon})}^{2_{*;2}}\dx\right)^{\frac{2}{2_{*;2}}}\le
  c\tx{M}_{\varepsilon}\left(\mint_{Q_{\rr}(x_{0})}\tx{V}_{\varepsilon}^{\frac{2_{*;2}}{2}}\dx\right)^{\frac{2}{2_{*;2}}},
\end{eqnarray}
for $c\equiv c(N,L,L_{\mu},p,q)$. From \eqref{5.1.0}-\eqref{5.1.1} we see that the assumptions of Lemma \ref{revh} are satisfied with 
\eqn{*.1}
$$
\tx{s}_{0}:=\frac{q}{p},\qquad \qquad \hat{c}\equiv \max\{c,\tx{s}_{0}\},\qquad \qquad \tx{M}\equiv \tx{M}_{\varepsilon},\qquad \quad \tx{m}:=\frac{2_{*;2}}{2}=\frac{1}{2},
$$
therefore we obtain an exponent $\tx{t}\in (1,2)$ as in \eqref{rrr.1} such that $\nra{\tx{V}_{\varepsilon}}_{L^{\tx{t}}(Q_{1/2})}\le 2^{12}\nra{\tx{V}_{\varepsilon}}_{L^{1}(Q_{1})}$,
which in particular implies
\eqn{6.2.8}
$$
\left(\mint_{Q_{1/2}}\snr{\nabla V_{1,q^{\prime}}(F^{\prime}(\nabla v_{\varepsilon}))}^{2\tx{t}}\dx\right)^{\frac{1}{2\tx{t}}}\le 2^{6}\left(\mint_{Q_{1}}\tx{V}_{\varepsilon}\dx\right)^{\frac{1}{2}}.
$$
Next, we let $\eta_{0}\in \CC^{1}_{c}(Q_{1/2})$ such that $\mathds{1}_{Q_{1/4}}\le \eta_{0}\le \mathds{1}_{Q_{1/2}}$ and $\snr{\nabla \eta_{0}}\le 2$,
observe that we have $\eta_{0}V_{1,q^{\prime}}(F^{\prime}(\nabla v_{\varepsilon}))\in \WW^{1,2\mf{t}}(\mathbb{R}^{2},\mathbb{R}^{N\times 2})$  and via \eqref{6.2.8}, bound
\begin{eqnarray}\label{6.2.9}
  \nr{\nabla(\eta_{0}V_{1,q^{\prime}}(F^{\prime}(\nabla v_{\varepsilon}))) }_{\LL^{2\mf{t}}(\mathbb{R}^{n})}&\le &
  2\nr{V_{1,q^{\prime}}(F^{\prime}(\nabla v_{\varepsilon}))}_{\LL^{2\mf{t}}(Q_{1/2})}+\nr{\nabla V_{1,q^{\prime}}(F^{\prime}(\nabla v_{\varepsilon}))}_{\LL^{2\mf{t}}(Q_{1/2})}\nonumber \\
&\le&2\nr{V_{1,q^{\prime}}(F^{\prime}(\nabla v_{\varepsilon}))}_{\LL^{2\mf{t}}(Q_{1/2})}+2^{6}\nr{\tx{V}_{\varepsilon}}_{\LL^{1}(Q_{1})}^{\frac{1}{2}}.
\end{eqnarray}
We then apply Proposition \ref{shso}, and use \eqref{6.2.9} and $\eqref{rrr.1}$ to deduce that
\begin{eqnarray}\label{6.2.10}
\nr{V_{1,q^{\prime}}( F^{\prime}(\nabla v_{\varepsilon}))}_{\LL^{\infty}(Q_{1/4})}&\le& \nr{\eta_{0}V_{1,q^{\prime}}(F^{\prime}(\nabla v_{\varepsilon}))}_{\LL^{\infty}(\mathbb{R}^{n})}\nonumber \\
&\le&\left(\frac{2\mf{t}-1}{2\mf{t}-2}\right)^{\frac{2\mf{t}-1}{2\mf{t}}}
\frac{\snr{\supp (\eta_{0}V_{1,q^{\prime}}(F^{\prime}(\nabla v_{\varepsilon})))}^{\frac{1}{2}-\frac{1}{2\mf{t}}}}{2^{1/(2\mf{t})}\omega_{2}^{1/2}(2N)^{-1/2}}
\nr{\nabla (\eta_{0}V_{1,q^{\prime}}(F^{\prime}(\nabla v_{\varepsilon})))}_{\LL^{2\mf{t}}(\mathbb{R}^{n})}\nonumber \\
&\le&2^{8}\sqrt{\frac{N}{\pi}}\left(\frac{2\mf{t}-1}{2\mf{t}-2}\right)^{\frac{2\mf{t}-1}{2\mf{t}}}
\left(\nr{V_{1,q^{\prime}}(F^{\prime}(\nabla v_{\varepsilon}))}_{\LL^{2\mf{t}}(Q_{1/2})}+\nr{\tx{V}_{\varepsilon}}_{\LL^{1}(Q_{1})}^{\frac{1}{2}}\right).
\end{eqnarray}
Before proceeding further, let us make explicit the values of the constants/exponents involving $\mf{t}$ above by means of \eqref{rrr.1}, \eqref{*} and \eqref{*.1}. We have: 
$$
\left\{
\begin{array}{c}
\displaystyle
\ \frac{2\mf{t}-1}{2\mf{t}-2}=2c_{*}\tx{M}_{\varepsilon}\\ [14pt]\displaystyle
\  \frac{2\mf{t}-1}{2\mf{t}}=\frac{2c_{*}\tx{M}_{\varepsilon}}{4c_{*}\tx{M}_{\varepsilon}-1}\stackrel{\tx{M}_{\varepsilon}\ge 1}{\le}
\frac{2c_{*}}{4c_{*}-1}\stackrel{\eqref{*.1}_{1,2}}{\le}\frac{2q}{4q-p}\\ [14pt]\displaystyle
\ \frac{\mf{t}-1}{\mf{t}}=\frac{1}{4c_{*}\tx{M}_{\varepsilon}-1}\stackrel{\tx{M}_{\varepsilon}\ge 1}{\le}\frac{1}{4c_{*}-1}\stackrel{\eqref{*.1}_{1,2}}{\le}\frac{p}{4q-p},
\end{array}
\right.
$$
where we used that $1/(2\mf{t})^{\prime}$, $(\tx{t}-1)/ \tx{t}$ are decreasing with respect to both $\tx{M}_{\varepsilon}$ and $c_{*}$.
Incorporating this information in \eqref{6.2.10}, scaling back on $Q$, and letting
$\tx{U}_{\varepsilon}:=\snr{\nabla V_{\mu,p}(\nabla u_{\varepsilon})}^{2}+\snr{\nabla V_{1,q^{\prime}}(F^{\prime}(\nabla u_{\varepsilon}))}^{2}+
\gamma_{\varepsilon}\snr{\nabla V_{1,q}(\nabla u_{\varepsilon})}^{2}$, by $\eqref{vpvqn}$ we obtain after standard manipulations
\begin{eqnarray}\label{6.2.10.1}
  \snr{F^{\prime}(\nabla u_{\varepsilon}(x_{Q}))}^{\frac{q^{\prime}}{2}}&\le& c\tx{M}_{\varepsilon}^{\frac{2q}{4q-p}}
  \nr{V_{1,q^{\prime}}(F^{\prime}(\nabla u_{\varepsilon}))}_{\LL^{\infty}(Q)}^{\frac{\mf{t}-1}{\mf{t}}}\left(\mint_{Q}\snr{V_{1,q^{\prime}}(F^{\prime}(\nabla u_{\varepsilon}))}^{2}\dx\right)^{\frac{1}{2\mf{t}}}\nonumber \\
&&+c\tx{M}_{\varepsilon}^{\frac{2q}{4q-p}}\left(\mint_{Q}r_{0}^{2}\tx{U}_{\varepsilon}\dx\right)^{\frac{1}{2}}+c\nonumber \\
&\le&c\tx{M}_{\varepsilon}^{\frac{q(4q-3p)}{2(q-p)(4q-p)}}\left[1+\left(\mint_{Q}\snr{V_{1,q^{\prime}}(F^{\prime}(\nabla u_{\varepsilon}))}^{2}\dx\right)^{\frac{1}{2}}\right]\nonumber \\
&&+c\tx{M}_{\varepsilon}^{\frac{2q}{4q-p}}\left(\mint_{Q}r_{0}^{2}\tx{U}_{\varepsilon}\dx\right)^{\frac{1}{2}}+c,
\end{eqnarray}
with $c\equiv c(N,L,L_{\mu},p,q)$. Next, we fix parameters $2^{-3}\texttt{r}\le \tau_{2}<\tau_{1}\le 2^{-5/2}\texttt{r}$, correspondingly, cubes concentric with $B$, i.e.:
$B/8\subset Q_{2^{-3}\texttt{r}}(x_{B})\subset Q_{\tau_{2}}(x_{B})\subset Q_{\tau_{1}}(x_{B})\subset Q_{2^{-5/2}\texttt{r}}(x_{B})\Subset B/4$, and observe that for any $x_{Q}\in Q_{\tau_{2}}(x_{B})$,
cube $Q_{(\tau_{1}-\tau_{2})/8}(x_{Q})\Subset Q_{\tau_{1}}(x_{B})$, so, keeping in mind that there is no loss of generality in assuming that
$\nr{F^{\prime}(\nabla u_{\varepsilon})}_{\LL^{\infty}(Q_{2^{-3}\texttt{r}}(x_{B}))}\ge 1$ (otherwise the proof would be finished already), we apply \eqref{6.2.10.1} with
$Q\equiv Q_{(\tau_{1}-\tau_{2})/8}(x_{Q})$ and $\tx{M}_{\varepsilon}\equiv \nr{F'(\nabla u_{\varepsilon})}_{\LL^{\infty}(Q_{\tau_{1}}(x_{B}))}^{(q-p)/(q-1)}$, which verifies \eqref{mmm}, to get
\begin{eqnarray*}
  \snr{F^{\prime}(\nabla u_{\varepsilon}(x_{Q}))}^{\frac{q^{\prime}}{2}}&\le&
  c\nr{F^{\prime}(\nabla u_{\varepsilon})}_{\LL^{\infty}(Q_{\tau_{1}}(x_{B}))}^{\frac{q^{\prime}(4q-3p)}{2(4q-p)}}\left[1+\frac{1}{\tau_{1}-\tau_{2}}
    \left(\int_{Q_{\frac{\tau_{1}-\tau_{2}}{8}}(x_{Q})}\snr{V_{1,q^{\prime}}(F^{\prime}(\nabla u_{\varepsilon}))}^{2}\dx\right)^{\frac{1}{2}}\right]\nonumber \\
&&+c\nr{F^{\prime}(\nabla u_{\varepsilon})}_{\LL^{\infty}(Q_{\tau_{1}}(x_{B}))}^{\frac{q^{\prime}4(q-p)}{2(4q-p)}}\left(\int_{Q_{\frac{(\tau_{1}-\tau_{2})}{8}}(x_{Q})}\tx{U}_{\varepsilon}\dx\right)^{\frac{1}{2}}+c\nonumber \\
  &\le&c\nr{F^{\prime}(\nabla u_{\varepsilon})}_{\LL^{\infty}(Q_{\tau_{1}}(x_{B}))}^{\frac{q^{\prime}(4q-3p)}{2(4q-p)}}
  \left[1+\frac{1}{\tau_{1}-\tau_{2}}\left(\int_{Q_{\tau_{1}}(x_{B})}\snr{V_{1,q^{\prime}}(F^{\prime}(\nabla u_{\varepsilon}))}^{2}\dx\right)^{\frac{1}{2}}\right]\nonumber \\
&&+c\nr{F^{\prime}(\nabla u_{\varepsilon})}_{\LL^{\infty}(Q_{\tau_{1}}(x_{B}))}^{\frac{q^{\prime}4(q-p)}{2(4q-p)}}\left(\int_{Q_{\tau_{1}}(x_{B})}\tx{U}_{\varepsilon}\dx\right)^{\frac{1}{2}}+c,
\end{eqnarray*}
for $c\equiv c(N,L,L_{\mu},p,q)$. Taking the supremum for all $x_{Q}\in Q_{\tau_{2}}(x_{B})$, and setting $\theta_{1}:=(4q-3p)/(4q-p),\ \theta_{2}:=4(q-p)/(4q-p)\in (0,1)$, by Young inequality we obtain
\begin{eqnarray}\label{6.2.11}
  \nr{F^{\prime}(\nabla u_{\varepsilon})}_{\LL^{\infty}(Q_{\tau_{2}}(x_{B}))}^{\frac{q^{\prime}}{2}} &\le&
  c\nr{F^{\prime}(\nabla u_{\varepsilon})}_{\LL^{\infty}(Q_{\tau_{1}}(x_{B}))}^{\frac{\theta_{1} q^{\prime}}{2}}\left[1+\frac{1}{\tau_{1}-\tau_{2}}
    \left(\int_{B/4}\snr{V_{1,q^{\prime}}(F^{\prime}(\nabla u_{\varepsilon}))}^{2}\dx\right)^{\frac{1}{2}}\right]\nonumber \\
&&+c\nr{F^{\prime}(\nabla u_{\varepsilon})}_{\LL^{\infty}(Q_{\tau_{1}}(x_{B}))}^{\frac{\theta_{2} q^{\prime}}{2}}\left(\int_{B/4}\tx{U}_{\varepsilon}\dx\right)^{\frac{1}{2}}+c\nonumber \\
&\le&\frac{1}{4}\nr{F^{\prime}(\nabla u_{\varepsilon})}_{\LL^{\infty}(Q_{\tau_{1}}(x_{B}))}^{\frac{q^{\prime}}{2}}+c\left(\int_{B/4}\tx{U}_{\varepsilon}\dx\right)^{\frac{1}{2(1-\theta_{2})}}\nonumber \\
&&+\frac{c}{(\tau_{1}-\tau_{2})^{\frac{1}{1-\theta_{1}}}}\left(\int_{B/4}\snr{V_{1,q^{\prime}}(F^{\prime}(\nabla u_{\varepsilon}))}^{2}\dx\right)^{\frac{1}{2(1-\theta_{1})}}+c
\end{eqnarray}
with $c\equiv c(N,L,L_{\mu},p,q)$. We can then apply Lemma \ref{l5}, \eqref{4.2.0}, \eqref{approx}$_{1,2}$, and \eqref{hdes.1} (keep Remark \ref{rem} in mind), to conclude with
\begin{eqnarray}\label{6.2.12}
   \nr{F^{\prime}(\nabla u_{\varepsilon})}_{\LL^{\infty}(B/8)}^{\frac{q^{\prime}}{2}}&\le&c\left(\mint_{B/4}\snr{V_{1,q^{\prime}}(F^{\prime}(\nabla u_{\varepsilon}))}^{2}\dx\right)^{\frac{1}{2(1-\theta_{1})}}
   +c\left(\int_{B/4}\tx{U}_{\varepsilon}\dx\right)^{\frac{1}{2(1-\theta_{2})}}+c\nonumber \\
 &\le&c\left(\mint_{B}F(\nabla u_{\varepsilon})+\gamma_{\varepsilon}(1+\snr{\nabla u_{\varepsilon}}^{2})^{\frac{q}{2}}+1\dx\right)^{\tx{b}}\nonumber \\
 &\le&c\left(\mint_{B}1+F(\nabla u)\dx\right)^{\tx{b}}+\texttt{o}(\varepsilon),
\end{eqnarray}
with $c\equiv c(N,L,L_{\mu},p,q)$, $\tx{b}\equiv \tx{b}(p,q)$. Now, by $\eqref{approx}_{3}$ we know that $\nabla u_{\varepsilon}\to \nabla u$ strongly in $\LL^{p}(B,\mathbb{R}^{N})$,
so by \eqref{assf}$_{1}$ (up to subsequences) $F^{\prime}(\nabla u_{\varepsilon})\to F^{\prime}(\nabla u)$ almost everywhere.
This means that we can send $\varepsilon\to 0$ in \eqref{6.2.12} and pass to the weak$^{*}$-limit to deduce
$$
\nr{F^{\prime}(\nabla u)}_{\LL^{\infty}(B/8)}^{\frac{q^{\prime}}{2}}\le c\left(\mint_{B}1+F(\nabla u)\dx\right)^{\tx{b}},
$$
for $c\equiv c(N,L,L_{\mu},p,q)$, $\tx{b}\equiv \tx{b}(p,q)$, which, together with \eqref{f'} gives \eqref{ll}
(with possibly different exponent $\tx{b}\equiv \tx{b}(p,q)$), and a standard covering argument leads to $\nabla u\in \LL^{\infty}_{\loc}(\Omega,\mathbb{R}^{N\times 2})$, and,
whenever $\Omega_{2}\Subset \Omega_{1}\Subset \Omega$ are open sets as in the statement of Theorem \ref{2dreg}, it is
\eqn{cover}
$$
\nr{\nabla u}_{\LL^{\infty}(\Omega_{2})}\le c(N,L,L_{\mu},p,q,\F(u;\Omega_{1}),\dist(\Omega_{2},\partial\Omega_{1})).
$$

\subsection{Proof of Theorem \ref{2dreg}: full regularity}
Once it is known, by either of the arguments in the previous two sections, that minima of $\F$ are locally Lipschitz continuous,
the nonuniform ellipticity of the integrand $F$ prescribed by \eqref{assf}$_{3}$ becomes immaterial. The gradient boundedness, Theorem \ref{hd},
and a simple difference quotients argument guarantee that we can recover \eqref{elsd} with $F$, $u$, instead of $F_{\varepsilon}$, $u_{\varepsilon}$, so,
fixing open sets $\Omega_{2}\Subset \Omega_{1}\Subset \Omega$ as in \eqref{cover}, and cubes $Q_{\rr}(x_{0})\Subset Q\Subset \Omega_{2}$, with
$B/16\Subset Q\equiv Q_{2^{-15/4}\texttt{r}}(x_{Q})\Subset B/8$, computations analogous to those in \eqref{5.1.1} eventually yield 
\begin{flalign*}
  &\mint_{Q_{\rr/2}(x_{0})}\snr{\nabla V_{\mu,p}(\nabla u)}^{2}+\snr{\nabla V_{1,q^{\prime}}(F^{\prime}(\nabla u))}^{2}\dx\le
  \frac{c}{\rr^{2}}\mint_{Q_{\rr}(x_{0})}\snr{F^{\prime\prime}(\nabla u)}\snr{\nabla u-(\nabla u)_{Q_{\rr}(x_{0})}}^{2}\dx\nonumber \\
&\qquad \qquad \qquad \le\frac{c\tx{M}}{\rr^{2}}\mint_{Q_{\rr}(x_{0})}\snr{V_{\mu,p}(\nabla u)-(V_{\mu,p}(\nabla u))_{Q_{\rr}(x_{0})}}^{2}\dx\nonumber \\
  &\qquad \qquad \quad \ \stackrel{\eqref{cover}}{\le}c\left(\mint_{Q_{\rr}(x_{0})}\left(\snr{\nabla V_{\mu,p}(\nabla u)}^{2}+
  \snr{\nabla V_{1,q^{\prime}}(F^{\prime}(\nabla u))}^{2}\right)^{\mf{m}}\dx\right)^{\frac{1}{\mf{m}}},
\end{flalign*}
where we set $\tx{M}:=1+\nr{F^{\prime}(\nabla u)}_{\LL^{\infty}(Q)}^{(q-p)/(q-1)}$, and it is $c\equiv c(N,L,L_{\mu},p,q,\F(u;\Omega_{1}),\dist(\Omega_{2},\partial\Omega_{1}))$, $\mf{m}=1/2$.
We can now apply classical Gehring lemma \cite[Theorem 6.6]{giu} to prove the existence of an exponent
$\mf{t}\equiv \mf{t}(N,L,L_{\mu},p,q,\F(u;\Omega_{1}),\dist(\Omega_{2},\partial\Omega_{1}))\in (1,2)$, cf. \eqref{rrr.1}, such that 
\begin{flalign*}
&\nra{\nabla V_{\mu,p}(\nabla u)}_{\LL^{2\tx{t}}(B/16)}+\nra{\nabla V_{1,q^{\prime}}(F^{\prime}(\nabla u))}_{\LL^{2\tx{t}}(B/16)}\nonumber \\
&\qquad \qquad \qquad\qquad \le c\nra{\nabla V_{\mu,p}(\nabla u)}_{\LL^{2\tx{t}}(Q/2)}+c\nra{\nabla V_{1,q^{\prime}}(F^{\prime}(\nabla u))}_{\LL^{2\tx{t}}(Q/2)}\nonumber \\
&\qquad \qquad \qquad \qquad\le c\nra{\nabla V_{\mu,p}(\nabla u)}_{\LL^{2}(Q)}+c\nra{\nabla V_{1,q^{\prime}}(F^{\prime}(\nabla u))}_{\LL^{2}(Q)}\le c\nra{(F(\nabla u)+1)}_{\LL^{1}(B)}^{\tx{b}},
\end{flalign*}
where we also used \eqref{hdes}, and it is $c\equiv c(N,L,L_{\mu},p,q,\F(u;\Omega_{1}),\dist(\Omega_{2},\partial\Omega_{1})$, $\tx{b}\equiv \tx{b}(p,q)$, and \eqref{ll.1} is proven.
By Sobolev Morrey embedding theorem we further obtain that $V_{\mu,p}(\nabla u)\in \CC^{0,1-1/\mf{t}}(Q/2,\mathbb{R}^{N\times 2})$, which in turn implies that
$\nabla u\in \CC^{0,\beta_{0}}(Q/2,\mathbb{R}^{N\times 2})$, $\beta_{0}\equiv \beta_{0}(N,L,L_{\mu},p,q,,\F(u;\Omega_{1}),\dist(\Omega_{2},\partial\Omega_{1}))\in (0,1)$.
With this last information at hand, if $\mu>0$ in \eqref{assf}$_{3}$ we can follow a well-known strategy \cite[Section 8.8]{giu} to conclude with gradient H\"older
continuity up to any exponent less than one. If in addition to the nondegeneracy condition $\mu>0$, the integrand governing $\F$ is also real analytic,
by-now standard theory \cite[Chapter 6]{MorreyB}, \cite{m1} yields that $u$ is real analytic. A standard covering argument finishes the proof.
\subsubsection{Proof of Theorem \ref{polyt}, \emph{(}ii.\emph{)}}\label{iiii} The proof of the second claim of Theorem \ref{polyt} is a direct consequence
of Proposition \ref{poly} and Theorem \ref{2dreg} formulated for minima of variational integrals governed by nondegenerate, analytic integrands. Recalling the content of Section \ref{poly11}, the proof of Theorem \ref{polyt} is complete.
\begin{remark}\label{rempq}
\emph{Our techniques apply also to the genuine $(p,q)$-nonuniform ellipticity conditions, i.e. 
\eqn{r72}
$$
\ell_{\mu}(z)^{p-2}\mathds{I}_{N\times 2}\lesssim F^{\prime\prime}(z)\lesssim \ell_{\mu}(z)^{p-2}\mathds{I}_{N\times 2}+\ell_{\mu}(z)^{q-2}\mathds{I}_{N\times 2}
$$
in the sense of bilinear forms, with $\mu\in [0,1]$, and $2\le p< q<2p$, the usual bound in two space dimensions.
To adapt the arguments in Sections \ref{2dm}-\ref{2dg} to the $(p,q)$-framework, we recall that \eqref{r72} implies 
\eqn{remelr}
$$
\frac{\snr{F^{\prime\prime}(z)}}{\ell_{\mu}(z)^{p-2}}\lesssim 1+\ell_{\mu}(z)^{q-p},
$$
accordingly, the following (very minor) modifications need to be (respectively) implemented - higher differentiability results in the spirit
of Theorem \ref{hd} that allow controlling the resulting right-hand side term being available in \cite{elm2,ELM}.
\begin{itemize}
\item Since \eqref{remelr} is in force, in \eqref{4.12} H\"older inequality needs to be applied with conjugate exponents $(p/(q-p),p/(2p-q))$,
  so in the last line of \eqref{4.12}, term $\tx{P}_{\varepsilon}'(\rr)^{\frac{q-p}{2p}}$ replaces $\tx{S}_{\varepsilon}'(\rr)^{\frac{q-p}{2q}}$, and
  $\nra{\tx{V}_{\varepsilon,p}}_{\LL^{\frac{2p}{2p-q}}(\partial B_{\rr})}$ appears instead of $\nra{\tx{V}_{\varepsilon,p}}_{\LL^{\frac{2q}{p}}(\partial B_{\rr})}$.
  As a consequence, $V_{\mu,p}(\nabla u_{\varepsilon})$ replaces $V_{1,q^{\prime}}(F^{\prime}(\nabla u_{\varepsilon}))$ everywhere in \eqref{log.1},
  and the term in parenthesis that needs to be controlled via trace theorem, now is $\nra{V_{\mu,p}(\nabla u_{\varepsilon})}_{\LL^{2}(\partial B_{\rr})}^{\frac{q-p}{p}}$,
  so in \eqref{mt} the argument of the exponential must be replaced by $\snr{V_{\mu,p}(\nabla u_{\varepsilon})}^{2}$ (any control on the stress tensor
  $F^{\prime}(\nabla u_{\varepsilon})$ being lost in the genuine $(p,q)$-case).
\item In \eqref{mmm} the lower bound on $\tx{M}_{\varepsilon}$ needs to be replaced by $\max\{1,\nr{\nabla u_{\varepsilon}}_{\LL^{\infty}(Q)}^{q-p}\}$;
  in \eqref{6.2.8} there must be $\snr{\nabla V_{\mu,p}(\nabla u_{\varepsilon})}$ instead of $\snr{\nabla V_{1,q^{\prime}}(F^{\prime}(\nabla u_{\varepsilon}))}$
  that is no longer available; and in \eqref{*.1}$_{1}$ we need that $\tx{s}_{0}= p/(2p-q)$ to be able to apply Young inequality and reabsorbe terms in \eqref{6.2.11}.
\end{itemize}
Very recently, Sch\"affner \cite{sch24} improved the assumptions for the validity of Lipschitz regularity for scalar local minimizers of $(p,q)$-nonuniformly elliptic functionals in two space dimensions by updating the bound on exponents $(p,q)$ to $q<3p$. The approach in \cite{sch24} relies on a fine argument based on $1$-d interpolation and maximum principle, that are in general not available in the vectorial setting.}
\end{remark}

\section{The scalar case}\label{scalar}
This section is devoted to the proof, based on a renormalized \cite{demi3} version of the Moser's iteration designed in \cite{bs,demi1,demi3,ma1},
of local Lipschitz continuity for scalar minimizers of functional $\F$ under Legendre $(p,q)$-growth conditions.
Given the quite technical content of this last part of the paper, for simplicity we split it in two steps, ultimately yielding the proof of Theorem \ref{scat}.

\subsection{An abstract \texorpdfstring{$\WW^{1,2}$}{} to \texorpdfstring{$\LL^{\infty}$}{} iteration}
Here we derive an abstract $\LL^{\infty}$-$\WW^{1,2}$ estimate valid for any sufficiently regular function satisfying a suitable reverse H\"older type
inequality for arbitrarily large powers.
\begin{lemma}\label{am}
Let $-1\le\alpha_{0}<\infty$, $0<\gamma_{0}<1$ be numbers, $v\in \LL^{\infty}_{\loc}(B_{1}(0))$ be a function such that $v^{(\alpha_{0}+2)/2}\in \WW^{1,2}_{\loc}(B_{1}(0))$ and,
fixed balls $B_{1/8}(0)\subset B_{\sigma}(0)\Subset B_{1}(0)$, $\sigma\in (1/8,1]$ arbitrary, for all nonnegative $\eta\in \CC^1_{c}(B_{\sigma}(0))$, any $\alpha\ge \alpha_{0}$,
and some absolute constants $c_{0}\ge 1$, $\tx{M}\ge 1$, inequality
\eqn{6.1}
$$
\nr{\eta\nabla v^{\frac{\alpha+2}{2}}}_{\LL^{2}(B_{1})}\le c_{0}\tx{A}_{\alpha}\tx{M}^{\frac{1}{2}}\nr{(v^{\frac{\alpha+2}{2}}+1)\nabla \eta}_{\LL^{2}(B_{1})},
$$
holds true with 
\eqn{aa}
$$
\tx{A}_{\alpha}:=\left\{
\begin{array}{c}
\displaystyle
\ 1\qquad \quad \mbox{if} \ \  \alpha=\alpha_{0}\\ [8pt]\displaystyle
\ \frac{\alpha+2}{\alpha+1}\qquad\quad \mbox{if} \ \ \alpha>\alpha_{0}.
\end{array}
\right.
$$
Then, for every ball $B_{1/8}(0)\subseteq B_{\tau_{2}}(0)\subset B_{\tau_{1}}(0)\subseteq B_{\sigma}(0)$, the $\LL^{\infty}$-$\WW^{1,2}$ estimate
\eqn{6.8}
$$
\nr{v}_{\LL^{\infty}(B_{\tau_{2}})}\le \frac{c\tx{M}^{\frac{\gamma}{(2+\alpha_{0})(1-\gamma)}}}{(\tau_{1}-\tau_{2})^{\tx{a}_{0}}}\nra{(v^{\frac{\alpha_{0}+2}{2}}+1)}_{\WW^{1,2}(B_{\tau_{1}})}^{\frac{2}{\alpha_{0}+2}}
$$
is verified, where
\eqn{gamma}
$$
\gamma:=\left\{
\begin{array}{c}
\displaystyle
\ \frac{n-3}{n-1}\qquad\quad  \mbox{if} \ \  n\ge 4\\ [8pt]\displaystyle
\ \gamma_{0}\qquad\quad  \mbox{if} \ \ n\in \{2,3\},
\end{array}
\right.
$$
and it is $c,\tx{a}_{0}\equiv c,\tx{a}_{0}(n,\alpha_{0},\gamma_{0})$.
\end{lemma}
\begin{proof}
We fix parameters $1/8\le \tau_{2}<\tau_{1}\le \sigma< 1$, and, for $i\in \N\cup \{0\}$ introduce radii $\rr_{i}:=\tau_{2}+2^{-i}(\tau_{1}-\tau_{2})$, balls $B_{i}:=B_{\rr_{i}}(0)$, cut-off functions $\eta_{i}\in \CC^{1}_{c}(B_{i})$ such that $\mathds{1}_{B_{i+1}}\le \eta_{i}\le \mathds{1}_{B_{i}}$, and apply \cite[Lemma 3]{bs}, see also \cite[Lemma 3]{krs}, to the right-hand side of \eqref{6.1} with $\eta= \eta_{i}$. We obtain
\begin{flalign}\label{6.3}
\nr{(v^{\frac{\alpha+2}{2}}+1)\nabla \eta_{i}}_{\LL^{2}(B_{i})}^{2}\le \frac{2}{(\rr_{i}-\rr_{i+1})^{\frac{\delta+1}{\delta}}}\left(\int_{\rr_{i+1}}^{\rr_{i}}\left(\int_{\partial B_{s}}(v^{\alpha+2}+1)\d\mathcal{H}^{n-1}(x)\right)^{\delta}\ds\right)^{\frac{1}{\delta}}
\end{flalign}
for all $\delta\in (0,1]$. With $\gamma\in (0,1]$ as in \eqref{gamma}, we record that 
\eqn{6.2}
$$
\left(\left(\frac{2}{\gamma}\right)_{*;n-1}\right)^{-1}=\min\left\{\frac{\gamma}{2}+\frac{1}{n-1},1\right\}\ge \frac{1}{2},
$$
define $\ti{v}:=v^{\frac{\gamma(\alpha+2)}{2}}$, and estimate
\begin{eqnarray}\label{6.2.0}
\nra{v^{\alpha+2}+1}_{\LL^{1}(\partial B_{s})}^{\delta}&=&\nra{\ti{v}^{\frac{2}{\gamma}}+1}^{\delta}_{\LL^{1}(\partial B_{s})}\nonumber \\
&\stackrel{\eqref{spv.2}}{\le}&cs^{\frac{2\delta}{\gamma}}\nra{\nabla \ti{v}}_{\LL^{\left(\frac{2}{\gamma}\right)_{*;n-1}}(\partial B_{s})}^{\frac{2\delta}{\gamma}}+c\nra{(\ti{v}+1)}_{\LL^{\left(\frac{2}{\gamma}\right)_{*;n-1}}(\partial B_{s})}^{\frac{2\delta}{\gamma}}\nonumber \\
&\stackrel{\eqref{6.2}}{\le}&cs^{\frac{2\delta}{\gamma}}\nra{\nabla \ti{v}}_{\LL^{2}(\partial B_{s})}^{\frac{2\delta}{\gamma}}+c\nra{(\ti{v}+1)}_{\LL^{2}(\partial B_{s})}^{\frac{2\delta}{\gamma}},
\end{eqnarray}
for $c\equiv c(n,\gamma_{0})$. We plug \eqref{6.2.0} in \eqref{6.3}, choose $\delta=\gamma$, and restore the original notation to get
\begin{eqnarray}\label{6.4}
\nr{(v^{\frac{\alpha+2}{2}}+1)\nabla \eta_{i}}_{\LL^{2}(B_{i})}^{2}\le \frac{c}{(\rr_{i}-\rr_{i+1})^{\frac{\gamma+1}{\gamma}}}\nr{(v^{\frac{\gamma(\alpha+2)}{2}}+1)}_{\WW^{1,2}(B_{i})}^{\frac{2}{\gamma}},
\end{eqnarray}
where we also used that $\rr_{i}\ge 2^{-3}$ for all $i\in \N\cup \{0\}$, and it is $c\equiv c(n,\gamma_{0})$. 
Merging \eqref{6.4} and \eqref{6.1} we obtain,
\begin{eqnarray}\label{6.5}
\nr{(v^{\frac{\alpha+2}{2}}+1)}_{\WW^{1,2}(B_{i+1})}^{2}&\stackrel{\eqref{spv.3}}{\le}&c\nr{\nabla v^{\frac{\alpha+2}{2}}}_{\LL^{2}(B_{i+1})}^{2}+c\nr{(v^{\frac{\alpha+2}{2}}+1)}_{\LL^{2\gamma}(B_{i+1})}^{2}\nonumber \\
&\stackrel{\eqref{6.1},\eqref{6.4}}{\le}&
 \frac{c\tx{A}_{\alpha}^{2}\tx{M}}{(\rr_{i}-\rr_{i+1})^{\frac{\gamma+1}{\gamma}}}\nr{(v^{\frac{\gamma(\alpha+2)}{2}}+1)}_{\WW^{1,2}(B_{i})}^{\frac{2}{\gamma}},
\end{eqnarray}
with $c\equiv c(n,c_{0},\gamma_{0})$. We have only one degree of freedom left in \eqref{6.5}, that is $\alpha\ge \alpha_{0}\ge -1$,
so we introduce sequence $\{\alpha_{i}\}_{i\in \N\cup\{0\}}$, recursively defined as
$$
\alpha_{0}\ge-1 \ \ \mbox{arbitrary}, \qquad \qquad \alpha_{i}:=\frac{1}{\gamma}\alpha_{i-1}+2\left(\frac{1}{\gamma}-1\right),\ \ i\in \N.
$$
The above position immediately implies that
$$
\alpha_{i}=\frac{2+\alpha_{0}}{\gamma^{i}}-2\to \infty \quad  \mbox{as}\ \ i\to \infty.
$$
Set $\tx{V}_{i}:=\nra{(v^{\frac{\alpha_{i}+2}{2}}+1)}_{\WW^{1,2}(B_{i})}^{\frac{2}{\alpha_{i}+2}}$, and for $i\in \N$ rearrange \eqref{6.5} as
\begin{flalign}\label{6.6}
  \tx{V}_{i}\le \left[c\tx{A}_{\alpha_{i}}^{2}
    \left(\frac{2^{i}}{\tau_{1}-\tau_{2}}\right)^{\frac{\gamma+1}{\gamma}}\right]^{\frac{1}{\alpha_{i}+2}}\tx{M}^{\frac{1}{\alpha_{i}+2}}\tx{V}_{i-1}^{\frac{\alpha_{i-1}+2}{\gamma(\alpha_{i}+2)}},
\end{flalign}
with $c\equiv c(n,c_{0},\gamma_{0})$.
Iterating \eqref{6.6} over $j\in \{0,\cdots,i-1\}$ we end up with
\begin{flalign}\label{6.7}
  \tx{V}_{i}\le \tx{M}^{\frac{1}{\alpha_{i}+2}\sum_{k=0}^{i-1}\frac{1}{\gamma^{k}}} \tx{V}_{0}^{\frac{\alpha_{0}+2}{\gamma^{i}(\alpha_{i}+2)}} \prod_{k=0}^{i-1}\left[c\tx{A}_{\alpha_{i-k}}^{2}
    \left(\frac{2^{i-k}}{\tau_{1}-\tau_{2}}\right)^{\frac{\gamma+1}{\gamma}}\right]^{\frac{1}{\alpha_{i}+2}\sum_{j=0}^{k}\frac{1}{\gamma^{j}}}.
\end{flalign}
Let us examine the behavior as $i\to \infty$ of the various quantities appearing in \eqref{6.7}. We have
\begin{flalign*}
\frac{1}{\alpha_{i}+2}\sum_{k=0}^{i-1}\frac{1}{\gamma^{k}}
\to \frac{\gamma}{(2+\alpha_{0})(1-\gamma)},\qquad \qquad \qquad \frac{\alpha_{0}+2}{\gamma^{i}(\alpha_{i}+2)}=1,
\end{flalign*}
and
\begin{flalign*}
&\prod_{k=0}^{i-1}\left[c\tx{A}_{\alpha_{i-k}}^{2}\left(\frac{2^{i-k}}{\tau_{1}-\tau_{2}}\right)^{\frac{\gamma+1}{\gamma}}\right]^{\frac{1}{\alpha_{i}+2}\sum_{j=0}^{k}\frac{1}{\gamma^{j}}}\nonumber \\
  & \qquad \qquad \quad =\exp\left\{\sum_{k=0}^{i-1}\frac{1}{\alpha_{i}+2}\sum_{j=0}^{k}\frac{1}{\gamma^{j}}\log\left(c\tx{A}_{\alpha_{i-k}}^{2}
  \left(\frac{2^{i-k}}{\tau_{1}-\tau_{2}}\right)^{\frac{\gamma+1}{\gamma}}\right)\right\}\nonumber \\
&\qquad \qquad \quad \le \exp\left\{\frac{4(\gamma+1)}{(1-\gamma)^{2}(2+\alpha_{0})}\left(\log\left(\frac{c}{\tau_{1}-\tau_{2}}\right)+\log\left(\frac{2+\alpha_{0}}{2+\alpha_{0}-\gamma}\right)\right)\right\}\nonumber \\
& \qquad\qquad \quad \quad \cdot\exp\left\{\frac{(\gamma+1)\log(2)}{\gamma(1-\gamma)(2+\alpha_{0})}\sum_{k=0}^{i-1}(i-k)\gamma^{(i-k)}\right\}\le \frac{c}{(\tau_{1}-\tau_{2})^{\tx{a}_{0}}}<\infty,
\end{flalign*}
for $c,\tx{a}_{0}\equiv c,\tx{a}_{0}(n,\alpha_{0},\gamma_{0})$. We then send $i\to \infty$ in \eqref{6.7} to conclude with \eqref{6.8} and the proof is complete.
\end{proof}

\subsection{A Caccioppoli type inequality for powers}\label{s62}
For $z\in \mathbb{R}^{n}$ and $\alpha\ge -1$, we introduce functions
\eqn{gg.1}
$$
\tx{L}(z):=\ell_{\mu}(z)^{p},\qquad \qquad \quad \tx{l}_{\alpha}(z):=\tx{L}(z)^{\frac{(\alpha+2)}{2}}+1,
$$
and recall that whenever $w$ is a twice differentiable function, the inequalities
\eqn{gg}
$$
\snr{\nabla V_{\mu,p}(\nabla w)}^{2}\approx \ell_{\mu}(\nabla w)^{p-2}\snr{\nabla^{2}w}^{2}\gtrsim \snr{\nabla \tx{L}(\nabla w)^{\frac{1}{2}}}^{2}
$$
holds up to constants depending only on $p$. Next, we look back at the family of approximating integrals defined in Section \ref{as},
and record that by \eqref{cdq}$_{1,2,3}$, and classical scalar regularity results \cite[Chapter 8]{giu} it is $u_{\varepsilon}\in \WW^{1,\infty}_{\loc}(B)\cap \WW^{2,2}_{\loc}(B)$.
Then we scale $u_{\varepsilon}$ in such a way that function $v_{\varepsilon}(x):=(u_{\varepsilon}(x_{0}+\rrr x)-(u_{\varepsilon})_{B})/\rrr$ minimizes $\F$ on $B_{1}(0)$,
name $v(x):=(u(x_{0}+\rrr x)-(u)_{B})/\rrr$, and observe that by construction it is 
\eqn{apri}
$$
v_{\varepsilon}\in \WW^{1,\infty}_{\loc}(B_{1}(0))\cap \WW^{2,2}_{\loc}(B_{1}(0)).
$$ 
Now we are ready to prove a homogenized Caccioppoli type inequality involving arbitrary powers of $\tx{L}(\nabla v_{\varepsilon})$.
\begin{lemma}\label{l62}
For all nonnegative $\eta\in \CC^{1}_{c}(B_{1/4}(0))$, $\alpha\ge -1$ and all constants $\tx{M}_{\varepsilon}\ge 1$ such that
\eqn{mm}
$$
\tx{M}_{\varepsilon}\ge \max\left\{\nr{F_{\varepsilon}'(\nabla v_{\varepsilon})}_{\LL^{\infty}(\supp(\eta))}^{\frac{q-p}{q-1}},1\right\}
$$
the Caccioppoli type inequality for powers
\eqn{cacc}
$$
\nr{\eta \nabla \tx{l}_{\alpha}(\nabla v_{\varepsilon})}_{\LL^{2}(B_{1})}\le c\tx{A}_{\alpha}\tx{M}_{\varepsilon}^{\frac{1}{2}}\nr{\tx{l}_{\alpha}(\nabla v_{\varepsilon})\nabla \eta}_{\LL^{2}(B_{1})},
$$
holds with $c\equiv c(n,N,L,p,q)$, and $\tx{A}_{\alpha}$ as in \eqref{aa}.
\end{lemma}
\begin{proof}
With $\eta\in \CC^{1}_{c}(B_{1/4})$ being a nonnegative cut-off function, $\alpha > -1$, and $s\in \{1,\cdots,n\}$,
we test \eqref{elsd} against $w_{s,\alpha}:=\eta^{2}\tx{L}(\nabla v_{\varepsilon})^{\alpha+1}\partial_{s}v_{\varepsilon}$, admissible by \eqref{apri}, to get
\begin{eqnarray}\label{6.0.1}
0&=&\sum_{s=1}^{n}\int_{B_{1}}\eta^{2}\tx{L}(\nabla v_{\varepsilon})^{\alpha+1}\langle F_{\varepsilon}''(\nabla v_{\varepsilon})\partial_{s}\nabla v_{\varepsilon},\partial_{s}\nabla v_{\varepsilon}\rangle\dx\nonumber \\
&&+(\alpha+1)\sum_{s=1}^{n}\int_{B_{1}}\eta^{2}\tx{L}(\nabla v_{\varepsilon})^{\alpha}\langle F_{\varepsilon}''
(\nabla v_{\varepsilon})\partial_{s}v_{\varepsilon}\partial_{s}\nabla v_{\varepsilon},\nabla \tx{L}(\nabla v_{\varepsilon})\rangle\dx\nonumber \\
&&+2\sum_{s=1}^{n}\int_{B_{1}}\eta\tx{L}(\nabla v_{\varepsilon})^{\alpha+1}\langle F_{\varepsilon}''
(\nabla v_{\varepsilon})\partial_{s}\nabla v_{\varepsilon}\partial_{s}v_{\varepsilon},\nabla \eta \rangle\dx=:\mbox{(I)}+\mbox{(II)}+\mbox{(III)}.
\end{eqnarray}
We immediately bound below
\begin{eqnarray}\label{6.0.2}
\mbox{(I)}&\stackrel{\eqref{cdq}_{2}}{\ge}&\frac{1}{\Lambda}\int_{B_{1}}\eta^{2}\tx{L}(\nabla v_{\varepsilon})^{\alpha+1}\ell_{\mu}(\nabla v_{\varepsilon})^{p-2}\snr{\nabla^{2}v_{\varepsilon}}^{2}\dx\nonumber \\
&\stackrel{\eqref{gg}}{\ge}&c\int_{B_{1}}\eta^{2}\tx{L}(\nabla v_{\varepsilon})^{\alpha+1}\snr{\nabla \tx{l}_{-1}(\nabla v_{\varepsilon})}^{2}\dx,
\end{eqnarray}
for $c\equiv c(n,L,p,q)$. Now notice that
\eqn{6.0.3}
$$
\nabla \tx{L}(\nabla v_{\varepsilon})=p\ell_{\mu}(\nabla v_{\varepsilon})^{p-2}\sum_{s=1}^{n}\partial_{s}v_{\varepsilon}\partial_{s}\nabla v_{\varepsilon},
$$
so we control
\begin{eqnarray*}
  \mbox{(II)}&\stackrel{\eqref{6.0.3}}{=}&\frac{(\alpha+1)}{p}\int_{B_{1}}\eta^{2}\tx{L}(\nabla v_{\varepsilon})^{\alpha}\left\langle
  \frac{F_{\varepsilon}''(\nabla v_{\varepsilon})}{\ell_{\mu}(\nabla v_{\varepsilon})^{p-2}}\nabla \tx{L}(\nabla v_{\varepsilon}),\nabla \tx{L}(\nabla v_{\varepsilon})\right\rangle\dx\nonumber \\
&\stackrel{\eqref{cdq}_{2}}{\ge}&\frac{\alpha+1}{p\Lambda}\int_{B_{1}}\eta^{2}\tx{L}(\nabla v_{\varepsilon})^{\alpha}\snr{\nabla \tx{L}(\nabla v_{\varepsilon})}^{2}\dx\nonumber \\
  &=&\frac{4(\alpha+1)}{p\Lambda(\alpha+2)^{2}}\int_{B_{1}}\eta^{2}\snr{\nabla \tx{L}(\nabla v_{\varepsilon})^{\frac{\alpha+2}{2}}}^{2}\dx=\frac{4(\alpha+1)}{p\Lambda(\alpha+2)^{2}}
  \int_{B_{1}}\eta^{2}\snr{\nabla \tx{l}_{\alpha}(\nabla v_{\varepsilon})}^{2}\dx.
\end{eqnarray*}
Finally, via Cauchy-Schwarz and Young inequalities we obtain
\begin{eqnarray*}
  \snr{\mbox{(III)}}&\stackrel{\eqref{6.0.3}}{=}&\frac{2}{p}\left|\int_{B_{1}}\eta\tx{L}(\nabla v_{\varepsilon})^{\alpha+1}
  \left\langle \frac{F_{\varepsilon}''(\nabla v_{\varepsilon})}{\ell_{\mu}(\nabla v_{\varepsilon})^{p-2}}\nabla \tx{L}(\nabla v_{\varepsilon}),\nabla \eta\right\rangle\dx\right|\nonumber \\
  &\le&\frac{\mbox{(II)}}{2}+\frac{8}{p(\alpha+1)}\int_{B_{1}\cap\{\snr{\nabla v_{\varepsilon}}>1\}}
  \snr{\nabla \eta}^{2}\tx{L}(\nabla v_{\varepsilon})^{\alpha+2}\frac{\snr{F_{\varepsilon}''(\nabla v_{\varepsilon})}}{\ell_{\mu}(\nabla v_{\varepsilon})^{p-2}}\dx\nonumber \\
&&+\frac{8}{p(\alpha+1)}\int_{B_{1}\cap \{\snr{\nabla v_{\varepsilon}}\le 1\}}\snr{\nabla \eta}^{2}\tx{L}(\nabla v_{\varepsilon})^{\alpha+1+\frac{2}{p}}\snr{F_{\varepsilon}''(\nabla v_{\varepsilon})}\dx\nonumber \\
  &\stackrel{\eqref{cdq}_{4}}{\le}&\frac{\mbox{(II)}}{2}+\frac{8\Lambda}{p(\alpha+1)}
  \int_{B_{1}\cap \{\snr{\nabla v_{\varepsilon}}>1\}}\snr{\nabla \eta}^{2}\tx{L}(\nabla v_{\varepsilon})^{\alpha+2}
  \left(1+\snr{F_{\varepsilon}'(\nabla v_{\varepsilon})}^{\frac{q-2}{q-1}}\right)\ell_{\mu}(\nabla v_{\varepsilon})^{2-p}\dx\nonumber \\
&&+\frac{c}{\alpha+1}\int_{B_{1}\cap \{\snr{\nabla v_{\varepsilon}}\le 1\}}\snr{\nabla \eta}^{2}\tx{L}(\nabla v_{\varepsilon})^{\alpha+1+\frac{2}{p}}\dx\nonumber \\
  &\stackrel{\eqref{f'}}{\le}&\frac{\mbox{(II)}}{2}+\frac{c}{\alpha+1}
  \int_{B_{1}\cap \{\snr{\nabla v_{\varepsilon}}>1\}}\snr{\nabla \eta}^{2}\tx{L}(\nabla v_{\varepsilon})^{\alpha+2}\left(1+\snr{F_{\varepsilon}'(\nabla v_{\varepsilon})}^{\frac{q-p}{q-1}}\right)\dx\nonumber \\
&&+\frac{c}{\alpha+1}\int_{B_{1}\cap \{\snr{\nabla v_{\varepsilon}}\le 1\}}\snr{\nabla \eta}^{2}\tx{L}(\nabla v_{\varepsilon})^{\alpha+1+\frac{2}{p}}\dx\nonumber \\
  &\le&\frac{\mbox{(II)}}{2}+\frac{c}{\alpha+1}
  \int_{B_{1}\cap \{\snr{\nabla v_{\varepsilon}}>1\}}\snr{\nabla \eta}^{2}\tx{L}(\nabla v_{\varepsilon})^{\alpha+2}\left(1+\snr{F_{\varepsilon}'(\nabla v_{\varepsilon})}^{\frac{q-p}{q-1}}\right)\dx\nonumber \\
&&+\frac{c}{\alpha+1}\int_{B_{1}\cap \{\snr{\nabla v_{\varepsilon}}\le 1\}}\snr{\nabla \eta}^{2}\tx{l}_{\alpha}(\nabla v_{\varepsilon})^{2}\dx\nonumber \\
  &\le&\frac{\mbox{(II)}}{2}+\frac{c}{\alpha+1}
  \int_{B_{1}}\snr{\nabla \eta}^{2}\tx{l}_{\alpha}(\nabla v_{\varepsilon})^{2}\left(1+\snr{F_{\varepsilon}'(\nabla v_{\varepsilon})}^{\frac{q-p}{q-1}}\right)\dx,
\end{eqnarray*}
with $c\equiv c(n,L,p,q)$. Merging the content of the two previous displays, we derive
\begin{flalign}\label{6.0}
  \int_{B_{1}}\eta^{2}\snr{\nabla \tx{l}_{\alpha}(\nabla v_{\varepsilon})}^{2}\dx\le c\left(\frac{\alpha+2}{\alpha+1}\right)^{2}
  \int_{B_{1}}\snr{\nabla \eta}^{2}\tx{l}_{\alpha}(\nabla v_{\varepsilon})^{2}\left(1+\snr{F_{\varepsilon}'(\nabla v_{\varepsilon})}^{\frac{q-p}{q-1}}\right)\dx,
\end{flalign}
with $c\equiv c(n,L,p,q)$. Let us briefly discuss the case $\alpha=-1$: in \eqref{6.0.1} term $\mbox{(II)}$ vanishes, so we only need to estimate
term $\mbox{(III)}$ in a slightly different way than what we did before. Via Cauchy Schwarz and Young inequalities we get
\begin{eqnarray*}
\snr{\mbox{(III)}}&\le&\frac{\mbox{(I)}}{2}+8\int_{B_{1}}\snr{F_{\varepsilon}''(\nabla v_{\varepsilon})}\snr{\nabla v_{\varepsilon}}^{2}\snr{\nabla \eta}^{2}\dx\nonumber \\
&\stackrel{\eqref{cdq}_{3,4}}{\le}&\frac{\mbox{(I)}}{2}+
8\Lambda\int_{B_{1}\cap\{\snr{\nabla v_{\varepsilon}}>1\}}\frac{\left(1+\snr{F_{\varepsilon}'(\nabla v_{\varepsilon})}^{\frac{q-2}{q-1}}\right)}{\ell_{\mu}(\nabla v_{\varepsilon})^{p-2}}
\ell_{\mu}(\nabla v_{\varepsilon})^{p}\snr{\nabla \eta}^{2}\dx\nonumber \\
&&+c\int_{B_{1}\cap \{\snr{\nabla v_{\varepsilon}}\le 1\}}\ell_{\mu}(\nabla v_{\varepsilon})^{2}\snr{\nabla \eta}^{2}\dx\nonumber \\
&\stackrel{\eqref{f'}}{\le}&\frac{\mbox{(I)}}{2}+
c\int_{B_{1}\cap\{\snr{\nabla v_{\varepsilon}}> 1\}}\left(1+\snr{F_{\varepsilon}'(\nabla v_{\varepsilon})}^{\frac{q-p}{q-1}}\right)\tx{L}(\nabla v_{\varepsilon})\snr{\nabla \eta}^{2}\dx\nonumber \\
&&+c\int_{B_{1}\cap\{\snr{\nabla v_{\varepsilon}}\le 1\}}\left(1+\tx{L}(\nabla v_{\varepsilon})^{\frac{1}{2}}\right)^{2}\snr{\nabla \eta}^{2}\dx\nonumber \\
&\le&\frac{\mbox{(I)}}{2}+c\int_{B_{1}}\left(1+\snr{F_{\varepsilon}'(\nabla v_{\varepsilon})}^{\frac{q-p}{q-1}}\right)\tx{l}_{-1}(\nabla v_{\varepsilon})^{2}\snr{\nabla \eta}^{2}\dx,
\end{eqnarray*}
for $c\equiv c(n,L,p,q)$, thus
\begin{flalign}\label{6.0..}
  \int_{B_{1}}\eta^{2}\snr{\nabla \tx{l}_{-1}(\nabla v_{\varepsilon})}^{2}\le
  c \int_{B_{1}}\snr{\nabla \eta}^{2}\tx{l}_{-1}(\nabla v_{\varepsilon})^{2}\left(1+\snr{F_{\varepsilon}'(\nabla v_{\varepsilon})}^{\frac{q-p}{q-1}}\right)\dx.
\end{flalign}
Finally, with \eqref{6.0}-\eqref{6.0..} at hand and keeping in mind the validity of \eqref{apri}, we can pick any constant $\tx{M}_{\varepsilon}\ge 1$
satisfying \eqref{mm} and turn \eqref{6.0}-\eqref{6.0..} into \eqref{cacc}. The proof is complete.
\end{proof}

\subsection{Proof of Theorem \ref{scat}}
We start by observing that there is no loss of generality in assuming 
\eqn{sc.0}
$$
\nr{F'_{\varepsilon}(\nabla v_{\varepsilon})}_{\LL^{\infty}(B_{1/8})}\ge 1,
$$
otherwise the proof would be already finished. We then fix parameters $1/8\le \tau_{2}<\tau_{1}\le 1/4$, corresponding to concentric balls
$B_{1/8}\subseteq B_{\tau_{2}}\subset B_{\tau_{1}}\subseteq B_{1/4}$, set $\tx{M}_{\varepsilon}:=\nr{F_{\varepsilon}^{\prime}(\nabla v_{\varepsilon})}_{\LL^{\infty}(B_{\tau_{1}})}^{\frac{q-p}{q-1}}$
and notice that \eqref{cacc} holds in particular for all nonnegative $\eta\in \CC^{1}_{c}(B_{\tau_{1}})$ and our choice of $\tx{M}_{\varepsilon}$ matches \eqref{mm}
(keep in mind \eqref{sc.0}). Moreover, by \eqref{apri} we see that $\tx{L}(\nabla v_{\varepsilon})\in \LL^{\infty}_{\loc}(B_{1})$, and
$\tx{l}_{-1}(\nabla v_{\varepsilon})\in \WW^{1,2}_{\loc}(B_{1})$. This, together with \eqref{cacc} allows applying Lemma \ref{am} to
$\tx{L}(\nabla v_{\varepsilon})$, with $\sigma=\tau_{1}$, $\alpha_{0}=-1$, $\gamma_{0}$ being any positive number less than $p/q$ (say $\gamma_{0}=p/(2q)<1$),
$\tx{A}_{\alpha}$ as in \eqref{aa}, and $\tx{M}=\tx{M}_{\varepsilon}$ to get
\begin{eqnarray}\label{sc.3}
\nr{F'_{\varepsilon}(\nabla v_{\varepsilon})}_{\LL^{\infty}(B_{\tau_{2}})}^{\frac{p}{q-1}}&\stackrel{\eqref{f'}}{\le}&c\nr{\tx{L}(\nabla v_{\varepsilon})}_{\LL^{\infty}(B_{\tau_{2}})}+c\nonumber \\
&\stackrel{\eqref{6.8}}{\le}& \frac{c}{(\tau_{1}-\tau_{2})^{\tx{a}_{0}}}
\nr{F_{\varepsilon}'(\nabla v_{\varepsilon})}_{\LL^{\infty}(B_{\tau_{1}})}^{\frac{\gamma(q-p)}{(q-1)(1-\gamma)}}\nra{\tx{l}_{-1}(\nabla v_{\varepsilon})}_{\WW^{1,2}(B_{\tau_{1}})}^{2}+c,
\end{eqnarray}
for $c\equiv c(n,L,p,q)$, $\tx{a}_{0}\equiv \tx{a}_{0}(n,p,q)$. Now, from \eqref{pq} and our choice of $\gamma_{0}$ we deduce that
$$
\frac{\gamma(q-p)}{(q-1)(1-\gamma)}<\frac{p}{q-1},
$$
so in \eqref{sc.3} we can apply Young inequality with conjugate exponents $\left(\frac{p(1-\gamma)}{\gamma(q-p)},\frac{p(1-\gamma)}{p-\gamma q}\right)$ to have
\begin{flalign*}
  \nr{F_{\varepsilon}(\nabla u_{\varepsilon})}_{\LL^{\infty}(B_{\tau_{2}})}^{\frac{p}{q-1}}\le
  \frac{1}{4}\nr{F_{\varepsilon}(\nabla u_{\varepsilon})}_{\LL^{\infty}(B_{\tau_{1}})}^{\frac{p}{q-1}}+
  \frac{c\nra{\tx{l}_{-1}(\nabla v_{\varepsilon})}_{\WW^{1,2}(B_{\tau_{1}})}^{\frac{2p(1-\gamma)}{p-\gamma q}}}{(\tau_{1}-\tau_{2})^{\frac{p\tx{a}_{0}(1-\gamma)}{p-\gamma q}}}+c.
\end{flalign*}
Lemma \ref{l5} then yields
\begin{flalign*}
\nr{F'_{\varepsilon}(\nabla v_{\varepsilon})}_{\LL^{\infty}(B_{1/8})}^{\frac{p}{q-1}}\le c\nra{\tx{l}_{-1}(\nabla v_{\varepsilon})}_{\WW^{1,2}(B_{1/4})}^{\frac{2p(1-\gamma)}{p-\gamma q}}+c,
\end{flalign*}
thus via \eqref{keymono}$_{1}$,
\begin{flalign*}
\nra{\nabla v_{\varepsilon}}_{\LL^{\infty}(B_{1/8})}\le c\nra{\tx{l}_{-1}(\nabla v_{\varepsilon})}_{\WW^{1,2}(B_{1/4})}^{\frac{2(1-\gamma)(q-1)}{(p-1)(p-\gamma q)}}+c
\end{flalign*}
with $c\equiv c(n,L,p,q)$. We then scale back on $B$ to deduce via \eqref{gg} and \eqref{hdes.1},
\begin{flalign*}
\nra{\nabla u_{\varepsilon}}_{\LL^{\infty}(B/8)}\le c\nra{V_{\mu,p}(\nabla u_{\varepsilon})}_{\WW^{1,2}(B/4)}^{\frac{2(1-\gamma)(q-1)}{(p-1)(p-\gamma q)}}+c\le c\nra{(F(\nabla u)+1)}_{\LL^{1}(B)}^{\tx{b}},
\end{flalign*}
for $c\equiv c(n,L,L_{\mu},p,q)$ and $\tx{b}\equiv \tx{b}(n,p,q)$ (possibly different from the one appearing in \eqref{hdes.1}).
Finally, we send $\varepsilon\to 0$ in the above display and use $\eqref{approx}_{3}$ to conclude with \eqref{scai}. A standard covering argument completes the proof.

\subsection*{Data availability statement} There are no data attached to this paper.

\end{document}